\documentclass{amsart}

\usepackage[T1]{fontenc}
\usepackage[utf8]{inputenc}
\usepackage[english]{babel}

\usepackage{amsmath}
\usepackage{amssymb}
\usepackage{amsthm}
\usepackage{amscd}
\usepackage{braket}
\usepackage{bm}
\usepackage{bbm}
\usepackage{yfonts}
\usepackage{guit}
\usepackage[babel]{csquotes}
\usepackage{etoolbox}

\usepackage{comment}
\usepackage{lmodern}

\newlength\longest

\usepackage{tikz}
\usetikzlibrary{matrix}

\newcommand{\N}{\mathbb{N}}
\newcommand{\Z}{\mathbb{Z}}

\newcommand{\R}{\mathbb{R}}
\newcommand{\C}{\mathbb{C}}
\newcommand{\B}{\mathbb{B}}

\newcommand{\eps}{\varepsilon}

\newcommand{\tens}{\otimes}
\newcommand{\del}{\partial}

\newcommand{\Sf}[1]{\mathbb{S}^{#1}}
\newcommand{\Sfp}{\mathbb{S}_+^n}

\newcommand{\sthat}{\text{ s.t. }}
\newcommand{\void}{\emptyset}
\newcommand{\ft}[1]{\widehat{#1}}

\newcommand{\indi}{\mathbbm 1}
\newcommand{\neigh}{\dot{\ni}}
\newcommand{\intern}{\dot{\in}}

\newcommand{\Ccal}{\mathcal C}

\newcommand{\Hcal}{\mathcal H}
\newcommand{\Cinf}{\mathcal C^{\infty}}
\newcommand{\Cinfdot}{\dot{\mathcal{C}}^\infty}
\newcommand{\Scal}{\mathcal S}
\newcommand{\Lcal}{\mathcal L}
\newcommand{\Ecal}{\mathcal E}
\newcommand{\Wcal}{\mathcal W}
\newcommand{\Bcal}{\mathcal{B}}
\newcommand{\Qcal}{\mathcal{Q}}
\newcommand{\Acal}{\mathcal{A}}

\newcommand{\sob}[1]{\mathit{H}_{#1}}
\newcommand{\sym}[1]{\mathit{S}^{#1}}
\newcommand{\leb}[1]{\mathit{L}^{#1}}
\newcommand{\psdo}[1]{\Psi^{#1}}
\newcommand{\sg}[1]{\mathit{SG}^{#1}}
\newcommand{\RG}{\mathcal{R}G}
\newcommand{\BG}[1]{\Bcal G^{#1}}
\newcommand{\HG}[1]{\mathit{HG}_{#1}}
\renewcommand{\lg}[1]{\mathit{LG}^{#1}}
\newcommand{\mi}{\mu}
\newcommand{\abs}[1]{\left|#1\right|}
\newcommand{\hnorm}[2]{\|#2\|_{\mathit{H}_{#1}}}
\newcommand{\lnorm}[2]{\|#2\|_{\mathit{L}^{#1}}}
\newcommand{\norm}[2]{\|#2\|_{#1}}

\newcommand{\push}[1]{{#1}_\ast}

\newcommand{\VF}{\mathfrak{X}}

\newcommand{\debar}{%
	\ooalign{\hidewidth $\de\,$\hidewidth\cr\rule[1.2ex]{1ex}{.6pt}}}

\renewcommand{\Im}{\mathcal{I}\mathit{m}\,}
\renewcommand{\bar}[1]{\overline{#1}}
\renewcommand{\sharp}{\#}
\renewcommand{\phi}{\varphi}

\def\partder#1#2{\frac{\partial #1}{\partial#2}}

\def\inner#1#2{\left(#1,#2\right)}
\def\setquotient#1#2{{}^{#1}\!\!\diagup_{#2}}

\def\SymG#1{\Sigma G^{#1}}
\def\sc#1{{}^{sc}{#1}}

\DeclareMathOperator{\dbar}{\debar\!}
\DeclareMathOperator{\de}{\mathrm{d}\!}
\DeclareMathOperator{\id}{\mathrm{id}}

\DeclareMathOperator{\ad}{\mathrm{ad}}

\DeclareMathOperator{\intprod}{\lrcorner}
\DeclareMathOperator{\spec}{\mathrm{spec}}
\DeclareMathOperator{\gr}{\mathrm{gr}}
\DeclareMathOperator{\Pol}{\mathrm{Pol}}
\DeclareMathOperator{\depth}{\mathrm{depth}}
\DeclareMathOperator{\graph}{\mathrm{graph}}

\DeclareMathOperator{\im}{\mathrm i}

\DeclareMathOperator{\FT}{\mathcal F}
\DeclareMathOperator{\supp}{\mathrm{supp}}

\DeclareMathOperator{\Op}{\mathrm{Op}}

\DeclareMathOperator{\Diff}{\mathrm{Diff}}

\theoremstyle{definition}
\newtheorem{defin}{Definition}[section]
\newtheorem{theo}[defin]{Theorem}
\newtheorem{prop}[defin]{Proposition}
\newtheorem{lemma}[defin]{Lemma}
\newtheorem{cor}[defin]{Corollary}
\newtheorem{rem}[defin]{Remark}
\newtheorem{eg}[defin]{Example}
\newtheorem{conv}{Convention}

\author{Alessandro Pietro Contini}
\title[$SG\Psi$DO, Singular Symplectic Geometry and OPIs]{$SG$-classes, Singular Symplectic Geometry and Order-Preserving Isomorphisms}

\begin{document}
    \maketitle
	\section*{Abstract}
        The geometric theory of pseudo-differential and Fourier Integral Operators relies on the symplectic structure of cotangent bundles. If one is to study calculi with some specific feature adapted to a geometric situation, the corresponding notion of cotangent bundle needs to be adapted as well and leads to spaces with a singular symplectic structure. Analysing these singularities is a necessary step in order to construct the calculus itself. 

        In this thesis we provide some new insights into the symplectic structures arising from asymptotically Euclidean manifolds. In particular, we study the action of the Poisson bracket on $SG$-pseudo-differential operators and define a new class of singular symplectomorphisms, taking into account the geometric picture. We then consider this notion in the context of the characterisation of order-preserving isomorphisms of the $SG$-algebra, and show that these are in fact given by conjugation with a Fourier Integral Operator of $SG$-type.

        \tableofcontents
	\section{Introduction}
        This thesis is concerned with aspects of the global calculus of $SG$-pseudo-differential operators, the corresponding classes of Fourier Integral Operators, and their relation as algebras and modules.

        Pseudo-differential operators ($\Psi$DOs in what follows) are one of the most important tools for the study of (elliptic) partial differential equations (PDEs) and have proven to be objects of interest for a number of different areas of modern mathematics. The basic idea is to construct a large class of operators where differential operators admit inverses, at least in an approximate sense, and with good formal properties which allow one to more or less freely take compositions, adjoints and so on, while at the same time being able to control the errors. This is achieved by a generalisation and formalisation of the techniques of asymptotic analysis, whose origin dates back at least to the 19th century, with the pioneering works of Laplace, Stokes and Kelvin on the method of stationary phase. In the 20th century, the study of singular integral operators, initiated by Hilbert and brought to completion by Mikhlin \cite{mikhlin1948singularintegral}, Calderón \cite{calderon1956singularintegrals},\cite{calderon1957singularintegraloperators}, was paired with the language of distributions of Schwartz fame and with many ideas from the world of quantum mechanics. This led Kohn and Nirenberg \cite{kohn1965pseudodifferentialoperators} and Hörmander \cite{hormander1965pseudodifferentialoperators} to develop a general calculus of $\Psi$DOs and study elliptic PDEs of a very general type, obtaining existence and uniqueness results for a swath of then-unsolved problems. In particular, on a compact manifold one can take advantage of the compactness of the Sobolev embeddings to prove regularity results for the solutions. Furthermore, thanks to the properties of the calculus, the parametrix construction of Hadamard, originally invented for differential equations, extends to pseudo-differential operators and shows that elliptic $\Psi$DOs on compact manifolds are Fredholm thanks to the fact that the ``residual'' operator of the construction is compact. Far-reaching subsequent generalisations led to a global theory of elliptic boundary-value problems on compact manifolds, including the global definition of the principal symbol of an operator as a function on the cotangent bundle, and finally to the celebrated index theorem of Atiyah and Singer \cite{atiyah1968index},\cite{atiyah1968index3}. This highlighted the incredible amount of topological and geometrical information these operators carried and became (in many senses still is to this day) one of the main motives of research in geometric and global analysis. Shortly thereafter, the study of limit and boundary-value problems for pseudo-differential equations led Louis Boutet de Monvel \cite{boutetdemonvel1971boundaryproblems} to construct a calculus of manifolds with boundary and to a topological index formula\footnote{See also \cite{fedosov1974analyticformulaellipticboundary1} for an analytical counterpart and the book \cite{rempel1985indextheoryboundary} for a overarching discussion.}.

        Tailored to the study of elliptic equations, the theory of $\Psi$DO required considerable effort to be adapted to other classes of PDEs. In relation to the blooming index theory, the study of the heat equation associated with a second order elliptic $\Psi$DO produced many insights into the analytical nature of the Atiyah-Singer formula and led to the local index theorem of Atiyah \cite{atiyah1973heatequationindex}. In the following years, a full-fledged theory of Dirac operators on spin manifolds, shedding light on their importance to mathematics and physics alike, was investigated and is to this day a very active area of research (we refer here to the books of Berline, Getzler and Vergne \cite{berline1992heatkernelsdirac} and Gilkey \cite{gilkey1995invariance} for a deep and interesting discussion). 

        On the other hand, even for simple hyperbolic equations it was clear that $\Psi$DOs could not provide a satisfactory answer on their own power and that a more general theory had to be developed. Building on ideas from geometrical optics and earlier work of Lax and Maslov, Hörmander \cite{hormander1971fourierintegraloperators} developed the calculus of Fourier Integral Operators (FIOs) and applied it\footnote{The name is historically controversial. While the theory of Hörmander is without a doubt more general, a great part of the key ideas can be found in \cite{maslov1972theoriedesperturbations}, which is a late translation from Russian of a 1965 opus of the same author. At the same time, it seems hard to criticise Dieudonné \cite[4]{dieudonne1978elementsdanalyse7} when he appends them the name ``opérateurs de Lax-Maslov'', glossing ``appelés malencontreusement aussi <<opérateurs intégraux de Fourier>>, ce qui est d'autant plus ridicule que la transformation de Fourier n'y joue aucun rôle''. While we acknowledge all these contributions and recognise the elements of truth, we stick here to the name of Fourier Integral Operators out of mere laziness.}, together with Duistermaat, to the study of hyperbolic systems \cite{duistermaat1972fourierintegraloperatorsapplications}. The theory of FIOs proved to be, in the following years, a fundamental tool to approach a large number of yet to be tackled problems, including but not limited to existence and uniqueness for hyperbolic equations, and invigorated the calculus of $\Psi$DOs by providing new methods to study elliptic equations. This is rendered possible by the celebrated theorem of Egorov \cite{egorov1969canonicaltransformation}, stating that conjugating a $\Psi$DO with principal symbol $p$ with an invertible FIO produces again a $\Psi$DO with a principal symbol given by pull-back of $p$ along an underlying canonical transformation of the cotangent bundle. Thanks to this fact, the theory of $\Psi$DOs and FIOs can be seen, in the context of the Heisenberg picture of quantum mechanics, as a quantisation scheme where observables are mapped to self-adjoint $\Psi$DOs and the evolution operator of the system, classically a canonical transformation, acts as an FIO on the space of observables. At the same time, this idea of ``quantised canonical transformation'' expressed by FIOs can be further characterised: by a theorem of Duistermaat and Singer \cite{duistermaat1976orderpreservingisomorphisms}, the only order-preserving isomorphisms (OPIs) of the algebra of (integer order, classical, properly supported) $\Psi$DOs are exactly given by conjugation with an invertible FIO. This reflects the classical property that if a diffeomorphism transforms Hamilton equations in Hamilton equations (namely, preserves the canonical 1-form on the phase bundle), then it is a canonical transformation.\footnote{We remark that here and later we use the terms \textit{quantisation} or \textit{quantisation scheme} without properly defining what we mean by it. In whole honesty, this correspondence does not give a full quantisation of a classical system according to the Dirac axioms, since we are not really addressing questions such as the classical limit, for example. In this sense the classical theory of FIOs is scale-invariant, since we could in principal work at the level of co-sphere bundles and contact forms, while a ``true'' quantisation should allow one to look at large scale approximation and makes more sense in the context of semi-classical analysis. Nevertheless the similarities are enough to justify our abuse.} 

        The question that at this point one might ask is: what can we say about calculi adapted to non-compact\footnote{The original calculus of Hörmander and the result of Duistermaat and Singer actually hold true for non-compact manifolds as well, under the assumption, necessary to define composition, that the involved operators are properly supported. This property, however, destroys the global effects of the operators and prevents one from studying the geometric structure ``at infinity'' of many specific examples.} manifolds? While the construction of the calculi with the same formal properties as above does not break down in this setting, one is confronted with the annoying fact that the residual operators of the parametrix construction, even though regularising (namely they smooth out all singularities of distribution on arbitrarily large compact sets), are not compact. This, together with the fact that the Sobolev embeddings are not compact, constitutes a fundamental obstacle to the process of constructing solutions with a certain regularity. The problem lies in the calculus itself: the standard class of $\Psi$DOs is only well-suited to control asymptotic behaviour in the ``cotangent direction'', namely if $(x,\xi)$ are coordinates on $\R^{2n}$, we define the class $S^m(\R^{2n})$ by imposing bounds for $x$ varying in a compact set $K$  and $\abs{\xi}\rightarrow\infty$. In particular, the behaviour of symbols in the $x$ variable is hardly restricted and it suffices that they are smooth. But then we cannot hope to get from this class any kind of reasonably sufficient information as $\abs{x}\rightarrow\infty$.

        In order to obviate to these issues, global calculi were introduced. The main feature\footnote{Notice also the earlier attempt of Grushin \cite{grushin1971boundedpseudodifferentialcalculus}, where only uniform bounds on $x$ are required.} of a global calculus (and main difference in comparison to ordinary $\Psi$DOs) on $\R^{n}$ is that we posit a bound on the symbols involving the spatial directions $x$, too. The two main (and more successful) examples in this setting are known as $\Gamma$-classes and $SG$-classes. Introduced by Shubin, the $\Gamma$-classes are also known as \textit{completely isotropic} symbols and contain those smooth functions $a(z)$ on $\R^{2n}$ such that $\abs{\del^\alpha a(z)}\lesssim\braket{z}^{m-\abs\alpha}$ for a fixed $m\in\R$ known as the order of $a$. The residual operators in this calculus are exactly integral operators with kernel in $\Scal(\R^{2n})$, which are known to be compact on $\leb{2}(\R^n)$. They have so far found wide ranging application to a number of different problems in index theory (\cite{elliott1996atiyahsingerclassicallimit,elliott1993heisenberggroup}), quantisation (\cite{fedosov1996deformationquantisationindex}), PDEs and spectral theory (\cite{shubin2001pseudodifferential}). Helffer \cite{helffer2013theoriespectraleglobalementelliptiques} has in addition introduced global FIOs modelled on the $\Gamma$-classes (at the level of the phase and the amplitude), and studied their spectral properties. However, to this day it seems that the classes haven't been defined on non-compact manifolds more general than $\R^n$. Furthermore, Helffer does not study the associated class of symplectomorphisms on $\R^{2n}$ that putatively should be quantised by his class of FIOs, but derives the properties of the calculus from purely analytical facts. 

        The other main approach (at least for our concerns) is that of $SG$-classes\footnote{Here $G$ stands for ``global''.}. Originally introduced by Parenti \cite{parenti1972operatoripseudodifferenzialiRn} to study PDEs on unbounded domains, the theory benefited from the contributions of many authors and in particular of Schrohe \cite{schrohe1987weightedsobolevspacesmanifolds,schrohe1988psistaralgebra}. Their usefulness was soon recognised by Cordes, who significantly enlarged the original calculus and presented a very general theory of $\Psi$DOs in \cite{cordes1995techniquepseudodifferentialoperators}. The core of the matter is as follows: instead of looking at completely isotropic symbols, introduce a new filtration to obtain a class $SG^{m_1,m_2}$, where $a$ is a symbol of bi-order $(m_1,m_2)$ if $a$ is bounded by (a positive constant times) $\braket{x}^{m_1}\braket{\xi}^{m_2}$ and each $x$-derivative, respectively $\xi$-derivative, improves the bound by 1 in $x$, respectively $\xi$. Then, the intersection of all these classes clearly consists of Schwartz functions and the residual operators are again exactly the integral operators with kernel in $\Scal(\R^{2n})$, but the more flexible structure of the bounds accounts for a more general class of operators to be studied (albeit of course with the extra complexity that has been introduced). In \cite{schrohe1987weightedsobolevspacesmanifolds}, the calculus has been generalised to a large class of non-compact manifolds, so-called $SG$-manifolds, although one might argue that the class might even be too large for some purposes (more on this later). The extra flexibility of $SG$-classes plays here a crucial role: apart from a technical condition on the charts (in practice, always satisfied), all it suffices to ask is that the changes of coordinates on the manifold have components which are in $\sg{1,0}$. This is of course to be expected if one wants to define $SG$-classes in an invariant way on a manifold, since in order to preserve the filtrations we have to require as a minimum that the transformed base, respectively cotangent, variables be of order $(1,0)$, respectively $(0,1)$. Therefore, the restrictions are in truth quite lax. $SG$-$\Psi$DOs have been applied to a variety of problems, including but not limited to spectral asymptotics on asymptotically Euclidean manifolds (cfr. \cite{coriasco2013spectralasymptoticsends, coriasco2020weyllawasymptoticallyeuclidean}), and mathematical physics (see for example \cite{battisti2011einsteinhilbert}). FIO calculi modelled on $SG$-classes have been introduced by Coriasco \cite{coriasco1999fourierintegraloperators} first and enriched by Andrews \cite{andrews2009SGFIOcomposition} in a second moment. Although they have been described only on $\R^n$ as globally defined operators, explicit and implicit hints to possible geometric generalisations were present in both pieces. However, to this day no such theory has been fully understood and in particular the study of canonical transformations associated with the existing classes hasn't been carried out. 

        In the context of analysis on non-compact spaces, another point of view has been introduced and studied by many authors falling, to various degrees, under the umbrella of the so-called ``Melrose school''. In this picture, one limits the study to classes of manifolds having a somewhat ``regular'' structure at infinity, namely one assumes that, outside of a compact centre, the non-compact manifold $X$ admits a Riemannian metric with a specific asymptotic behaviour as ``$\abs{x}\rightarrow\infty$'' (cf. \cite{melrose1995geometricscattering}). Then, one can construct (explicitly!) a compactified manifold with boundary $\bar{X}$ whose interior is diffeomorphic to the original manifold $X$, and obtain a metric on $\bar{X}$, smooth in the interior and with a prescribed singularity at the boundary. The main example is that of manifolds with \textit{ends}. Topologically these are just given by a compact manifold $X_0$ with boundary $\del X_0=B_1\cup\dots\cup B_d$, where each $B_j$ is a closed codimension 1 submanifold, and cylinders $C_1,\dots, C_d$, $C_j=\R^+\times B_j$, each glued to the corresponding connected component of $\del X_0$. These builds up the ``ends'' or ``exits'' of $X$. This setting, while included in $SG$-manifolds, allows for a more refined analysis of the metric structure on each end. For example, we might consider a metric which is asymptotically cylindrical, namely that on the end takes the form 
        \[
        g=\de t^2+h
        \]
        for $t$ the coordinate on $\R^+$ and $h$ a metric on $B_j$. Introducing the change of coordinates $x=e^{-t}$ maps the infinite cylinder to a finite one, and the above tensor is transformed to 
        \[
        \bar{g}=\frac{\de x^2}{x^2}+\bar{h},
        \]
        where $\bar{h}$ is simply the metric $h$ in the changed coordinates. Having so ``rescaled'' the metric structure, it becomes evident that we can consider our original manifold as the interior of a manifold with boundary by attaching a \textit{closed} cylinder to the boundary component $B_j$, provided that at the same time we keep in mind that we have performed a change of coordinates. The so-called $b$-geometry, namely the generalisation of this example to manifolds with corners, is built upon considering the properties of the Lie algebra of vector fields on $\bar{X}$ which are tangent to the boundary. Starting from this Lie algebra one constructs a calculus of differential and pseudo-differential operators, the so-called $b$-calculus\footnote{In this case the naming convention is not nearly as controversial as it is for FIOs, however these operators are sometimes known as \textit{totally characteristic}. This name, appearing for example in \cite{hormander1994analysispseudodifferential}, Section XVIII.3, directs the limelight more towards the fact that the operators indeed exhibit a singular behaviour at the boundary, whereas historically it turned out to be superseded by the shorter and suggestive ``$b$'', standing for \textit{boundary}.}. On the one hand, this is an extremely powerful tool and idea, on the other there are however a number of non-negligible technical difficulties. In particular, elliptic operators in the classical sense are not Fredholm and a sort of ``non-commutative boundary symbol'' (called \textit{indicial operator}) needs to be taken into account. Both the power and the difficulties of the $b$-calculus are beautifully expounded in \cite{melrose1993atiyahpatodisinger}, together with thorough discussion of the related aspects of index theory (specifically, the Atiyah-Patodi-Singer index theorem) on manifolds with boundary.

        With a similar approach one can work in the setting of asymptotically Euclidean manifolds by attaching cones to each $B_j$ instead of cylinders. While of course this produces the same underlying topological space as before, we are here imposing that the metric is conic ``at $\infty$''. After changing coordinates as before and compactifying, we are then working with a metric which near the boundary is of the form
        \[
        \bar g=\frac{\de x^2}{x^4}+\frac{\bar h}{x^2}.
        \]
        While this looks more singular at first glance (indeed recall that $x=0$ at the boundary), it turns out the associated Lie algebra of vector fields (the so-called \textit{scattering} vector fields) is much easier to study since it is actually commutative ``at the boundary''. Correspondingly, the differential and pseudo-differential operators posses a second \textit{commutative} symbol $\sigma_N$, defined at the boundary and related to the asymptotic behaviour as $\left|x\right|\rightarrow\infty$ of the full symbol in the $SG$-calculus. In fact, \textit{classical} $SG$-operators and \textit{classical} $sc$-pseudo-differential operators\footnote{Here $sc$ stands of course for ``scattering'', another instance of the above naming convention.} are two sides of the same coin: already in the '90s it was well known that the symbol classes for the two calculi are isomorphic\footnote{See for example \cite{egorov1997pseudodifferentialsingularities}, Section 8.2.2 or the more recent \cite{coriasco2017lagrangiansubmanifolds}.}. The choice of which approach to exploit over the other is a mainly just a matter of personal preference but they have advantages and disadvantages, in that the first is more explicitly computable whereas the latter is more manifestly global in nature, being defined directly on a class of manifolds. Part of the goal of this thesis is to explore this correspondence further, especially in its relation with the symplectic geometry of the cotangent bundle.

        Whereas $\Psi$DOs are a very well-explored topic in both these examples (and many others), the state--of--the--art of FIOs in singular situations, including their ellipticity properties, index formul\ae\;and the study of their geometrical theory has lagged behind. In fact, attempts here are somewhat scarce. For the case of the $b$-calculus, the early paper of Melrose \cite{melrose1981transformationboundaryproblems} already contains a study of the class of Lagrangian distributions of interest, based on the geometrical properties that one should expect from a Lagrangian relation on a manifold with boundary. However, a study of the respective ellipticity, Fredholmness and index properties has not been carried out (as far as the author knows, the only Atiyah-Singer--type formula for the index of an FIO has been derived in the setting of closed manifolds by Epstein and Melrose \cite{epstein1998contactdegreeindex} and Leichnam, Nest and Tsygan \cite{leichtnam2001localindexformula}). To this end, one would need to analyse closely the behaviour of the FIO at the boundary and construct a ``full calculus'' for these objects\footnote{This is Melrose-speak for a calculus in which the residual operators of the parametrix construction for elliptic operators are compact.}. More in the spirit of Boutet de Monvel, a calculus of FIOs on manifolds with boundary has been constructed by Battisti, Coriasco and Schrohe \cite{battisti2015FIOmanifoldswithboundary}. There, the geometric and analytical conditions at the boundary were studied in great detail and the authors were able to prove the Fredholm property for the elliptic elements in their calculus, thereby setting up the frame for an index problem in the spirit of Weinstein \cite{weinstein1997indexquantisedcanonicaltransformation}. In particular they showed that the notion of boundary canonical transformation of Section III in \cite{melrose1981transformationboundaryproblems} produces appropriate operators in this calculus. However, the peculiarities of the Boutet de Monvel calculus complicate the analytical picture and an index formula was not established.

        A third point of view is that of associating $\Psi$DO and FIO calculi with a \textit{groupoid}. The philosophy behind this is akin to the singular analytical approach: one considers operators with a specific degeneracy/peculiarity, tries to encode their properties in a geometric object (a groupoid in this case, in contrast to the metric on the compactified space in the Melrose approach), and takes advantage of an overarching calculus structure defined in general on/for the geometric object. This has been successfully brought to completion in a multitude of contributions. For $\Psi$DOs, Nistor, Weinstein and Xu \cite{nistor1999pseudodifferentialgroupoids} and Monthubert \cite{monthubert2003pseudodifferentialgroupoids} introduced the first calculi\footnote{Also notice the recent approaches for nilpotent Lie groups and filtered manifolds of van Erp and Yuncken \cite{vanerp2019pseudodifferentialgroupoids} and Ewert \cite{ewert2023pseudodifferentialgeneralisedfixedpoint}.}, while FIOs have required considerable more effort and only appeared in such a general setting very recently in \cite{lescure2017fiogroupoids}. Despite covering a lot of previously examined settings, it appears that an analysis of the conditions under which the calculus of FIOs contains Fredholm operators on an appropriately defined scale of Sobolev spaces is yet to be examined. Indeed, a direct specialisation of the techniques in the above papers only recovers the ``small calculus'', namely\footnote{This is lingo for a family of operators in which a parametrix construction for elliptic operators makes sense, but does not necessarily produce a compact remainder on Sobolev spaces.} constructs operator classes adapted to the geometric situation but without regard for the Fredholmness ``at the boundary''. Since these aspects are paramount to us, we shall not touch on this subject any further.

        The author's interest in the global calculi stems from an idea of Schrohe that a result like Theorem 1 in \cite{duistermaat1976orderpreservingisomorphisms}, namely the characterisation of order-preserving isomorphisms of $\Psi$DOs, might hold true for the classes $LG^{m_1,m_2}$ of $SG$-$\Psi$DOs which are classical and of order $(m_1,m_2)\in\Z\times\Z$. This is the main problem we set out to tackle in the thesis. While this looks like a fairly reasonable expectation (indeed, for example, the class $LG^{0,m_2}$ is a subclass of $\psdo{m_2}(\R^n)$ and the properly supported property, required for composition in the usual calculus, is substituted in the $SG$ picture by the estimates as $\abs{x}\rightarrow\infty$), it quickly turned out that a proof along the lines of the original paper and completely in terms of the ``local picture'' of the $SG$-calculus was cumbersome to say the least. On the other hand, the scattering approach, while being conceptually advantageous, is less explicit and requires to pick specific local coordinates for computations. Together with the fact that the existing parametrization results for Lagrangian submanifolds in the $SG$ setting already employed a ``mixed'' approach, we resolved to try and take as much advantage as possible of this double point of view.

        We describe briefly the organization of the manuscript. Section \ref{chap:sg calculus} contains the basics of the (classical) $SG$-calculus on $\R^n$. We follow the exposition in \cite{egorov1997pseudodifferentialsingularities} rather closely, especially in regard to classicality, however we prefer amplitudes over double symbols when it comes to composition. Most proofs are here omitted for the sake of brevity, and can be found in the cited literature. We proceed to analyse the relation of the symplectic structure with $SG$-symbols, in particular delineating the action of the Poisson brackets on principal symbols. We give an overview of the class of $SG$-FIOs of type $\Qcal$ introduced by Andrews \cite{andrews2009SGFIOcomposition}, which generalises the operators of Coriasco \cite{coriasco1999fourierintegraloperators} and appears naturally at the end of Chapter \ref{chap:scattering geometry}. Most of the material in this chapter is taken almost directly from the cited sources. Notable exceptions are Section 1.2, containing the analysis of the relation between the Poisson bracket and the principal symbol maps, and the $SG$-Egorov Theorem at the end of Section 1.3, which slightly generalises Proposition 14 in \cite{coriasco1999fourierintegraloperators}.

        Section \ref{chap:scattering geometry} is an introduction to the geometric structure underlying the scattering calculus of Melrose, as presented in \cite{melrose1995geometricscattering} and \cite{melrose1994spectralscatteringtheory}. We start with an overview of manifolds with corners and the corresponding spaces of distributions and vector fields. We proceed with a discussion of the scattering cotangent bundle and the symbol spaces, together with the associated operator classes and the symbol maps. We specialise thereafter to the example of the radial compactification of $\R^n$, on which the equivalence between the classical $SG$- and $sc$-calculi is mostly evident, and which will be our main focus in Chapter \ref{chap:OPI}. Here again we refer the reader to the cited literature for the majority of well-known proofs. Novel work starts to appear here: We introduce a definition of ``scattering canonical transformation'' (SCT), analyse its geometric properties, and show that, locally in a suitable sense, its graph admits a parametrisation via an $SG$-phase function, parallel to previous work on $sc$-Lagrangian distributions.

        Section \ref{chap:OPI} contains the main results we obtained. We employ the machinery exposed in the previous Chapters, together with the ideas of Mathai and Melrose \cite{mathai2017geometrypseudodifferentialalgebra}, to give a proof of the $SG$-analogue of Lemma 2 in \cite{duistermaat1976orderpreservingisomorphisms}. In particular, we prove that the notion of scattering canonical transformation introduced in Chapter \ref{chap:scattering geometry} appears naturally. The approximation scheme of the original paper is then adapted to show that the OPI is ascertained at the level of the formal symbol algebra by an elliptic $SG$-FIO of type $\Qcal$, associated with the scattering canonical transformation above. We exploit Lemma 3 of \cite{duistermaat1976orderpreservingisomorphisms} to find an Eidelheit-type isomorphism in our setting and compare it to the $SG$-FIO appearing at the formal level. We prove that this composition is given by an $SG$-$\Psi$DO and show that its mapping properties determine it to be the identity up to an operator with kernel in the Schwartz class. This allows us to conclude that the Eidelheit isomorphism is itself, up to a smoothing operator, an operator of type $\Qcal$, thereby bringing our task to a close.
        
        The author would like to thank Sandro Coriasco and Philipp Schmitt, for many interesting discussions and comments that have led to a better understanding and exposition and Elmar Schrohe, under whose supervision the thesis was completed.
	
	\section{The $SG$-calculus}
\label{chap:sg calculus}

We present here a collection of concepts and facts concerning the $SG$-calculus, beginning with a discussion of symbol spaces. A thorough analysis of classical symbols is included, before moving to the associated operators and the relation between these classes. 

\subsection{$SG$-symbols and operators}

For later reference, we start by defining H\"ormander classes.
\begin{defin}
	\label{def:hormander symbols}
	The class of \textit{H\"ormander symbols} of order $m\in\R$ is the set $\sym{m}(\R^n\times\R^N)$ containing all functions $p\in\Cinf(\R^{n}\times\R^{N})$ such that for each $\alpha\in\N^n,\beta\in\N^N$ and each compact $K\subset\R^n$ one has
	\begin{equation}
		\label{eq:hormander symbols}
		\abs{\del^\alpha_x\del^\beta_\xi p(x,\xi)}\lesssim_{\alpha,\beta,K}\braket{\xi}^{m-\abs{\beta}}.\quad x\in K,\xi\in\R^N.
	\end{equation}
	A symbol $p\in\sym{m}(\R^n\times\R^N)$ with $m\in\Z$ is said to be \textit{classical} if for all $j\geq 0$ there exists a smooth function $p_{m-j}(x,\xi)$, $\xi$-homogeneous of degree $m-j$ outside of a compact neighbourhood of $0\in\R^N$, such that for all $M\in\N$ we have the asymptotic expansion
	\begin{equation}
		\label{eq:classical symbols}
		p(x,\xi)-\sum_{j=0}^{M}p_{m-j}(x,\xi)\in\sym{m-M-1}(\R^n\times\R^N).
	\end{equation}
	In case the original symbol does not depend on $x$, we write $\sym{m}(\R^N)$ and speak about \textit{global classical symbols}.
\end{defin}
\begin{defin}
	\label{def:sgsymbols}
	The class $\sg{m}(\R^n\times\R^N)$ of \textit{SG-symbols} of order $m=(m_e,m_\psi)\in\R^2$ is the set of all $\Cinf$ functions $a\colon\R^n\times\R^N\rightarrow\R$ such that for all $\alpha\in\N^n,\beta\in\N^N$ there exists $c=c(\alpha,\beta)>0$ with
	\begin{equation}
		\label{eq:sgestimates}
		\abs{\del^\alpha_x\del^\beta_\xi a(x,\xi)}\lesssim_{\alpha,\beta} \braket{x}^{m_e-\abs{\alpha}}\braket{\xi}^{m_\psi-\abs{\beta}},\quad x\in\R^n,\xi\in\R^N.
	\end{equation}
	These are all Fréchet spaces with respect to the semi-norms $\norm{(\alpha,\beta)}{\cdot}$ given by the best possible $c(\alpha,\beta)$ in \eqref{eq:sgestimates}.
	We often write $\sg{m}=\sg{m}(\R^n\times\R^n)$ since we will work mainly on $\R^n\times\R^n\cong T^\ast\R^n$. We call $m_e$ the \textit{exit order} and $m_\psi$ the \textit{pseudo-differential order}, see Remark \ref{rem:exit and pseudo boxdot}.
\end{defin}

We collect basic properties of these classes in the Lemma \ref{lemma:sg filtration properties}.

\begin{lemma}
	\label{lemma:sg filtration properties}
	The following holds true.
	\begin{enumerate}
		\item There is a double filtration on the union $\sg{}(\R^n\times\R^N)$ of the classes $\sg{m}$, that is, if $p=(p_e,p_\psi)\leq m=(m_e,m_\psi)$, then $\sg{p}(\R^n\times\R^N)\subset\sg{m}(\R^n\times\R^N)$. 
		\item The projective limits $\sg{m_e,-\infty}$ and $\sg{-\infty,m_\psi}$ are isomorphic to $\Scal(\R^N,S^{m_e}(\R^n))$ and $\Scal(\R^n,S^{m_\psi}(\R^N))$, respectively, while the (double) projective limit $\sg{-\infty\indi}(\R^n\times\R^N)$ is a Fréchet space equalling the class of Schwartz functions $\Scal(\R^n\times\R^N)$ (we call elements of these projective limits \textit{$\psi$-smoothing, $e$-smoothing} and {smoothing}, respectively). 
		\item Pointwise multiplication on $\Cinf(\R^n\times\R^N)$ restricts to $\sg{}(\R^n\times\R^N)$ to make it (together with addition) into a commutative bi-filtered algebra.
		\item For $m\in\R^2$ the functions $\lambda^m(x,\xi)=\braket{x}^{m_e}\braket{\xi}^{m_\psi}\in\sg{m}(\R^n\times\R^N)$ are nowhere zero. Multiplication by $\lambda^m(x,\xi)$ induces isomorphisms of Fréchet algebras $\sg{p}(\R^n\times\R^N)\rightarrow\sg{m+p}(\R^n\times\R^N)$ for all $p\in\R^2$.
	\end{enumerate}
\end{lemma}

The symbols $\lambda^m$ will be used to give a characterization of the following scale of Sobolev spaces adapted to the $SG$-calculus. 

\begin{defin}
	\label{def:sg sobolev spaces}
	For $m=(m_e,m_\psi)\in\R^2$ we define the $\leb{2}$-based \textit{SG-Sobolev spaces} as
	\begin{equation}
		\label{eq:sg sobolev spaces}
		\HG{m}\equiv\braket{x}^{-m_e}\sob{m_\psi}(\R^n).
	\end{equation}
\end{defin}

Much like for H\"ormander classes, a notion of asymptotic expansion is defined and the principle of asymptotic completeness holds true. The existence of the second filtration implies that we can define multiple notions of asymptotic sums, so we summarize them in the following theorem.

\begin{theo}
	\label{theo:asymptotic completeness}
	The following holds true.
	\begin{enumerate}
		\item Let $a_j(x,\xi)\in\sg{m^{(j)}}(\R^n\times\R^N)$ be a sequence of functions with $m^{(j)}=(m_e^{(j)},m_\psi^{(j)})\rightarrow-\infty\indi$ as $j\rightarrow\infty$. There exists $a\in\sg{m}(\R^n\times\R^N), m=(\max m_e^{(j)},\max m_\psi^{(j)}),$ such that, given any $c\in\R$, we can find $K=K(c)\in\N$ with
		\begin{equation}
			\label{eq:full asymptotic completeness}
			a(x,\xi)-\sum_{j=0}^K a_j(x,\xi)\in\sg{m-c\indi}(\R^n\times\R^N),
		\end{equation}
		and $a$ is furthermore unique mod $\Scal(\R^n\times\R^N)$;
		\item Let $a_j(x,\xi)\in\sg{m_e^{(j)},m_\psi}(\R^n\times\R^N)$ be a sequence of functions with $m_e^{(j)}\rightarrow-\infty$ as $j\rightarrow\infty$. There exists $a\in\sg{m}(\R^n\times\R^N), m=(\max m_e^{(j)},m_\psi),$ such that, given any $c\in\R$, we can find $K=K(c)\in\N$ with
		\begin{equation}
			\label{eq:e-asymptotic completeness}
			a(x,\xi)-\sum_{j=0}^K a_j(x,\xi)\in\sg{m-c\indi_e}(\R^n\times\R^N),
		\end{equation}
		and $a$ is furthermore unique mod $\sg{-\infty,m_\psi}(\R^n\times\R^N);$
		\item Let $a_j(x,\xi)\in\sg{m_e,m_\psi^{(j)}}(\R^n\times\R^N)$ be a sequence of functions with $m_\psi^{(j)}\rightarrow-\infty$ as $j\rightarrow\infty$. There exists $a\in\sg{m}(\R^n\times\R^N), m=(m_e,\max m_\psi^{(j)}),$ such that, given any $c\in\R$, we can find $K=K(c)\in\N$ with
		\begin{equation}
			\label{eq:psi-asymptotic completeness}
			a(x,\xi)-\sum_{j=0}^K a_j(x,\xi)\in\sg{m-c\indi_\psi}(\R^n\times\R^N),
		\end{equation}
		and $a$ is furthermore unique mod $\sg{m_e,-\infty}(\R^n\times\R^N).$
	\end{enumerate}
	In all of the above cases we write $a\sim\sum a_j$ to indicate that $a$ is the asymptotic sum of the sequence $a_j$. The scale which we refer to will be in general clear from the context.
\end{theo}

Recall that, in the classical theory of $\Psi$DOs, homogeneous functions can be turned into symbols with the help of an excision function, namely if $b\in\Cinf(\R^n\times\R^N_0)$ is homogeneous of degree $k$ in $\xi$ and\footnote{Namely, $b(x,\mi \xi)=\mi^kb(x,\xi)$ for all $x,\xi\in\R^n,\xi\neq 0,\mi>0$.} $\chi(\xi)=0$ near 0 and $\chi(\xi)=1$ for large $\abs\xi$, then $a(x,\xi)=\chi(\xi)b(x,\xi)\in S^k(\R^n\times\R^N)$ and $a(x,\xi)-b(x,\xi)=(1-\chi(\xi))b(x,\xi)$ is compactly supported in $\xi$. Similarly, asymptotic sums can be made convergent with the help of such a $\chi$ and a sequence $\R^+\ni c_j\rightarrow\infty$ sufficiently fast as $j\rightarrow\infty$, by setting
\begin{equation}
	\label{eq:convergent asymptotic sum}
	a(x,\xi)\equiv\sum_{j\geq 0}\chi\left(\frac{\xi}{c_j}\right)a_j(x,\xi).
\end{equation}
The same process works for $SG$-classes with respect to both sets of variables separately, so that each of the asymptotic sums in Theorem \ref{theo:asymptotic completeness} can be made convergent up to some smoothing term. 

Our main object of interest is the subclass of \textit{classical} (also known as poly-homogeneous) symbols. Since the situation is slightly more complex than in the case of the Hörmander classes, we exercise some extra care here in order to define the notion. In particular, the following relaxed notions of homogeneity are required.

\begin{defin}
	\label{def:polyhomogneousfunctions}
	For $\bullet\in\{e,\psi\}$, define the classes of \textit{partially $m_\bullet$-homogeneous functions}, $m_\bullet\in\R$, by
	\begin{equation}
		\label{eq:partiallyhomogeneousfunctions}
		\begin{aligned}
			\Hcal^{(m_\psi)}_{\psi}&=\left\{a(x,\xi)\in\Cinf(\R^n\times\R^N_0)\sthat \forall \lambda>0,x\in\R^n,\xi\in\R^N_0\; a(x,\lambda \xi)=\lambda^{m_\psi}a(x,\xi)\right\}\\
			\Hcal^{(m_e)}_e&=\left\{a(x,\xi)\in\Cinf(\R^n_0\times\R^N)\sthat \forall \lambda>0,x\in\R^n_0,\xi\in\R^N \;a(\lambda x,\xi)=\lambda^{m_e}a(x,\xi)\right\}.
		\end{aligned}
	\end{equation}	
	Also define the class of \textit{bi-homogeneous functions}, letting for each $m=(m_e,m_\psi)\in\R^2$
	\begin{multline}
		\label{eq:bi-homogeneous functions}
		\Hcal^{(m)}=\left\{a(x,\xi)\in\Cinf(\R^n_0\times\R^N_0)\sthat\right. \\
		\left.\forall \lambda,\mi>0,(x,\xi)\in \R^n_0\times\R^N_0\; a(\lambda x,\mi \xi)=\lambda^{m_e}\mi^{m_\psi}a(x,\xi)\right\}.
	\end{multline}
	The conditions defining these classes can be relaxed to define \textit{eventually homogeneous} functions, that is, homogeneous outside $\B_c(0)$ for some $c>0$. Namely
	\begin{equation}
		\label{eq:eventually homogeneous functions}
		\begin{aligned}
			\Hcal^{[m_\psi]}_\psi&=\left\{a(x,\xi)\in\Cinf(\R^{n+N})\sthat\forall \mi\geq 1,x\in\R^n,\abs{\xi}>c,\; a(x,\mi\xi)=\mi^{m_\psi}a(x,\xi)\right\}\\
			\Hcal^{[m_e]}_e&=\left\{a(x,\xi)\in\Cinf(\R^{n+N})\sthat\forall \lambda\geq 1,\abs{x}>c,\xi\in\R^N,\; a(\lambda x,\xi)=\lambda^{m_e}a(x,\xi)\right\}\\
			\Hcal^{[m_e,m_\psi]}&=\left\{a(x,\xi)\in\Cinf(\R^{n+N})\sthat\forall \lambda,\mi\geq 1,\abs{x},\abs{\xi}>c, \; a(\lambda x,\mi\xi)=\lambda^{m_e}\mi^{m_\psi}a(x,\xi)\right\}.
		\end{aligned}
	\end{equation}
	We have then \textit{homogeneous symbols}:
	\begin{equation}
		\label{eq:homogeneous symbols}
		\begin{aligned}
			\sg{m_e,[m_\psi]}&=\Hcal^{[m_\psi]}_\psi\cap\sg{m_e,m_\psi}\\
			\sg{[m_e],m_\psi}&=\Hcal^{[m_e]}_e\cap\sg{m_e,m_\psi}\\
			\sg{[m_e],[m_\psi]}&=\Hcal^{[m_e],[m_\psi]}\cap\sg{m_e,m_\psi}
		\end{aligned}
	\end{equation}
\end{defin}
\begin{defin}
	\label{def:classicalsymbols}
	The spaces of \textit{$\xi$-classically homogeneous SG-symbols} and \textit{$\xi$-classical SG-symbols} are:
	\begin{equation}
		\label{eq:psi-classical symbols}
		\begin{aligned}
			\sg{[m_e],m_\psi}_{cl(\psi)}&=\left\{a\in\sg{[m_e],m_\psi}\,\sthat \exists a_k\in\sg{[m_e],[m_\psi-k]}\forall N\; a-\sum_{k=0}^N a_k\in\sg{m_e,m_\psi-N-1}\right\},\\
			\sg{m_e,m_\psi}_{cl(\psi)}&=\left\{a\in\sg{m_e,m_\psi}\,\sthat \exists a_k\in\sg{m_e,[m_\psi-k]}\forall N \;a-\sum_{k=0}^N a_k\in\sg{m_e,m_\psi -N-1}\right\}.
		\end{aligned}
	\end{equation}
	Similarly we have $x$-classically homogeneous and $x$-classical \textit{SG}-symbols:
	\begin{equation}
		\label{eq:x-classical symbols}
		\begin{aligned}
			\sg{m_e,[m_\psi]}_{cl(e)}&=\left\{a\in\sg{m_e,[m_\psi]}\,\sthat \exists a_k\in\sg{[m_e-k],[m_\psi]}\forall N\; a-\sum_{k=0}^N a_k\in\sg{m_e-N-1,m_\psi}\right\},\\
			\sg{m_e,m_\psi}_{cl(e)}&=\left\{a\in\sg{m_e,m_\psi}\,\sthat \exists a_k\in\sg{[m_e-k],m_\psi}\forall N \;a-\sum_{k=0}^N a_k\in\sg{m_e-N-1,m_\psi}\right\}.
		\end{aligned}
	\end{equation}
\end{defin}

\begin{defin}
	\label{def:classicalsgsymbols}
	The space of \textit{classical SG-symbols} or \textit{classical symbols with exit condition} is the set $\sg{m}_{cl}$ consisting of those symbols $a\in\sg{m}$ satisfying:
	\begin{enumerate}
		\item $\forall k\in\N \;\exists a^\psi_k\in\sg{m_e,[m_\psi-k]}_{cl(e)}\sthat\forall N$
		\begin{equation}
			\label{eq:psi-asymptotic expansion}
			a(x,\xi)-\sum_{k=0}^N a^\psi_k(x,\xi)\in\sg{m_e,m_\psi-N-1}_{cl(e)};
		\end{equation} 
		\item $\forall j\in\N \;\exists a^e_j\in\sg{[m_e-j],m_\psi}_{cl(\psi)}\sthat\forall N$
		\begin{equation}
			\label{eq:e-asymptotic expansion}
			a(x,\xi)-\sum_{j=0}^N a^e_j(x,\xi)\in\sg{m_e-N-1,m_\psi}_{cl(\psi)}.
		\end{equation} 
	\end{enumerate}
	We will, from now on, only deal with classical $SG$-symbols (and later operators). Therefore, the subscripts $cl,cl(e),cl(\psi)$ will be omitted and existence of asymptotic expansions tacitly assumed throughout.
\end{defin}

\begin{rem}
	\label{rem:global classical symbols}
	We remark that an alternative definition for classical $SG$-symbols has been given in \cite{witt1998pseudodifferentialnonsmoothsymbols}. Therein, the structure of symbol classes with values in a Fréchet space is explored and, in particular, it is proven that $S^m_{cl}(\R^n;S^l_{cl}(\R^n))\cong  S^m_{cl}\hat\tens_\pi S^l_{cl}\cong \sg{m,l}_{cl}.$ Here $S^m_{cl}$ denotes here the space of \textit{global classical symbols in one variable of order $m$}, namely the space of those smooth functions $a(\xi)$ on $\R^n$ which satisfy symbol estimates of order $m$ and admit an asymptotic expansion in homogeneous functions $a_k(\xi)$ of degree $m-k$. This would justify the terminology ``product-type symbols'' for the $SG$-classes. However, we refrain from its use since it might be easily confused with other, similar classes (e.g. the bi-singular operators defined by Rodino \cite{rodino1975pseudodifferentialoperatorsproduct}). We remark that, with this definition, it is also directly possible to define classical operators of complex order $(s_e,s_\psi)$ by saying that they are exactly those operators with symbol in $\sym{s_e}(\R^n)\hat\tens_\pi\sym{s_\psi}(\R^n)$.
\end{rem}

\begin{rem}
	\label{rem:exit and pseudo boxdot}
	We will often use the terms "exit $\boxdot$" and "pseudo-differential $\boxdot$" when speaking about properties of an object $\boxdot$ associated with the $e$-asymptotic expansion and the $\psi$-asymptotic expansion. For example, we will speak in a short while of ``exit symbol of order $m_e-k$'' and ``pseudo-differential symbol of order $m_\psi-j$'' for the maps $\sigma^{m_e-k}_e$ and $\sigma^{m_\psi-j}_\psi$, respectively. Also, for brevity's sake and convenience of notation, we often shorten $\sigma_\bullet^{m_\bullet-l}(a)$ with $a^\bullet_{m_\bullet-l}$ for $\bullet\in\{e,\psi,\psi e\}$ and, in particular, $a_\bullet\equiv\sigma_\bullet^{m_\bullet}(a)$. We call the maps $\sigma_\bullet^{m_\bullet}$ the $\bullet$-\textit{principal symbol maps}. We will see in the next chapter how helpful this lingo is in identifying properties of functions defined on different boundary hyper-surfaces of the scattering cotangent bundle.
\end{rem}

\begin{rem}
	\label{rem:classical symbols and homogeneity}
	We will, in general, use round brackets to denote homogeneity in the respective part of the domain (functions will be defined outside of the corresponding ``zero section''), while square brackets indicate eventual homogeneity as in \eqref{eq:eventually homogeneous functions}. It is then clear that, being interested only in the asymptotic behaviour of the symbols, we can pass from round to square brackets and vice-versa by multiplying with an excision function and adapting therefore the notion of ``convergence'' of asymptotic sums as in \eqref{eq:convergent asymptotic sum}. For a symbol $a$ in $\sg{[m_e],m_\psi}$, respectively $\sg{m_e,[m_\psi]}$, the asymptotic expansion \eqref{eq:e-asymptotic expansion}, respectively \eqref{eq:psi-asymptotic expansion}, is then trivial, in the sense that it consists of the function itself. 
\end{rem}
 
\begin{rem}
	\label{rem:asymptotic expansions are unique}
	The terms in the asymptotic sums of Definition \ref{def:classicalsgsymbols} are uniquely determined modulo elements in $\sg{-\infty,m_\psi}$ and $\sg{m_e,-\infty},$ respectively. That is, the \textit{eventual} behaviour (outside a compact neighbourhood of 0) is well-defined.
\end{rem}

The conditions in Definition \ref{def:classicalsgsymbols} allow us to canonically identify maps
\begin{equation}
	\label{eq:homogeneous symbol maps}
	\begin{aligned}
		\sigma^{m_e-k}_e&\colon\sg{m}\rightarrow\Hcal^{(m_e-k)}_e, \quad\sigma^{m_e-k}_e(a)(x,\xi)=a^e_{m_e-k}(x,\xi),\\
		\sigma^{m_\psi-j}_\psi&\colon\sg{m}\rightarrow\Hcal^{(m_\psi-j)}_\psi, \quad\sigma^{m_\psi-j}_\psi(a)(x,\xi)=a^\psi_{m_\psi-j}(x,\xi),
	\end{aligned}
\end{equation}
taking values, respectively, in the classes $\sg{(m_e-k),m_\psi}$ and $\sg{m_e,(m_\psi-j)}$. In particular, $\sigma_e^{m_e-k}(a)$ admits an asymptotic expansion in the classes $\sg{(m_e-k),m_\psi-j}, j\geq0$, so that we can canonically identify bi-homogeneous elements $\sigma_{\psi}^{m_\psi-j}\sigma_e^{m_e-k}(a)\in\Hcal^{(m_e-k,m_\psi-j)}$. The same process, applied to $\sigma_\psi^{m_\psi-j}(a)$ in the classes $\sg{m_e-k,(m_\psi-j)}$, $k\geq0$, produces bi-homogeneous functions $\sigma_{e}^{m_e-k}\sigma_\psi^{m_\psi-j}(a)\in\Hcal^{(m_e-k,m_\psi-j)}$, so that we naturally are interested in the relation between the two. The following Lemma (Exercise 3, Section 8.2 in \cite{egorov1997pseudodifferentialsingularities}) tells us that it is actually (and luckily) quite simple.

\begin{lemma}
	\label{lemma:symbol maps compatibility}
	For any $m\in\R^2, k,j\in\N$ the maps $\sigma_\psi^{m_\psi-k}$ and $\sigma_e^{m_e-j}$ commute and define functions $a^{\psi e}_{jk}\equiv\sigma_{\psi e}^{m_e-k,m_\psi-j}(a)\equiv\sigma_e^{m_e-k}\sigma_\psi^{m_\psi-j}(a)$. In particular, with any classical $SG$-symbol of order $m\in\R^2$, there is canonically associated an ``infinite-dimensional matrix'' (we sometimes call this an \textit{asymptotic matrix}) of bi-homogeneous functions $\{a^{\psi e}_{jk}\}_{j,k\geq 0}$ with $a^{\psi e}_{jk}\in\Hcal^{(m_e-k,m_\psi-j)}$, such that each ``row $j$'' or ``column $k$'' can be asymptotically summed to give $\sigma_\psi^{m_\psi-k}(a)$ or $\sigma_e^{m_e-j}(a),$ respectively. 
\end{lemma}

\begin{rem}
	In order to lighten the notation, we often omit the superscript $\psi e$ when dealing with asymptotic matrices. Notice, in addition, that an asymptotic matrix can always be considered as a single asymptotic expansion. It suffices to consider the triangular enumeration of $\N^2$ and sum first each diagonal (a finite sum), so that we are left with a usual asymptotic expansion and we can determine its sum modulo $\Scal$. Therefore, with the help of an excision function, we can always sum an asymptotic matrix.
\end{rem}

Similarly to the standard class $\psdo{}(\R^n)$, we define a notion of principal symbol. 

\begin{defin}
	\label{def:principal sg-symbol}
	For $a\in\sg{m}$, the \textit{principal symbol} of $a$ is the triple of functions $\sigma_{pr}^m(a)\equiv(\sigma_e^{m_e}(a),
	\sigma_\psi^{m_\psi}(a),\sigma_{\psi e}^{m}(a))\equiv(a_e,a_\psi,a_{\psi e})$ canonically associated with $a$ as in Lemma \ref{lemma:symbol maps compatibility}.
\end{defin}

\begin{prop}[Properties of the principal symbol]
	\label{prop:principalsgsymbols}
	The following holds true:
	\begin{enumerate}
	\item 
		For $m\in\R^2$ the quotient $\SymG{m}\equiv\setquotient{\sg{m}}{\sg{m-\indi}}$ contains the principal symbols $(a_e,a_\psi,a_{\psi e})$ of the $a\in\sg{m}$. Equivalently, $\SymG{m}$ contains the pairs $(a_e,a_\psi)$ with $a_e\in\sg{(m_e),m_\psi}$ and $a_\psi\in\sg{m_e,(m_\psi)}$ such that $\sigma_e^{m_e}(a_\psi)=\sigma_\psi^{m_\psi}(a_e)$. 
	\item 
		The direct sum of the principal symbol spaces $\SymG{}=\bigoplus_{m\in\Z}\SymG{m}$ has the structure of a commutative graded module over $\SymG{0}$.
	\item 
		We can compute the exit and pseudo-differential symbols as
		\begin{equation}
			\label{eq:principal symbol as limit}
			\begin{aligned}
				\sigma_{e}^{m_e}(a)(x,\xi)&=\lim_{\mi\rightarrow\infty}\mi^{-m_e}a(\mi x,\xi),\\
				\sigma_{\psi}^{m_\psi}(a)(x,\xi)&=\lim_{\mi\rightarrow\infty}\mi^{-m_\psi}a(x,\mi\xi).
			\end{aligned}
		\end{equation}
	\item 
		The $\bullet$-principal symbol maps are multiplicative on the respective components, i.e. $\sigma_\bullet^{m_\bullet+l_\bullet}(ab)=\sigma_\bullet^{m_\bullet}(a)\sigma_\bullet^{l_\bullet}(b)$ if $a\in\sg{m},b\in\sg{l}$. 
		
	\item 
	    If $a\in\sg{m}$ and $\sigma^{m_e}_{e}(a)=0=\sigma^{m_\psi}_{\psi}(a)$, then $a\in\sg{m-\indi}.$ 
	\item
	    $\sigma_{pr}\equiv\bigoplus_{m\in\Z}\sigma_{pr}^m$ defines a surjective homomorphism of $\sg{}$ onto $\SymG{}$.
	\end{enumerate}
\end{prop}

\begin{rem}
	\label{rem:principal symbols are functions}
	Importantly, in Proposition \ref{prop:principalsgsymbols} we are \textit{not} identifying functions with asymptotic expansions in $\sigma_{\psi}\sg{m}$ and $\sigma_e\sg{m}$, that is, we are not working in $\setquotient{\sg{m}}{\sg{-\infty,m_\psi}}$ and $\setquotient{\sg{m}}{\sg{m_e,-\infty}}$. In these quotient spaces the natural operation to consider is the Leibniz product, namely the one defined as the (asymptotic expansion of the) symbol of the composition of pseudo-differential operators (defined below in Proposition \ref{prop:sg calculus}). While these can be endowed with the structure of a commutative algebra, we shall need to speak of the pointwise value of elements in the quotient later in Chapter \ref{chap:OPI}. Were we to identify elements up to a Schwartz function, we would lose this property. Thus, the components $a_\bullet$ are well-defined smooth \textit{functions}, homogeneous of degree $m_\bullet$ in the corresponding set of variables, which \textit{admit} an asymptotic expansion in terms of homogeneous functions in the other variables. 
\end{rem}

The notion of principal symbol just introduced has but one problem: it is not given by a single function. While this is often not an issue, it might be comfortable to have at hand a function containing in itself the composite asymptotic information of the principal symbol without keeping track of all the terms in the asymptotic matrix. The following concept of associated symbol fulfils this rôle, however it depends on the choice of an excision function, which is used to patch together the components defined on different spaces.

\begin{defin}
	\label{def:associated sg symbol}
	For a symbol $p\in\sg{m}$ denote
	\begin{equation}
		\label{eq:associated sg symbol}
		\check{p}(x,\xi)\equiv \chi (\xi)p_\psi(x,\xi)+\chi(x)p_{e}(x,\xi)-\chi(x)\chi(\xi)p_{\psi e}(x,\xi),
	\end{equation}
	for $\chi$ a smooth excision function in $\R^n$. $\check{p}$ is called the \textit{associated symbol} or \textit{principal part} of $p$, terminology which is justified in view of Lemma \ref{lemma:associated and principal symbols} below.
\end{defin}

\begin{lemma}
	\label{lemma:associated and principal symbols}
	For any $p\in\sg{m}$ we have $\check{p}\in\sg{m}$ and $p-\check{p}\in\sg{m-\indi}$. That is, $\sigma_{pr}(\check{p})=\sigma_{pr}(p)$.
\end{lemma}

\begin{rem}
	\label{rem:heuristics of associated symbols}
	Any choice of an excision function produces therefore a 1-1 correspondence between principal and associated symbols, as the process of constructing $\check{p}$ from $\sigma_{pr}(p)$ is easily reversed by computing $\sigma_{pr}(\check{p})$.
	Heuristically speaking, the associated symbol is the ``sum'' of the outermost row and column of the asymptotic matrix. Lemma \ref{lemma:associated and principal symbols} is then just the statement that subtracting these from an asymptotic matrix reduces the order by $\indi$. 
\end{rem}

We recall hereafter some facts about ellipticity in the $SG$-calculus. The notion we define here is a refinement of the usual ellipticity condition for operators on compact manifolds.

\begin{defin}
	\label{def:elliptic sg-symbol}
	$a\in\sg{m}$ is \textit{elliptic} if all components of its principal symbol do not vanish on their respective domains.
\end{defin}

The previous discussion shows that the principal symbol is the equivalence class of $p\in\sg{m}$ in the quotient $\SymG{m}$. Elliptic symbols admit ``inverses'' in $\SymG{-m}$. Namely, it suffices to construct the symbol
\[
q(x,\xi)=\chi(x) p_e(x,\xi)^{-1}+\chi(\xi) p_\psi(x,\xi)^{-1}-\chi(x)\chi(\xi)p_{\psi e}(x,\xi)^{-1}
\]
to see that the products $pq,qp$ lie in $\sg{0}$ and their principal symbols equal (1,1,1). 

Parallel to the classical theory of Hörmander, we can also introduce amplitudes modelled after $SG$-symbols to gain more flexibility.
\begin{defin}
	\label{def:sg amplitude}
	A function $a\in\Cinf(\R^{n}\times\R^n\times\R^k)$ is called an \textit{amplitude of $SG$-type}, or just $SG$-amplitude, of order $(m_1,m_2,m_3)$, if for all $\alpha,\beta\in\N^n,\gamma\in\N^k$ it satisfies the global estimate on $\R^n$
	\begin{equation}
		\label{eq:sg amplitude}
		\abs{\del^\alpha_x\del^\beta_y\del^\gamma_{\xi}a}\lesssim\braket{x}^{m_1-\abs\alpha}\braket{y}^{m_2-\abs\beta}\braket{\xi}^{m_3-\abs\gamma}.
	\end{equation}
\end{defin}

Notice that an $SG$-symbol of order $(m_e,m_\psi)$ is just an $SG$-amplitude of order $(m_e,0,m_\psi)$ which is independent of $y$. The discussion on classicality can be generalised directly to the case of an amplitude by asking that it admits asymptotic expansions separately in all three sets of variables. We come now to $SG$-pseudo-differential operators. We remark that most definitions and results below make sense or could be phrased for non-classical operators as well. However, in view of our future needs, we limit ourselves to classical objects and, as before, drop the subscripts ``cl'' from our notation.

\begin{defin}
	\label{def:sgoperator}
	For $a\in\sg{m_1,m_2,m_3}$ and $u\in\Scal$ let (in the sense of oscillatory integrals)
	\begin{equation}
		\label{eq:sgoperator}
		\Op(a)u(x)\equiv\int e^{\im (x-y)\xi} a(x,y,\xi)u(y)\de y\dbar\xi
	\end{equation}
	and call it the \textit{operator} defined by the amplitude $a$. 
\end{defin}
\begin{lemma}
	\label{lemma:operators with amplitudes or symbols}
	Le $A=\Op(a)$ be an $SG$-pseudo-differential operator defined by an $SG$-amplitude of order $(m_1,m_2,m_3)$. There exists an $SG$-symbol $b\in\sg{m_1+m_2,m_3}$ such that $\Op(a)=\Op(b)$. Vice-versa, for every $SG$-symbol $b\in\sg{m_e,m_\psi}$ and any $t\in\R$ we can find an $SG$-amplitude $a_t$ of order $(t,m_e-t,m_\psi)$ with $\Op(a_t)=\Op(b)$.
\end{lemma}
\begin{defin}
	\label{def:sgoperators}
	We let $\lg{m}$ denote the class of pseudo-differential operators defined by either symbols or amplitudes of $SG$-type and symbol order $m$ and employ the notation $\RG\equiv\lg{-\infty,-\infty}$. $\RG$ is a two-sided ideal in each $\lg{m}$ and we let $\BG{}$ denote the quotient algebra $\setquotient{\lg{}}{\RG}$.
\end{defin}
\begin{lemma}
	\label{lemma:sg operators are continuous}
	For any $m\in\R^2$ the operators in $\lg{m}(\R^n)$ act continuously between $\Cinf_c(\R^n)$ and $\Cinf(\R^n)$, and also on $\Scal(\R^n)$. Moreover operators $A\in\lg{0,m_\psi}$ act continuously on $\Cinf_b(\R^n)$ and we can compute the symbol via
	\begin{equation}
		\label{eq:recovery of symbol}
		\sigma(A)(x,\xi)=e^{-\im x\xi}Ae^{\im x\xi}.
	\end{equation}
\end{lemma}

\begin{prop}
	\label{prop:sg calculus}
	The following holds true:
	\begin{enumerate}
		\item Each $A\in\RG$ is of the form $\Op(a)$ with $a\in\Scal(\R^{2n})$;
		\item The map $\Op\colon\sg{m}\rightarrow \lg{m}$ is bijective.
	\end{enumerate}
\end{prop}

\begin{theo}[Theorem 7, Section 8.2 in \cite{egorov1997pseudodifferentialsingularities}]
	\label{theo:composition}
	Let $A=\Op(a)\in\lg{m}, B=\Op(b)\in\lg{l}$. Then $AB=\Op(c)\in\lg{m+l}$ and we have the asymptotic expansion (the so-called \textit{Leibniz product} of $a$ and $b$)
	\begin{equation}
		\label{eq:asymptotic expansion composition}
		c(x,\xi)\sim\sum_{\alpha\geq 0}\frac{(-\im)^{\abs\alpha}}{\alpha!}\partder{{}^\alpha a}{\xi^\alpha}(x,\xi)\partder{{}^\alpha b}{x^\alpha}(x,\xi).
	\end{equation}
	Moreover, the principal symbol maps are multiplicative, i.e. 
	\begin{equation}
		\label{eq:principal symbol is multiplicative}
		\begin{aligned}
			\sigma_e^{m_e+l_e}(AB)&=\sigma_e^{m_e}(A)\sigma_e^{l_e}(B),\\
			\sigma_\psi^{m_\psi+l_\psi}(AB)&=\sigma_\psi^{m_\psi}(A)\sigma_\psi^{l_\psi}(B),\\
			\sigma_{\psi e}^{m+l}(AB)&=\sigma_{\psi e}^{m}(A)\sigma_{\psi e}^{l}(B).
		\end{aligned}
	\end{equation}
	In particular, there is an algebra isomorphism $\BG{}\cong\setquotient{\sg{}}{\sg{-\infty\indi}}$, where on the left we have composition and on the right the Leibniz product. 
\end{theo}

\begin{theo}
	\label{theo:adjoint}
	Let $A\in\lg{m}$ and $A^\dagger$ be the formal $\leb{2}$-adjoint, defined by $\inner{Au}{v}=\inner{u}{A^\dagger v}$ for all $u,v\in\Scal$. Then $A^\dagger\in\lg{m}$ and, if $A=\Op(a)$, we have $A^\dagger=\Op(a^\dagger)$ for $a^\dagger$ admitting the asymptotic expansion
	\begin{equation}
		a^\dagger(x,\xi)\sim\sum_{\alpha\geq 0}\frac{(-\im)^{\abs\alpha}}{\alpha!}\del^\alpha_x \del^\alpha_\xi\bar{a(x,\xi)}.
	\end{equation}
\end{theo}

\begin{lemma}
	\label{lemma:principal symbol sequence}
	There is a short exact sequence for any $m\in\Z^2$
	\begin{equation}
		\label{eq:principal symbol sequence}
		0\rightarrow \lg{m-\indi}\rightarrow\lg{m}\xrightarrow{\sigma_{pr}}\SymG{m}\rightarrow 0.
	\end{equation}
\end{lemma}

We point out that for operators in $\lg{}$ we have a characterization of the Fredholm property in terms of the ellipticity. This was historically one of the reasons for the introduction of global calculi.

\begin{theo}
	\label{theo:fredholm}
	The following holds true:
	\begin{enumerate}
		\item The space $\lg{-\infty\indi}(\R^n)$ consists of compact operators on $\leb{2}(\R^n)$;
		\item An operator $P\in\lg{m}$ is elliptic if and only if it admits a parametrix $Q\in\lg{-m}$, i.e. an operator $Q$ such that $PQ-I,QP-I\in\lg{-\infty\indi}$;
		\item The following are equivalent: 
		\begin{enumerate}
			\item $P\in\lg{m}$ is elliptic.
			\item $P$ extends to a Fredholm operator $P\colon\HG{l}\rightarrow\HG{l-m}$ for some $l\in\Z^2$.
			\item $P$ extends to a Fredholm operator $P\colon\HG{l}\rightarrow\HG{l-m}$ for all $l\in\Z^2$.
		\end{enumerate}
	\end{enumerate}
\end{theo}

We end this section by recalling the existence of so-called \textit{order reducing operators}.

\begin{lemma}
	\label{lemma:classical order reductions}
	There exist classical, elliptic, invertible operators $P\in\lg{\indi_e},Q\in\lg{\indi_\psi}$ giving isomorphisms $\lg{m}\rightarrow\lg{m+\indi_\bullet}$ by composition. In particular we can take $P=\Op\braket{x}, Q=\Op\braket{\xi}$.
\end{lemma}

Taking advantage of Lemma \ref{lemma:operators with amplitudes or symbols} we introduce a notion of ellipticity for amplitudes.

\begin{defin}
	\label{def:elliptic amplitudes}
	We say that an amplitude $a\in\sg{m_1,m_1m_e}$ is \textit{elliptic} if can be quantised to an elliptic symbol. Namely, $a$ is elliptic if and only if $a$ defines an operator $A\in\lg{m_1+m_2,m_3}$ whose symbol is elliptic.
\end{defin}

\subsection{$SG$-symbols and the symplectic structure}

We equip $\R^{2n}\cong T^\ast \R^n$ with the standard symplectic structure $\omega=\de\xi_i\wedge\de x^i$, where $\xi_i$ is the canonically dual coordinate to $x^i$. Recall that this induces a Poisson bracket on smooth functions by 
\begin{equation}
	\label{eq:poisson bracket}
	\{f,g\}=\partder{f}{\xi_i}\partder{g}{x^i}-\partder{f}{x^i}\partder{g}{\xi_i}.
\end{equation}
The interplay between this operation and the $SG$-calculus will help us clarify the situation for the study of singular symplectomorphisms in Chapter \ref{chap:scattering geometry}.

\begin{prop}
	\label{prop:symplectic properties of sg}
	The following holds true:
	\begin{enumerate}
		\item 
			$\sg{}\subset\Cinf(\R^n\times\R^n)$ is a commutative algebra with respect to the pointwise product, which is, in addition, bi-filtered.
		\item 
			The Poisson bracket gives the structure of a Lie algebra to $\sg{}$ and in particular is a bi-filtered bi-derivation of $\sg{}$. That is, $\{a,b\}\in\sg{m+k-\indi}$ if $a\in\sg{m},b\in\sg{k},m,k\in\R^2$. 
		\item 
			$\sg{0}$ is a sub-algebra and Lie sub-algebra of $\sg{}$, and $\SymG{0}$ inherits the structure of a commutative Lie algebra (namely, the Poisson bracket is trivial in the quotient).
		\item 
			The Poisson bracket induces Lie algebra structures on $\sg{\indi}$ and $\SymG{\indi}$, and the principal symbol map is then an homomorphism of Lie algebras. More specifically, the $\bullet$-principal symbol of $\{a,b\}$ only depends on the $\bullet$-principal symbols of $a$ and $b$ and can be computed explicitly as
			\begin{equation}
				\begin{aligned}
					\sigma_\psi^1(\{a,b\})&=\{\sigma_\psi(a),\sigma_\psi(b)\},\\
					\sigma_e^1 (\{a,b\})&=\{\sigma_e(a),\sigma_e(b)\},\\
					\sigma_{\psi e}^\indi(\{a,b\})&=\{\sigma_{\psi e(a)},\sigma_{\psi e}(b)\}.
				\end{aligned}
			\end{equation}
		\item 
			$\lg{\indi}$ is a Lie algebra with respect to the commutator and we have an isomorphism \[(\setquotient{\lg{\indi}}{\lg0},[\,,])\cong(\SymG{\indi},\im\{\,,\}).\]			
	\end{enumerate}
\end{prop}
\begin{proof}
	The structure of $\sg{}$ as an algebra with respect to the pointwise product has already been analysed in the previous section. We turn to the statements concerning the Poisson bracket.
	Recall that asymptotic expansions and derivatives commute, since the derivative of a classical symbol is again classical. Therefore if $a\sim\sum_{j,k\geq0}a_{kj}^{\psi e},b\sim\sum_{j,k\geq0}b^{\psi e}_{kj}$ we can write the Poisson bracket $\{a,b\}$ asymptotically as (omitting the superscript $\psi e$ and writing $\del^r=\partder{}{\xi_r}$ and $\del_s=\partder{}{x^s}$ for convenience):
	\begin{equation}
		\label{eq:asymptotic expansion product}
		\begin{aligned}
			\del^ra\del_sb&\sim\del^ra_{00}\del_sb_{00}\\
			&+\sum_{j\geq 1}\del^ra_{00}\del_s b_{0j}+\sum_{k\geq 1}\del^ra_{00}\del_sb_{k0}+\sum_{k\geq 1}\del^ra_{k0}\del_sb_{00}+\sum_{j\geq 1}\del^ra_{0j}\del_sb_{00}\\
			&+\sum_{j,k\geq1}\del^r a_{k0}\del_s b_{0j}+\sum_{k,k'\geq1}\del^ra_{k0}\del_s b_{k'0}+\sum_{k,j\geq 1}\del^ra_{00}\del_sb_{kj}\\
			&+\sum_{j,j'\geq 1}\del^ra_{0j}\del_sb_{0j'}+\sum_{j,k\geq1}\del^ra_{0j}\del_s b_{k0}+\sum_{j,k\geq 1}\del^ra_{kj}\del_sb_{00}\\
			&+\sum_{j,j',k\geq 1}\del^ra_{0j}\del_sb_{kj'}+\sum_{k,k',j\geq 1}\del^ra_{k0}\del_sb_{k'j}+\sum_{j,k,k'\geq 1}\del^ra_{kj}\del_sb_{k'0}+\sum_{j,j',k\geq 1}\del^ra_{kj}\del_sb_{0j'}\\
			&+\sum_{k,k',j,j'\geq1}\del^ra_{kj}\del_sb_{k'j'}.
		\end{aligned}
	\end{equation}
	We group the terms according to their homogeneity to obtain
	\begin{equation}
		\label{eq:asymptotic expansion product 2}
		\begin{aligned}
			\del^ra\del_sb&\sim\del^ra_{00}\del_s b_{00}\\
			&\quad+\del^ra_{10}\del_sb_{00}+\del^ra_{01}\del_s b_{00}+\del^r a_{00}\del_s b_{10}+\del^ra_{00}\del_sb_{01}\\
			&\quad+\del^ra_{20}\del_sb_{00}+\del^ra_{02}\del_sb_{00}+\del^ra_{00}\del_sb_{20}+\del^ra_{00}\del_sb_{02}\\
			&\quad+\del^ra_{11}\del_sb_{00}+\del^ra_{10}\del_sb_{10}+\del^ra_{10}\del_sb_{01}\\
			&\quad+\del^ra_{01}\del_sb_{10}+\del^ra_{01}\del_sb_{01}+\del^ra_{00}\del_sb_{11}\\
			&\quad+\dots\\
			&=\sum_{m\geq0}\sum_{\substack{k,j\geq 0\\k+j=m}}c^r_s(k,j),
		\end{aligned}
	\end{equation}
	where
	\[
	c^r_s(k,j)=\sum_{\substack{l_1+l_3=k\\l_1,l_3\geq 0}}\sum_{\substack{l_2+l_4=j\\l_2,l_4\geq0}}\del^ra_{l_1l_2}\del_sb_{l_3l_4}\in\Hcal^{(m_e+l_e-1-k,m_\psi+l_\psi-1-j)}
	\]
	are the components of the asymptotic matrix of $\del^ra\del_sb$. We can of course obtain a similar expression for $\del_sa\del^rb$. Taking the trace $r=s$ and subtracting the two expressions gives then the asymptotic matrix of the Poisson bracket $\{a,b\}$ for symbols of general orders $m,l\in\R^2$. 
	
	Since taking traces and differences in $\sg{m}$ cannot increase the order of the symbols in any fashion, the class of $\{a,b\}$ in the quotient spaces $\SymG{}$ can be computed from the classes of $a$ and $b$, namely, from their principal symbols. Indeed we see directly from \eqref{eq:asymptotic expansion product 2} that the outer row and column of the asymptotic matrix of $\del^ra\del_s b$ correspond to taking $k=0$ or $j=0$ in $c^r_s(k,j)$, and thus only depend on the outer row and column of $a$ and $b$. Furthermore it is clear that the Poisson bracket commutes with our asymptotic expansions: indeed, the bracket of classical symbols is again a classical symbol.
	
	We are mainly interested in the algebraic structure of the spaces $\SymG{0}$ and $\SymG{\indi}$. For the former, notice that $a,b\in\sg{0}$ implies that $\del^ra\del_sb\in\sg{-\indi}$ and the same must be true for $\{a,b\}$. Hence, the Poisson bracket vanishes on $\SymG{0}$. For the latter, we start from \eqref{eq:asymptotic expansion product 2} to compute
	\begin{equation}
		\label{eq:poisson bracket first order asymptotic expansion}
		\{a,b\}\sim\{a_{00},b_{00}\}+\sum_{k\geq 1}\sum_{l_1+l_2=k}\{a_{l_10},b_{l_20}\}+\sum_{j\geq 1}\sum_{l_3+l_4=j}\{a_{0l_3},b_{0l_4}\}+\sum_{j,k\geq 1}(c^r_r(k,j)-\tilde{c}^r_r(k,j)),
	\end{equation}
	where $\tilde{c}^r_s$ are the components of the asymptotic matrix of $\del_sa\del^rb$. Now, all the terms in the last sum are at most in $\sg{0}$, while the others are just, respectively, the symbol $\sigma_{\psi e}(\{a,b\})$, the $e$-asymptotic expansion of $\sigma_\psi(\{a,b\}-\{a_{00},b_{00}\})$ and the $\psi$-asymptotic expansion of $\sigma_e(\{a,b\}-\{a_{00},b_{00}\})$.
	
	In fact, more is true: the $\bullet$-principal symbol of $\{a,b\}$ only depends on the $\bullet$-principal symbols of $a$ and $b$ and can be computed explicitly (notice that the previous computation only gives the classes of $\sigma_\bullet(\{a,b\})$ up to $\sg{-\indi_\bullet\infty})$. To show this, consider for example symbols $a,b\in\sg{1,k}$ and $r\in\sg{0,l}$, and look at
	\begin{equation}
		\label{eq:poisson bracket first order}
		\{a+r,b\}=\{a,b\}+\{r,b\}.
	\end{equation}
	It is clear that $\{a,b\}\in\sg{1,k+l-1}$ and $\{r,b\}\in\sg{0,k+l-1}$, in view of the properties of the calculus. Then, the class of $\{a+r,b\}$ in $\SymG{1,\bullet}$ does not depend on $r$, the limits $\lim\limits_{\lambda\rightarrow\infty}\lambda^{-1}\del^j(a+r)(\lambda x,\xi)\del_jb(\lambda x,\xi)$ and $\lim\limits_{\lambda\rightarrow\infty}\lambda^{-1}\del_j(a+r)(\lambda x,\xi)\del^jb(\lambda x,\xi)$ exist, and do not depend on $r$ either. We can then directly compute, for $x\neq 0$, that
	\begin{equation}
		\label{eq:poisson bracket exit symbol}
		\begin{aligned}
			\sigma^1_e(\{a,b\})&=\lim\limits_{\lambda\rightarrow\infty}\lambda^{-1}(\del^ja(\lambda x,\xi)\del_jb(\lambda x,\xi)-\del_j a(\lambda x,\xi)\del^jb(\lambda x,\xi))\\
			&=\del^ja_e(x,\xi)\del_jb_e(x,\xi)-\del_ja_e(x,\xi)\del^jb_e(x,\xi)\\
			&=\{a_e,b_e\}.
		\end{aligned}
	\end{equation}
	Since this can be done in the same way for the other components of $\sigma_{pr}(\{a,b\})$, we conclude that $\SymG{\indi}$ is a Lie algebra with respect to the Poisson bracket acting component-wise.
	
	We now turn to examine the commutator of two $SG$-pseudo-differential operators. For $A,B\in\lg{\indi}$ with symbols $a,b$ the asymptotic expansions for the products $AB, BA$ can be written as
	\begin{equation}
		\label{eq:asymptotic expansions products}
		\begin{aligned}
			\sigma(AB)&\sim ab+\del_{\xi_j}a(x,\xi)D_{x^j}b(x,\xi)+\sum_{\abs\alpha\geq 2}\frac{1}{\alpha!}\del_{\xi}^\alpha a(x,\xi)D_{x}^\alpha b(x,\xi),\\
			\sigma(BA)&\sim ab+\del_{\xi_j}b(x,\xi)D_{x^j}a(x,\xi)+\sum_{\abs\alpha\geq 2}\frac{1}{\alpha!}\del_{\xi}^\alpha b(x,\xi)D_{x}^\alpha a(x,\xi),
		\end{aligned}
	\end{equation}
	so that, taking the difference, we obtain
	\begin{equation}
		\label{eq:asymptotic expansion commutator}
		\begin{aligned}
			\sigma([A,B])&\sim \im\left(\del_{\xi_j}a(x,\xi)\del_{x^j}b(x,\xi)-\del_{\xi_j}b(x,\xi)\del_{x^j}a(x,\xi)\right)\\
			&+\sum_{\abs\alpha\geq 2}\frac{\im^\alpha}{\alpha!}\left(\del_\xi^\alpha a\del_x^\alpha b-\del_x^\alpha a\del_\xi^\alpha b\right).
		\end{aligned}
	\end{equation}
	In view of the properties of $SG$-symbols we have $\del_\xi^\alpha a,\del_\xi^\alpha b\in\sg{-1,1},\del_{x}^\alpha a,\del_x^\alpha b\in\sg{1,-1}$ whenever $\abs{\alpha}\geq 2$, so that every time we take a product as in the second term in \eqref{eq:asymptotic expansion commutator} we obtain at most a symbol in $\sg{0}$. Hence, in the quotient it holds true that
	\begin{equation}
		\label{eq:principal symbol commutator}
		\sigma([A,B])\sim\im\{a,b\},
	\end{equation}
	and we can take advantage of \eqref{eq:poisson bracket exit symbol} and its $\psi-$ and $\psi e-$counterparts to get the required formulas at the level of principal symbols.
\end{proof}

\begin{rem}
	As one can directly deduce from the proof of Proposition \ref{prop:symplectic properties of sg}, it holds true that $\sigma_{pr}(\{a-\check{a},b\})=0$ for any $a,b\in\sg{\indi}$. Accordingly, we can also compute the Poisson bracket from the associated symbols as
	\begin{equation}
		\label{eq:associated symbol poisson bracket}
		\begin{aligned}
			\sigma^1_\psi(\{\check{p},\check{q}\})&=\{p_\psi,q_\psi\},\\
			\sigma^1_{e}(\{\check{p},\check{q}\})&=\{p_{e},q_{e}\},\\
			\sigma^\indi_{\psi e}(\{\check{p},\check{q}\})&=\{p_{\psi e},q_{\psi e}\}.
		\end{aligned}
	\end{equation}
\end{rem}

\subsection{$SG$-Fourier Integral Operators}
In this Section we briefly describe the calculus of $\Qcal$-operators of Andrews \cite{andrews2009SGFIOcomposition}. This is the class of FIOs $A\colon\Scal\rightarrow\Scal$ which we will need in our discussion in Chapter \ref{chap:OPI}. Most of the material hereafter is taken directly from \cite{andrews2009SGFIOcomposition}, with some notable exceptions. We begin with the standard class $\Qcal$, before introducing a slightly modified (compared with the original source) \textit{generalised type $\Qcal$ class}, in that we localise and assume classicality throughout. Subsequently, we give a small, albeit important to our purposes, generalisation of Coriasco's Egorov--type Theorem (Proposition 14 in \cite{coriasco1999fourierintegraloperators}) for operators in the class $\Qcal$. The functions $f,g$ appearing in the next definition, and also later in the definition of the generalised class, will be referred to as \textit{phase components}.

\begin{defin}
	\label{def:Q-phase function}
	We say that a real-valued function $\phi(x,y,\xi)=f(x,\xi)+g(y,\xi)$ is a \textit{type $\Qcal$ phase function}, and write $\phi=f+g\in\Qcal$, if the following assumptions are satisfied:
	\begin{enumerate}
		\item $f,g\in\sg{\indi}(\R^n\times\R^n)$;
		\item $\braket{\nabla_x f(x,\xi)},\braket{\nabla_y g(y,\xi)}\sim\braket{\xi}$
		\item $\braket{\nabla_\xi f(x,\xi)}\sim\braket{x};$
		\item $\braket{\nabla_\xi g(y,\xi)}\sim\braket{y};$
		\item $\det(\del_{x^i}\del_{\xi_j}f(x,\xi)),\det(\del_{y^i}\del_{\xi_j}g(y,\xi))\gtrsim 1;$
	\end{enumerate}
\end{defin}

\begin{defin}
	\label{def:Q-FIO}
	A \textit{type $\Qcal$ Fourier Integral Operator} ($\Qcal$-FIO) is an operator $A\colon\Scal(\R^n)\rightarrow\Scal(\R^n)$, defined by an oscillatory integral
	\begin{equation}
		\label{eq:Q FIO}
		FIO(\phi,a)u(x)=\int e^{\im(f(x,\xi)+g(y,\xi))}a(x,y,\xi)u(y)\de y\dbar\xi,
	\end{equation}
	with phase $\phi=f+g\in\Qcal$ and amplitude $a(x,y,\xi)\in\sg{{m_1,m_2,m_3}}(\R^{3n})$.
\end{defin}
\begin{prop}[Properties of $\Qcal$-FIOs]
	\label{prop:properties of QFIOs}
	Let $A$ be a $\Qcal$-FIO with symbol $a$ and phase $\phi=f+g$. Then
	\begin{enumerate}
		\item $A\colon\Scal\rightarrow\Scal$ is well-defined and continuous;
		\item With respect to the inner product $\inner{u}{v}=\int u(x)\bar{v(x)}\de x$ we have that the formal adjoint $A^\dagger$ is given by the $\Qcal$-FIO with phase $\phi^\dagger(x,y,\xi)=-g(x,\xi)-f(y,\xi)$ and symbol $a^\dagger(x,y,\xi)=\bar{a(y,x,\xi)}$;
		\item $A$ extends to $A\colon\Scal'\rightarrow\Scal'$ continuously.		
	\end{enumerate}
\end{prop}

We now introduce a generalised class of phases. We remark first that our upcoming definition differs slightly from the one in the original work of Andrews, in that he defines FIOs whose phases were asked to satisfy asymmetric assumptions in $x$ and $y$. Indeed, the conditions below were required to hold true globally in $y$ and only locally in $x$, in order to retain the global non-degeneracy of the second phase component. However, for our purposes, we only need the assumption of non-degeneracy to hold true on the supports of locally chosen amplitudes on a certain (singular) Legendrian submanifold. In other words, our discussion will always be localised and we have accordingly decided to ask that the same conditions hold true only locally for $x$ and $y$ variables.

We consider functions depending on $x,y\in\R^n$ and $\theta\in\R^{n+d}$ for some $n>0,d\geq 0$ (if $d=0$ then we set $\Qcal(a)=\Qcal$ in what follows). Also we relax some of the conditions imposed on $\Qcal$ to hold true only on the support of a given amplitude.
\begin{defin}
	\label{def:generalised Q-class}
	Let $a\in\sg{m_1,m_2,m_3}(\R^n\times\R^n\times\R^{n+d})$ be an $SG$-amplitude and $\phi(x,y,\theta)=f(x,\theta)+g(y,\theta)$ for some smooth $f,g$ and $(x,y,\theta)\in\R^n\times\R^n\times\R^{n+d}$. We write $\phi\in\Qcal(a)$ if, on $\supp (a)$, the following conditions hold true: we have $f,g\in\sg{\indi}$ with $\braket{\nabla_{\!x}\! f(x,\theta)},\braket{\nabla_{\!y}g(y,\theta)}\sim \braket{\theta}$, and we can find (possibly after rearranging) a splitting $\theta=(\xi,\eta)\in\R^n\times\R^d$ and an open set $V_\phi\subset\R^d$ with $\supp(a)\subset \R^n\times\R^n\times\R^n\times V_\phi$ such that:
	\begin{enumerate}
		\item
			$\braket{\nabla_{\!\xi}g(y,\theta)}\sim\braket{y}$ on $\R^n\times\R^{n}\times V_\phi$;
		\item
			$(\del_{y^i}\del_{\xi_j}g(y,\theta))$ has maximal rank on $\R^n\times\R^{n}\times V_\phi$ and the absolute value of its determinant is uniformly bounded away from 0;
		\item
			$\del_{y^i}\del_{\xi_j} g(y,\theta)\lesssim 1$ on $\R^n\times\R^{n}\times V_\phi$;
		\item
			For every fixed $y\in\R^n, \eta\in V_\phi$ we have $\abs{\de_{\;y}\!g(y,\theta)}\rightarrow\infty$ as $\abs\xi\rightarrow\infty$;
		\item 
			The same assumptions 1.--4. hold true also for $f(x,\theta)$ with $y^i$ replaced by $x^i$.	
	\end{enumerate}
	Given an amplitude $a\in\sg{m_1,m_2,m_3}$ and a phase $\phi\in\Qcal(a)$, the Fourier Integral Operator associated with $a$ and $\phi$ is defined by \eqref{eq:Q FIO}, replacing the variables $\xi$ with $\theta$. We will use the notation $\Qcal_{gen}$ to speak about operators in the classes $\Qcal(a)$ for an arbitrary $a\in\sg{m_1,m_2,m_3}$, and we will refer to the variables $\xi$ in a splitting as above as the \textit{regular variables} of the phase.
\end{defin}

The next result is (a specialisation to our setting of) Theorem 8.5.1 of \cite{andrews2009SGFIOcomposition}.
\begin{theo}[Composition of $\Qcal_{gen}$-operators]
	\label{theo:composition of Q operators}
	Given $a\in\sg{m_1,m_2,m_3},b\in\sg{l_1,l_2,l_3}$, phases $\phi=f(x,\theta)+g(y,\theta)\in\Qcal(a),\eta=u(y,\kappa)+v(z,\kappa)\in\Qcal(b)$ and corresponding operators $A=FIO(\phi,a),B=FIO(\varphi,b)$, the composition $A\circ B$ is, modulo a compact operator on $\R^n$, a Fourier Integral Operator of type $\Qcal_{gen}$. In particular, for each $p,q\in\R$ such that $p+q=m_1+m_2+l_1+l_2$ we can find an amplitude $c\in\sg{p,q,m_3+l_3}$ and a phase $\Phi\in\Qcal(c)$ such that $A\circ B=FIO(\Phi,c)$. If $\theta=(\xi,\eta)\in\R^{n+d_1}$ and $\kappa=(\mi,\nu)\in\R^{n+d_2}$ for $\xi,\mi\in\R^n$, the phase $\Phi$ has the form $\Phi(x,z,\gamma)=f(x,\theta)+h(x,\theta,\kappa)+v(z,\kappa)$ where $\gamma=(\mi,\tilde{y},\theta,\nu)$ are $3n+d_1+d_2$ frequency variables with $\mi$ the regular ones in a splitting as in Definition \ref{def:generalised Q-class}. Moreover $f+h$ and $v$ are phase components and $h$ satisfies:
	\begin{equation}
		\label{eq:extra phase term composition Q operators}
		d_1=d_2,\; g(y,\xi)=-u(y,\xi)\implies h(x,\theta,\kappa)=0.
	\end{equation}
\end{theo}
\begin{cor}
	\label{cor:calculus of Q-operators}
	The class $\Qcal_{gen}$ satisfies:
	\begin{enumerate}
		\item 
			$\lg{}\circ\Qcal_{gen}\circ\lg{}\subset \Qcal_{gen}$.
		\item 
			If $A\in\Qcal_{gen}$ with phase $\phi=f+g$ and amplitude $a(x,y,\theta)$, then $A^\dagger\in\Qcal_{gen}$ with phase $\phi^\dagger=(-g)+(-f)$ and amplitude $a^\dagger(x,y,\theta)=a(y,x,\theta)$. In particular if $a(x,y,\theta)=a(y,x,\theta)$ then $A\in\Qcal(a)$ if and only if $A^\dagger\in\Qcal(a)$.
		\item 
			For each $a\in\sg{m_1,m_2,m_3}$ we have $\Qcal(a)\circ\Qcal(a)^\dagger\subset\lg{},\Qcal(a)^\dagger\circ\Qcal(a)\subset\lg{}.$
		\item 
			$A\in\Qcal_{gen}$ extends to a continuous operator $A\colon\Scal'\rightarrow\Scal'$.
	\end{enumerate}
\end{cor}

Notice that, in view of the last theorem and in particular of the properties of the function $h$, the statements in the corollary are a generalisation of the corresponding assertions in Proposition \ref{prop:properties of QFIOs}. The following specialised composition results for the class $\Qcal$, Theorem 7.2.1 in \cite{andrews2009SGFIOcomposition}, also follow at once from the composition theorem for $\Qcal_{gen}$.

\begin{theo}
	\label{theo:composition of QFIO}
	Let $A,B$ be $\Qcal$-FIOs with amplitudes $a,b$ (of arbitrary orders $m,l\in\R^3$) and phases $\phi=f+g,\mathfrak{p}=r+s$. 
	\begin{enumerate}
		\item If $g(y,\xi)=-r(y,\xi)$ then $AB\in\Qcal$ with phase $f+s$ and symbol $c\in\sg{p,q,m_3+l_3}$, where we can choose $p,q$ so that $p+q=m_1+m_2+l_1+l_2$.
		\item If in addition $f(x,\xi)=-s(x,\xi)$, then $AB$ is an $SG$-$\Psi$DO and we can choose again $p,q$ with $p+q=m_1+m_2+l_1+l_2$ so that we have an amplitude $\tilde{c}\in\sg{p,q,m_3+l_3}$.
	\end{enumerate}
	Moreover, the amplitudes of $AB$ in both cases admits an asymptotic expansion (cf. \cite{andrews2009SGFIOcomposition}, Proposition 4.0.4).
\end{theo}

\begin{rem}
	We have to remark that Andrews does not assume classicality for any of the operator classes he introduces. However, looking at the asymptotic expansions he obtains, it becomes clear that the assumptions of classicality (for both phases and amplitudes) and integer order for all operators are preserved by his composition formul\ae\, and therefore we obtain without fuss a sub-calculus modelled after $\lg{}$.
\end{rem}

\begin{rem}
	\label{rem:QFIO and Coriasco-FIO}
	The FIOs of type $\Qcal$ are a direct generalisation of the operator calculus introduced by Coriasco \cite{coriasco1999fourierintegraloperators}. This calculus corresponds to the subclass where we always take $g(y,\xi)=-y\xi$ (type I operators) or $f(x,\xi)=x\xi$ (type II operators). It is an important feature for us that the Egorov theorem for type I and II operators (namely, Proposition 14 in \cite{coriasco1999fourierintegraloperators}) extends to $\Qcal$. It suffices for this to look at the first term in the asymptotic expansion of Proposition 4.0.4 of \cite{andrews2009SGFIOcomposition} and use the same formal proof as given by Coriasco. At the same time, we remark that $\Qcal$-FIOs are given by arbitrary compositions of type I and type II operators. Indeed according to 1.\ in Theorem \ref{theo:composition of QFIO}, when we compose a type I operator with phase $\phi=f(x,\xi)-y\xi$ and a type II operator with phase $\mathfrak{p}=x\xi-g(y,\xi)$, we obtain the $\Qcal$-operator with phase $\mathfrak{q}=f+g$. Furthermore, applying this line of thought in reverse, we see that any $\Qcal$-operator with phase $\mathfrak{q}=f+g$ can be written as a composition of a type I operator with phase $\phi=f(x,\xi)-y\xi$ and a type II operator with phase $\mathfrak{p}=x\xi-g(x,\xi)$. This justifies the following definition as a direct generalisation of Definition 9 in \cite{coriasco1999fourierintegraloperators}.
\end{rem}
\begin{defin}
	\label{def:elliptic FIO}
	An FIO of type $\Qcal$ is called \textit{elliptic} if its amplitude is elliptic in the sense of Definition \ref{def:elliptic amplitudes}.
\end{defin}
\begin{prop}
	Let $A=FIO(\phi,a)$ be an elliptic FIO of type $\Qcal$. Then $A$ admits a parametrix $A^\sharp\in\Qcal$, namely an operator such that $AA^\sharp-I,A^\sharp A-I\in\RG$.
\end{prop}
\begin{theo}[Egorov's Theorem for $\Qcal$]
	\label{theo:SG Egorov}
	Let $A$ be an elliptic global $\Qcal$l-FIO with phase $f(x,\theta)+g(y,\theta)$ and amplitude $a\in\sg{0}$, and let $P\in\lg{m}$. Then $A^\sharp PA\in\lg{m}$ and $\sigma_{pr}(A^\sharp P A)=C^\ast\sigma_{pr}(P)$ where $C$ is the triple of homogeneous symplectic maps defined by the principal symbols of the phase function $f+g$. Namely $C$ is given as a triple $(C_e,C_\psi,C_{\psi e})$ with each map acting by pull-back on the respective component of the principal symbol and such that $\phi_\bullet$ is a phase function parametrising the graph of $C_\bullet$ as 
	\[
	\nabla_\theta\phi_\bullet=0\implies \graph C_\bullet=(x,\nabla_x\phi_\bullet,y,\nabla_y\phi_\bullet).
	\]
\end{theo}
\begin{proof}
    Write $A=BC$ for $B$ a type I operator with phase $f(x,\theta)+y\theta$ and amplitude $\sqrt{a}$ and $C$ a type II operator having phase $x,\theta+g(x,\theta)$  and amplitude $\sqrt{a}$ in the notation of Remark \ref{rem:QFIO and Coriasco-FIO}. Since $A$ is elliptic, both $B$ and $C$ are elliptic and admit parametrices $B^\sharp$ and $C^\sharp$, respectively of type II with phase $-x\theta-f(y,\theta)$ and of type I with phase $-g(x,\theta)-y\theta$. Moreover $APA^\sharp=BCPC^\sharp B^\sharp$, so that applying Proposition 14 in \cite{coriasco1999fourierintegraloperators} to $CPC^\sharp=Q$ gives that this composition is a $SG-\Psi$DO of order $m$. A second application of the same result to $BQB^\sharp$ gives the final claim.
\end{proof}

\section{Scattering geometry}
\label{chap:scattering geometry}
\setcounter{equation}{0}

\subsection{Manifolds with corners and scattering geometry}
We give hereafter a short account of basic definitions of manifolds with corners, smooth structures with corners and so on, adopting in essence the same conventions as in \cite{coriasco2019lagrangiandistributions}. A more detailed exposition can be found, for example, in \cite{melroseXXXXanalysismanifoldswithcorners}, while a comparison of the different existing notions can be found in \cite{joyce2010manifoldswithcorners}. Notice that, for the sake of simplicity and clarity, we prefer here an extrinsic approach.

\begin{defin}
	\label{def:chart with corners}
	A \textit{parametrised patch of dimension $d$, with corners of codimension $k$}, $0\leq k\leq d$, on a para-compact Hausdorff topological space $Z$, is a pair $(U,\phi)$ where $V\subset [0,\infty)^k\times\R^{d-k}$ is open and $\phi\colon U\rightarrow\phi(U)\subset Z$ is a homeomorphism. If we can choose $k=0$ then $(U,\phi)$ is just a parametrisation of an interior patch on $Z$ (namely, $\phi(U)\cap\del Z=\void$), while if we can have $k=1$ we say that $(U,\phi)$ is a parametrisation of a boundary patch. We say that a pair $(V,\psi)$ is a \textit{chart of dimension $d$, with corners of codimension $k$}, if $(U,\phi)\equiv(\phi(U),\psi^{-1})$ is a parametrised patch of dimension $d$, with corners of codimension $k$. We adopt the same terminology with respect to interior and boundary charts.
\end{defin}
\begin{defin}
	\label{def:manifold with corners}
	A \textit{$\Cinf$-manifold of dimension $d$, with corners of (maximal) codimension $k$} is a para-compact, second countable, Hausdorff topological space $Z$, together with a collection $\{(U_i,\phi_i)\}$ of parametrised patches of dimension $d$ and corners of codimension $k$, such that $\{\phi(U_i)\}$ covers $Z$ and, whenever $\phi_j(U_j)\cap\phi_i(U_i)\neq \void$, the changes of coordinates $\phi_{ij}=\phi_i^{-1}|_{\phi_j(U_j)\cap\phi_i(U_i)}\circ\phi_j\colon U_j\rightarrow U_i$ are smooth maps, in the sense that there exists a smooth map $\tilde{\phi}_{ij}\colon\tilde{U}_j\rightarrow\tilde{U}_i$, with open sets $\tilde{U}_i,\tilde{U}_j\subset\R^{k}\times \R^{d-k}$ containing $U_i,U_j$, respectively, satisfying $\tilde{\phi}_{ij}|_{U_j}=\phi_{ij}$. If, at every point $p\in Z$, we can find parametrised patches with $k=0$, then $Z$ is a \textit{smooth manifold}. Similarly, if all the patches can be picked with $k=1$, then $Z$ is a \textit{smooth manifold with boundary}.
\end{defin}
\begin{lemma}
	\label{lemma:manifold with corners}
	Let $Z$ be a $\Cinf$-manifold of dimension $d>0$, with corners of codimension $k$. There exists a smooth manifold $\tilde{Z}$ of dimension $d$, without boundary, such that $\mathring{Z}$ is open and non-empty in $\tilde{Z}$.
\end{lemma}

\begin{defin}
	\label{def:smooth structure}
	The space of \textit{$\Cinf$-functions} on $Z$ is the set $\Cinf(Z)$ consisting of all restrictions of smooth functions from $\tilde{Z}$ to $Z$. If we do not specify further, every geometric object (for example, vector bundles, differential of a smooth map, and so on) defined on $Z$ is obtained as the restriction of the corresponding concept from $\tilde{Z}$.
\end{defin}

\begin{conv}
	In what follows, we \textit{always} assume that $Z$ is compact and that there exist a \textit{finite} collection of smooth functions $\rho_i, i\in I,$ on $\tilde{Z}$ such that $Z=\{p\in \tilde{Z}\sthat \forall i\in I\;\rho_i(p)\geq 0\}$ and such that, whenever for a sub-collection $J\subset I$ it holds true $\rho_j(p)=0$ for all $j\in J$, then the differentials $\de\rho_j $ are linearly independent. Near points $p$, for which each patch containing $p$ has codimension $k>0$ corners, we \textit{always} use coordinates in the form $(\rho_{i_1},\dots,\rho_{i_k},z_{j_1},\dots,z_{j_{d-k}})$ for $z$ coordinates on the codimension $k$ corner in the patch $U$. 
\end{conv}

\begin{rem}
	The structure of the spaces defined by our axioms corresponds to the notion of \textit{manifold with embedded corners} of Joyce \cite{joyce2010manifoldswithcorners}. Accordingly, we can always pick a local boundary-defining function, and since corners are embedded sub-manifolds we can always find a global one associated with any boundary hyper-surface.
\end{rem}

\begin{lemma}
	\label{lemma:stratified boundary}
	Any $\Cinf$-manifold $Z$ of dimension $d$, with corners of codimension $k$, admits a stratification $\cup_{i=0}^{k}Z_i$, where $Z_i$ is a $\Cinf$-manifold of dimension $d-i$, with corners of codimension $k-i$. We call the union of the strata $Z_i$ for $i\geq 1$ the \textit{boundary} of the manifold with corners $Z$.
\end{lemma}
\begin{defin}
	\label{def:depth of a point}
	The \textit{depth} of a point $p\in Z$, $\depth(p)$, is the number of independent boundary-defining functions vanishing at $p$. Equivalently, it is the codimension of the boundary stratum $Z_i$ to which $p$ belongs.  
\end{defin}

\begin{rem}
	\label{rem:boundary embedded or not}
	Joyce makes a distinction between the boundary of $Z$, interpreted as a manifold with corners in its own right and admitting the stratification of Lemma \ref{lemma:stratified boundary}, and the embedded boundary of $Z$, which is in general only a topological manifold. In our definition we insist that the corners are embedded, and we must therefore adopt the second point of view. This forces us to define ``smooth'' functions on $\del Z$ as the restriction of a smooth function on $Z$ to $\del Z$, which does not agree in general with the concept of smooth function on the boundary interpreted according to Joyce's point of view. This is however only a minor inconvenience for our future purposes and we shall stick with this definition, more widespread in the context of singular analysis.
\end{rem}

\begin{rem}
	\label{rem:other singularities}
	Notice that our definition of manifolds with corners \textit{does not} include many other singular situations that have been considered in the literature. Indeed, on manifolds with corners, the function $\depth$ has the following property: if $\depth(p)=k$ and $0\leq s<k$, then any neighbourhood of $p$ (in the topology of $Z$) contains a point $q\neq p$ with $\depth q=s$. This is obviously untrue in other settings. For example, any closed cone (more generally, any manifold with a conical singularity) clearly doesn't have this property.
\end{rem}

\begin{defin}
	\label{def:interior sets and smooth functions}
	A relatively open set $U\subset Z$ is said to be \textit{interior} if $\bar{U}\cap\del Z=\void$. In case $\bar{U}\cap\del Z\neq 0$, we always assume that $\bar{U}\cap\del Z\subset U$, and we call $U$ either a \textit{boundary neighbourhood} or a \textit{corner neighbourhood}, depending on whether $U$ intersects only the stratum of codimension 1 or not. The set $\Cinf(U)$ of smooth functions on $U$ consists of all the restrictions of smooth functions on $Z$ to $U$. For $U$ a boundary neighbourhood, intersecting a corner of codimension $s$, the set $\rho_{i_1}^{m_1}\dots\rho_{i_s}^{m_s}\Cinf(U)$ consists of those functions $h\in\Cinf(\intern{U})$ such that $\rho_{i_1}^{-m_1}\dots\rho_{i_s}^{-m_s}h$ extends to a smooth function $\tilde{u}\in\Cinf(U)$. We have then a natural notion of smooth functions on a boundary hyper-surface, namely, the restriction of a function on a boundary neighbourhood $U$ to $U\cap\del Z$.
\end{defin}

\begin{rem}
	\label{rem:smooth functions across the corner}
	Notice that, in our setup, the boundary $\del Z$ is not itself a manifold with corners and does not carry a natural smooth structure. We obviate to this problem by choosing the following notion of smoothness. For a relatively open $V\subset \del Z$, intersecting the codimension 2 stratum $Z_2$ and no higher-codimensional stratum, the smooth functions $h\in\Cinf(V)$ are given by a pair of smooth functions $h=(f,g)$ on the two boundary hyper-surfaces such that $f|_{V\cap Z_2}=g|_{V\cap Z_2}$. The notion for higher-codimensional strata is defined accordingly. This notion of smoothness \textit{across the corner}, while in a certain sense arbitrary, fulfils the natural requirement that the function $f|_{Z_2}$ identifies with a smooth function \textit{on the corner}. We will see later that, for the scattering calculus of pseudo-differential operators, the principal symbols are identified with continuous function on the boundary of a certain manifold with corners, smooth across the corner according to this definition. Notice how this contrasts with Joyce's convention, which implies that the boundary of $Z$ carries a natural smooth structure given by considering it as a manifold with corners on its own right.
\end{rem}

\begin{defin}
	\label{def:smooth and vanishing at the boundary, extendible distributions}
	The space $\Cinfdot(Z)$ consists of those functions $f\colon Z\rightarrow \C$ such that $f$ and all it derivatives vanish at the boundary. The space of \textit{extendible distributions} on $Z$, $\Ecal'(Z)$, is the (topological) dual space of $\Cinfdot(Z,\Omega (Z))$, the sections of the density bundle having coefficients in $\Cinfdot(Z)$.
\end{defin}

On any manifold with corners there is a natural Lie sub-algebra $\VF_b(Z)$ of $\VF(Z)$, consisting of vector fields which are tangent to all boundary hyper-surfaces. Namely, on an interior neighbourhood $U$ we have $\VF_b(U)\cong\VF(U)$, while if $U$ is a boundary neighbourhood with corners of codimension $k$ then $\VF_b(U)$ is the Lie algebra generated, over $\Cinf(Z)$ and in standard local coordinates for a patch with codimension $k$ corners, by 
\begin{equation}
	\label{eq:generators of b-vector fields}
	\rho_{1}\del_{\rho_1},\dots,\rho_k\del_{\rho_k},\del_{x_1},\dots,\del_{x_{d-k}}.
\end{equation}
Equivalently, $V$ is a $b$-vector fields if $V\rho_i=\alpha_i\rho_i$ for any boundary-defining function on $Z$, where $\alpha_i$ are smooth functions. These have become known as \textit{b-vector fields} ($b$ for "boundary"). The dual $\Cinf(Z)$-module is the module of $b$-differential 1-forms ${}^b\Lambda^1(Z)$. Namely, it is the module generated, locally near the boundary, by 
\begin{equation}
	\label{eq:generators of b-differential 1-forms}
	\frac{\de\rho_1}{\rho_1},\dots, \frac{\de\rho_k}{\rho_k},{\de x^1},\dots,\de x^{d-k}.
\end{equation} 
There is an obvious perfect duality ${}^b\Lambda^1(Z)\times\VF_b(Z)\rightarrow \C$. Namely, this pairing identifies ${}^b\Lambda^1(Z)$ with the dual of $\VF_b(Z)$ and vice-versa.
\begin{lemma}
	\label{lemma:b-vector fields and 1-forms}
	\begin{enumerate}
		\item 
			There exist vector bundles ${}^bTZ\rightarrow Z$ and ${}^bT^\ast Z$ such that $\VF_b(Z)$ and ${}^b\Lambda^1(Z)$ are respectively the spaces of smooth sections of ${}^bTZ$ and ${}^bT^\ast Z$ over $Z$.  
		\item 
			If $U\subset Z$ is interior, then $\VF_b(U)\cong\VF(U)$ and ${}^b\Lambda^1(U)\cong\Lambda^1(Z)$.
		\item 
			There are natural vector bundle maps ${}^bTZ\rightarrow TZ$ and $T^\ast Z\rightarrow {}^bT^\ast Z$, dual to each other, which are isomorphisms over any interior neighbourhood.
		\item 
			The $b$-vector fields on a manifold with boundary $(Z,\rho)$ are identified, in a collared neighbourhood of $\del Z$, with the sections of $TZ$ having \textit{bounded length} with respect to the \textit{exact $b$-metric}
			\begin{equation}
				\label{eq:cylindrical metric}
				g=\frac{\de \rho^2}{\rho^2}+h,
			\end{equation}
			that is, $g(V,V)<\infty$ for any $b$-vector field $V$. Here $h$ is the pull-back of a metric $h_{\del Z}$ on the embedded boundary $\del Z$ to the collared neighbourhood $\del Z\times[0,1)$.
	\end{enumerate}
\end{lemma}

\begin{rem}
	\label{rem:algebroids}
	Lemma \ref{lemma:b-vector fields and 1-forms} could be reformulated in the language of Lie algebroids. In particular, $\VF_b(Z)$ is a Lie algebroid with anchor map given by the natural bundle map of item 3. 
\end{rem}

Having concluded our brief recap on manifolds with corners, we recall the notion of \textit{scattering structure} of Melrose, which is known to yield a pseudo-differential calculus equivalent on $\R^d$ to the classical $SG$-calculus. This will be stated precisely later on.

\begin{defin}
	\label{def:scattering manifold}
	A topological space $X$ is called a \textit{scattering manifold} if $X$ is a compact manifold with boundary, with boundary defining function $\rho$, equipped with a Riemannian metric $g$ which in a collared neighbourhood of the boundary takes the form
	\begin{equation}
		\label{eq:scattering metric}
		g=\frac{\de \rho^2}{\rho^4}+\frac{h}{\rho^2}.
	\end{equation}
	In \eqref{eq:scattering metric} $h$ is a symmetric, 2-covariant tensor field containing no $\de\rho$ factors. Namely it is the pull-back of metric on $\del X$ to the collared neighbourhood $\del X\times [0,\eps)$.
\end{defin}

On a scattering manifold $X$, we have obviously a notion of $b$-vector fields given by the geometric structure. However, the scattering metric on $X$ is quite different from a $b$-metric. The structure of the manifolds near the boundary (``at infinity'') can be identified either with a cone (scattering) or with a cylinder ($b$-metric). Correspondingly, there is another Lie algebra of vector fields which describes the scattering structure. These so-called scattering vector fields are the elements of 
\begin{equation}
	\label{eq:scattering vector fields}
	\VF_{sc}(X)\equiv\rho\VF_b(X),
\end{equation}
that is, they are sections of $TX$ which are tangent to the boundary and have \textit{bounded length} w.r.t.\ $g$, i.e. $g(V,V)<+\infty$ for $V\in\VF_{sc}(X)$. They are generated (near the boundary $\rho=0$, parametrized by coordinates $y$) by 
\begin{equation}
	\label{eq:scattering vactor fields generators}
	\rho^2\del_\rho,\quad\rho\del_y.
\end{equation}
Furthermore, they are the sections of a vector bundle over $X$ called the \textit{scattering tangent bundle}, $\sc{T X}$. 

\begin{rem}
	\label{rem:algebroids 2}
	The process of constructing these Lie sub-algebras of $\VF(X)$, adapted to the geometric situation of interest, can be described in a much more general framework by the process of \textit{rescaling} of Lie algebroids. Building on ideas of Melrose (who first dealt with the rescaling of Lie sub-algebras of $\VF(X)$) and Scott (see \cite{scott2016geometryb^kmanifolds}), Lanius \cite{lanius2021symplecticpoissonscattering} described the process for a general Lie algebroid and initiated the study of scattering-symplectic manifolds, at the same time exploring the Poisson-geometric side of the matter. In this picture, the scattering algebroid $\sc{TX}$ is exactly the rescaling of the $b$-algebroid ${}^bTX$ along the algebroid of so-called $0$-vector fields of Mazzeo and Melrose, namely those vector fields which vanish at the boundary. However, at the boundary the $b$- and $0$-calculus are highly non-trivial, in the sense that there is a non-commutative algebra of ``indicial operators'' which need to be inverted when considering ellipticity. We will see that the situation for the scattering structure is much nicer.
\end{rem}

\begin{eg}
	\label{example:compactification on Rn}
	We can turn $\R^n$ into a scattering manifold by considering the \textit{radial compactification}\footnote{Different schools in the literature use different names for the object here described. The continental school tends to stick to the name \textit{stereographic compactification} for the map $R$ we are about to introduce, and reserves the term \textit{radial compactification} for the construction of the upcoming Remark \ref{rem:the other compactification 1}. Overseas, this last concept takes the name \textit{quadratic compactification}.}. It is obtained from the stereographic projection as follows. Consider $\Sfp$, the upper closed half-sphere of radius $1$ in $\R^{n+1}$ with coordinates $(x^1,\dots, x^{n+1})$, and identify $\R^n$ with the hyperplane $x^{n+1}=1$ in $\R^{n+1}$. A point $p\in\R^n$ is mapped bijectively to $q\in \mathring{\mathbb{S}}^n_{+}$ by taking the line $l_p$ joining $p$ to the origin and setting $q=$intersection of $l_p$ with $\Sfp$. Let us denote by $R$ this embedding. Then we ``add the points at $\infty$'' to $\R^n$ by embedding $\Sf{n-1}$ as the boundary of $\Sfp=\{(x^0,\dots,x^{n})\in\Sf{n}\sthat x^{n+1}\geq 0\}$ in the radially compactified picture. The lingo is justified by the fact that, approaching $\Sf{n-1}$ along a (half of a) maximal circle of $\Sf{n}$, we are in fact going to infinity along the corresponding ray in $\R^n$. We introduce coordinates near the boundary of $\Sfp$ as follows. Describe $\R^n$, at least outside a compact neighbourhood of $0$, using polar coordinates $(r,y)$ with $y$ angular coordinates (that is, coordinates on $\Sf{n-1}\subset\R^n$). Using $R$ we map this description to coordinates on the open half-sphere $\mathring{\mathbb{S}}^n_{+}$, and take $\rho\equiv1/r$. One shows easily that $\rho$ is a boundary-defining function and that the pull-back via $R$ of the Euclidean metric to $\Sfp$ produces a metric of the form \eqref{eq:scattering metric}, with $h=$ standard metric on $\Sf{n-1}$ embedded in $\R^n$. We also remark that the only (eventually) homogeneous functions on $\R^n$ which extend to a smooth function to $\Sfp$ are those of non-positive order. Specifically, those of order $0$ extend to the boundary with their radial limit while those of negative order take value $0$ at the boundary.
\end{eg}
\begin{rem}
	\label{rem:the other compactification 1}
	In the previous example, one might wonder why we don't consider, as a compactified space, the projection of $\Sfp$ onto the plane $x^{n+1}=1$, namely a closed ball $\B^n$ of radius $1$ in $\R^n$. Notice that, with such a choice, the projection map, certainly bijective and smooth $\Sfp\rightarrow \B^n$, does not have a smooth inverse, since this has a square-root singularity. On the other hand, stereographic projection from $(0,\dots,0,-1)$ onto $x^{n+1}=1$, restricted to the upper closed half-sphere, gives a diffeomorphism from $\Sfp$ to $\B^n$, so we can understand this from both points of view, if only with the need to make the correct identifications.  Notice, in addition, that, in the literature with an $SG$ approach, one often uses a different boundary defining function, namely one takes a diffeomorphism $Q$ of $\R^n$ onto the open ball $B_1(0)$, given for $\abs{x}>3$ by $Q(x)=\frac{x}{\abs{x}}\left(1-\frac{1}{\abs{x}}\right)$. For $[x]$ any smooth function such that $[x]=\abs{x}$ for $\abs{x}>3$, we obtain that $(Q^{-1})^\ast[x]$ is a boundary defining function. It can be checked directly that it is equivalent to $(R^{-1})^\ast \braket{x}$. Namely that, in sufficiently small neighbourhoods of the boundary, they are just a multiple of each other by a positive smooth function. It follows that the two approaches are really equivalent. A third approach, yet again equivalent, would be to map a point $x$ to $\frac{x}{\braket{x}}\in\B_1(0)$ and applying the same process as before. This would result in another choice of ``standard'' boundary-defining function. We will mainly stick to $\Sfp$ and $(R^{-1})^\ast\braket{x}$ for conceptual purposes. However we will at times switch to a different picture for convenience of notation.
\end{rem}

With any scattering manifold, as we have seen, is associated a rescaling of the tangent bundle. The dual construction applied to $T^\ast X$ yields the \textit{scattering cotangent bundle} $\sc{T^\ast X}$. Namely, it is the bundle whose sections are the rescaled 1-forms $\sc{\Lambda^1}(X)$, generated, as a $\Cinf(X)$-module near the boundary, by
\begin{equation}
	\label{eq:scattering cotangent bundle generators}
	\frac{\de\rho}{\rho^2},\quad \frac{\de y}{\rho}.
\end{equation}

In the scattering approach, it turns out that it's quite convenient to consider a compactified version of this space. Namely, given $X$ a scattering manifold and $\sc{T^\ast X}$ its scattering cotangent bundle, we compactify each fibre from $\R^n$ to $\Sfp$ with the map $R$ and consider the total space so obtained, which we denote by $\sc{\bar{T}^\ast X}$. This is now a manifold with corners. Indeed, we have two boundary defining functions $\rho_e$ and $\rho_\psi$, respectively, for the boundary $\del X$ and the boundary of the half-spheres in the compactification of the fibres. The common zero locus of these functions, i.e. the space $\rho_e=\rho_\psi=0$, is a codimension 2 corner. 

\begin{eg}
	\label{example:compactification of Rn 2}
	In the example of $X=\Sfp$, i.e.\ of the compactification of $\R^n$, we have that $\sc{T^\ast X}$ is trivial, so the compactification process in the fibres yields the manifold $\Sfp\times\Sfp$. The boundary of this manifold is traditionally called the \textit{$SG$-wave-front space} (cfr. \cite{coriasco2003wavefrontsetatinfinity} and \cite{cordes1995techniquepseudodifferentialoperators}) and can be decomposed as
	\begin{equation}
		\label{eq:boundary components scattering cotangent bundle}
		\del(\Sfp\times\Sfp)=(\Sf{n-1}\times \R^n)\;\dot\cup\; (\R^n\times\Sf{n-1})\;\dot\cup\;(\Sf{n-1}\times\Sf{n-1}).
	\end{equation}
	For later reference we denote the three pieces, respectively, by $\widetilde{\Wcal}_e,\widetilde{\Wcal}_\psi,\widetilde{\Wcal}_{\psi e}$.
\end{eg}

The \textit{scattering differential operators} on $X$ are the elements of $\Diff_{sc}(X)$, the $\Cinf(X)$-enveloping algebra of $\VF_{sc}(X)$. That is to say, $\Diff_{sc}(X)$ is the filtered algebra generated, on a boundary neighbourhood $U$, by $\rho^2 \del_\rho,\rho \del_y,1$ over $\Cinf(U)$, and isomorphic to $\Diff(U)$, the usual differential operators, if $U$ is an interior neighbourhood. There is a well-defined (principal) symbol map $\sigma_{sc}$ on scattering operators defined as follows. For a scattering vector field $V$, consider it as a section of $\sc{TX}$. At each point $p$, we can identify $V(p)$ with a linear map on the fibres of the dual bundle (since a finite-dimensional vector space is canonically isomorphic with its bi-dual), so $V(p)\colon \sc{T^\ast_p X}\rightarrow \C$, and obtain a smooth function on $\sc{T^\ast X}$. Set then $\sigma_{sc,1}(V)\equiv \im V$ and extend it multiplicatively to the whole $\Diff_{sc}^m(X)$ to a map $\sigma_{sc,m}$, taking values in $\Pol^{(m)}(\sc{T^\ast X})$ (as before, round brackets mean homogeneity). This gives the usual short exact sequence
\begin{equation}
	\label{eq:scattering symbol sequence}
	0\rightarrow \Diff_{sc}^{m-1}(X)\rightarrow\Diff_{sc}^m(X) \xrightarrow{\sigma_{sc,m}}\Pol^{(m)}(\sc{T^\ast X})\rightarrow0.
\end{equation}
Moreover, in view of the homogeneity of the polynomials in \eqref{eq:scattering symbol sequence}, we can identify the principal symbol$\sigma_{sc,m}(P)$ with a smooth function on $\sc{\Sf{\ast}X}$, the \textit{scattering co-sphere bundle} of $X$. This is just the sub-bundle of $\sc{T^\ast X}$ with fibre the sphere of radius 1 with respect to the inverse of the metric \eqref{eq:scattering metric}.

The main difference with the usual differential operators is the fact that invertibility of $\sigma_{sc,m}(P)$ does not guarantee the existence of a ``good'' parametrix (one with compact remainder). This is due to the fact that the coefficients of an ``elliptic'' operator might not have good growth/decay properties as $\abs{x}\rightarrow\infty$. There is, on the other hand, a way to take this behaviour into account, which we describe hereafter. For scattering vector fields, the Lie bracket satisfies
\begin{equation}
	\label{eq:commutator of scattering vector fields}
	[\VF_{sc}(X),\VF_{sc}(X)]\subset \rho\VF_{sc}(X),
\end{equation}
so that for each point $p\in\del X$ the evaluation map defines a Lie algebra homomorphism into a trivial (namely, commutative) Lie algebra
\begin{equation}
	\label{eq:normal homomorphism}
	N_{sc,p}\colon\VF_{sc}(X)\rightarrow \sc{T_pX}.
\end{equation}
Functions can be evaluated at a point, too, and the two evaluations are compatible (that is, $(fV)(p)=f(p)V(p)$), so we have a unique multiplicative extension to $\Diff_{sc}(X)$ with values in translation-invariant (namely, constant coefficients) differential operators on $\sc{T_pX}$. On a vector space, the Fourier transform identifies these with (non-homogeneous) polynomial functions, so that, at each point $p\in\del X$, we obtain a map
\begin{equation}
	\label{eq:scattering normal operator}
	\ft{N}_{sc,p}\colon\Diff_{sc}^m(X)\rightarrow \Pol^m(\sc{T^\ast_pX}).
\end{equation}

This is known as the \textit{normal symbol} or \textit{normal operator} and gives another short exact sequence,
\begin{equation}
	\label{eq:normal symbol sequence}
	0\rightarrow\rho\Diff_{sc}^m\rightarrow\Diff_{sc}^m\xrightarrow{\ft{N}_{sc,p}}\Pol^m(\sc{T^\ast_p X})\rightarrow0.
\end{equation}

Notice that the only relation between the symbol and the normal operator is that evaluation of the symbol at a boundary point should equal the leading term of the normal operator at that point (compare with the $SG$-principal symbol). Namely, at $p\in\del X$ it holds true
\begin{equation}
	\label{eq:scattering compatibility condition}
	\sigma_{sc,m}(P)|_p-\ft{N}_{sc,p}(P)\in\Pol^{m-1}(\sc{T^\ast_p X}).
\end{equation}

The two symbol maps are combined in the so-called \textit{joint symbol} map
\begin{equation}
	\label{eq:joint symbol}
	j_{sc,m}(P)\equiv(\sigma_{sc,m}(P),\ft{N}_{sc}(P))\in\sc{\widetilde{\Pol}^{m,0}}(X),
\end{equation}
where $\sc{\widetilde{\Pol}^{m,0}}(X)$ is the space of all pairs of functions $(q,\ft{N})$ with $q\in\Pol^{(m)}(\sc{T^\ast X})$, $\ft{N}\in\Pol^m(\sc{T^\ast_{\del X}X})$ and such that $\ft{N}-p|_{\del X}\in\Pol^{m-1}(\sc{T^\ast_{\del X}X})$. There is then a combined short exact sequence:
\begin{equation}
	\label{eq:scattering combined sequence}
	0\rightarrow \rho\Diff_{sc}^{m-1}(X)\rightarrow\Diff^m_{sc}(X)\xrightarrow{j_{sc,m}}\sc{\widetilde{\Pol}^{m,0}}(X)\rightarrow 0.
\end{equation}

\begin{lemma}
	\label{lemma:principal symbol scattering}
	The space $\sc{\widetilde{\Pol}^{m,0}}(X)$ can be canonically identified with a subalgebra of $\rho_{\sigma}^{-m} \Cinf(\del(\sc{\bar{T}^\ast X}))$. 
\end{lemma}
\begin{proof}
	A function $f$ is an element of $\Cinf(\del(\sc{\bar{T}^\ast X}))$ if it is given as $f=(f_N,f_\sigma)$ for two smooth functions $f_N\in\Cinf(\sc{T^\ast_{\del X} X})$, $f_\sigma\in\Cinf(\sc{\Sf{\ast}}X)$ satisfying \eqref{eq:scattering compatibility condition}. Clearly the normal symbol $\ft{N}_{sc}$ is such a $f_N$. On the other hand, identifying the boundary of the fibre-wise compactification with the co-sphere bundle $\Sf{\ast}X$, the function $\sigma_{sc,m}$ is determined by homogeneity by an element $f_\sigma\in\Cinf(\sc{\Sf{\ast}X})$. For a $P\in\Diff_{sc}^m(X)$, then we obtain a pair of functions as above, The proof is complete.
\end{proof}

To define pseudo-differential operators on a scattering manifold, we start with the model case of $\Sfp$. The Weyl calculus of Hörmander with respect to the temperate metric $g=\braket{x}^{-2}\de x^2+\braket{\xi}^{-2}\de\xi^2$ gives a class of operators on $\Scal(\R^n)$ having distributional kernels given by 
\begin{equation}
	\label{eq:scattering operators schwartz kernel}
	K(x,y)=\int e^{\im(x-y)\xi}p_L(x,\xi)\dbar \xi,
\end{equation}
where the function $p_L(x,\xi)$ is the \textit{left-symbol} of the operator $P$. Then, the function $p_L$ satisfies the estimates \eqref{eq:sgestimates}. Namely, $p_L$ is a symbol with respect to the above metric, the standard symplectic form and the order/weight function $\braket{x}^{l}\braket{\xi}^m$. Using the stereographic projection $R$ (recall the definition in Example \ref{example:compactification on Rn}) we can transfer these to $\Sfp$. Set $\Cinfdot(\Sfp)$ to be the space of all smooth functions on $\Sfp$ which vanish at the boundary together with all their derivatives. 
\begin{defin}
	\label{def:scattering pseudo-differential operator}
	The space of \textit{scattering-conormal pseudo-differential operators} on $\Sfp$, $\Psi^{l,m}_{scc}(\Sfp)$, is the set of all the linear operators $A\colon\Cinfdot(\Sfp)\rightarrow\Cinfdot(\Sfp)$ such that, if $P$ is defined by $R^\ast(Au)=P (R^\ast u)$ for all $u\in\Cinfdot(\Sfp)$, then $P$ is given as an operator with Schwartz kernel as in \eqref{eq:scattering operators schwartz kernel}, with a left symbol $p_L$ of order $(l,m)$.
\end{defin}
We let $R_2\equiv R\times R\colon\R^n\times\R^n\rightarrow\Sfp\times\Sfp$ be separate radial compactification in each factor and choose boundary defining functions $\rho_\sigma\colon\Sfp\times\Sf{n-1}\times [0,1)\rightarrow \R$, $\rho_N\colon\Sf{n-1}\times[0,1)\times\Sfp\rightarrow \R$ for the two boundary hyper-surfaces (for example, $R^\ast\rho_\sigma=\braket{\xi}^{-1},R^\ast\rho_N=\braket{x}^{-1}).$ Let $\Diff_{b}(\Sfp\times\Sfp)$ be the enveloping algebra (over $\Cinf(\Sfp\times\Sfp)$) of $\VF_b(\Sfp\times\Sfp)$, the so-called \textit{totally characteristic} or \textit{b-differential operators}. We define a space of distributions, conormal to boundary in the sense of Hörmander, of order $(l,m)$ as
\begin{equation}
	\label{eq:scattering conormal distributions}
	\begin{aligned}
		I^{l,m}(\Sfp\times\Sfp)\equiv\{u&\in\rho_N^{-l}\rho_\sigma^{-m}\leb{\infty}(\Sfp\times\Sfp)\\
		&\sthat\Diff_{b}(\Sfp\times\Sfp)u\subset\rho_N^{-l}\rho_\sigma^{-m}\leb{\infty}(\Sfp\times\Sfp)\}.	
	\end{aligned}
\end{equation}
This defines a global space of kernels whose microlocal representation is given by oscillatory integrals as in \eqref{eq:scattering operators schwartz kernel}. Indeed, it is easily seen that $p_L$ satisfies \eqref{eq:sgestimates} if and only if $p_L\in R_2^\ast I^{l,m}(\Sfp\times\Sfp)$ with $m_e=l,m_\psi=m$ (confer \cite{melrose1994spectralscatteringtheory}, Section 4).

\begin{rem}
	\label{rem:b and 0 normal symbol}
	For $b$- and $0$-differential operators, one has a well-defined normal operator $N_{p}$ at the boundary. However, the condition \eqref{eq:commutator of scattering vector fields} fails. Indeed, no extra vanishing factors appear when commuting elements of the forms $\rho\del_\rho,\rho\del_y,\del_y$, so the normal homomorphism takes values in a non-commutative algebra. This is the reason why those structures are much more complicated from an analytical perspective.
\end{rem}

To obtain classical operators, we refine Definition \ref{def:scattering pseudo-differential operator} by asking that the left-symbol is actually a (weighted) smooth function on $\Sfp\times\Sfp$.
\begin{defin}
	\label{def:classical scattering operators}
	The space of \textit{classical scattering pseudo-differential operators} $\Psi^{l,m}_{sc}(\Sfp)$ is the subspace of $\Psi^{l,m}_{scc}(\Sfp)$ consisting of those operators with
	\begin{equation}
		p_L\in R_2^\ast (\rho_N^{-l}\rho_\sigma^{-m}\Cinf(\Sfp\times\Sfp)).
	\end{equation}
	Despite a lot of what follows being true \textit{mutatis mutandis} also for the larger class $\Psi_{scc}$, in the sequel we will consider classical operators, avoiding repeated explicit mention.
\end{defin}

For the sake of completeness, we include a definition of scattering $\Psi$DOs on a general scattering manifold. Although this will not be needed in the sequel (we will only concern ourselves with classical scattering operators on $\Sfp\times\Sfp$), it reveals that the more ``global'' nature of the scattering calculus transfers more easily to general settings in comparison with the $SG$-calculus. On the other hand, however, we recall that the class of $SG$-manifolds as defined by Schrohe \cite{schrohe1987weightedsobolevspacesmanifolds} (or even just the class of $\Scal$-manifolds in the sense of Cordes \cite{cordes1995techniquepseudodifferentialoperators}) is significantly larger. The following lemma expresses coordinate invariance and is a direct consequence of the calculus of \cite{melrose1994spectralscatteringtheory}.

\begin{lemma}
	\label{lemma:scattering class is coordinate invariant}
	Let $F\colon\Sfp\rightarrow\Sfp$ be a diffeomorphism. Then for any $P\in\Psi_{sc}^{l,m}(\Sfp)$ we have $F_\ast P F^\ast\in\Psi_{sc}^{l,m}(\Sfp)$, namely, conjugation with a diffeomorphism defines an order-preserving automorphism of $\Psi_{sc}^{l,m}(\Sfp)$.
\end{lemma}

For a scattering manifold $X$, an operator $P\colon\Cinfdot(X)\rightarrow\Cinfdot(X)$ has a kernel which in general is an extendible distribution $K_P\in\Ecal'(X^2,\pi_R^\ast\Omega)$, where $\pi_R$ is the projection onto the second factor and $\Omega$ is the density bundle on $X$. We define regularising $\Psi$DOs as exactly those integral operators with kernel in $\Cinfdot(X^2,\pi_R^\ast\Omega)$. We notice that $R^\ast\Cinfdot(\Sfp)=\Scal(\R^n)$. We can then introduce the calculus on $X$ along the lines of the usual definition by localisation for manifolds without boundary, simply replacing any instance of `manifold' with `manifold with corners', `open set in $\R^n$' with `open set in $[0,\infty)^k\times\R^{n-k}$' and so on (effectively, we are giving the same definition as Definition 18.1.20 in \cite{hormander1994analysispseudodifferential} in the new category of `manifolds with corners', modelling our operators on $\Psi^{l,m}_{scc}(\Sfp)$). 

We collect some of the properties of this algebra before turning our attention to the symbol calculus for scattering operators.

\begin{prop}
	\label{prop:properties of scattering calculus}
	The following holds true.
	\begin{enumerate}
		\item 
			The spaces $\Psi_{sc}^{l,m}(X)$ sit in a partial order where $\Psi_{sc}^{l,m}(X)\subset\Psi_{sc}^{l',m'}$ if and only if $l\leq l'$ and $m\leq m'$; they form a bi-filtered algebra $\Psi_{sc}(X)$ under composition.
		\item 
			$\Diff_{sc}^m(X)\subset\Psi_{sc}^{0,m}(X)$. 
		\item 
			Multiplication with a boundary-defining function defines an order reduction for the first filtration, namely $\Lambda_N=M_\rho$ is a classical, invertible, scattering operator of order $\indi_e$.
		\item 
			The operator $\Lambda_\sigma=\sqrt{1-\Delta_{sc}}$, for $\Delta_{sc}$ the Laplace operator associated with the metric \eqref{eq:scattering metric}, is an order reduction for the second filtration, namely $\Lambda_\psi$ is a classical, invertible, scattering operator of order $\indi_\psi$.
	\end{enumerate}
\end{prop}

Recall that for scattering differential operators the principal symbol is a continuous function on the boundary of $\sc{\bar{T}^\ast X}$, smooth across the corner in the sense of Remark \ref{rem:smooth functions across the corner}. We let $B_{sc}X\equiv\del(\sc{\bar{T}^\ast X})=\sc{\bar{T}^\ast_{\del X}X}\cup_{\sc{\Sf\ast_{\del X}X}}\sc{\Sf{\ast}X}$ and denote by $\Cinf(B_{sc}X)$ the set of smooth functions according to this definition. If we are given vector bundles $E_N,E_\sigma$ over $\sc{\bar{T}^\ast_{\del X}X},\sc{\Sf{\ast}X}$ respectively, with a specified identification of their restrictions to $\sc{\Sf\ast_{\del X}X}$, then we can also consider 
\begin{multline}
	\label{eq:smooth sections on CscX}
	\Cinf(B_{sc}X;(E_N,E_\sigma))=\{(u_N,u_\sigma)\in\Cinf(\bar{\sc{T^\ast_{\del X}X}};E_N)\times\Cinf(\sc{\Sf{\ast}X};E_\sigma)\\
	\sthat u_N|_{\sc{\Sf{\ast}_{\del X}X}}=u_\sigma|_{\sc{\Sf{\ast}_{\del X}X}}\}.
\end{multline}
Notice in particular that this is the case if we are given a vector bundle over the whole of $\sc{\bar{T}^\ast X}$.

\begin{lemma}
	\label{lemma:bundle of symbol sections}
	The elements of $\rho_N^{-l}\rho_\sigma^{-m}\Cinf(\sc{\bar{T}^\ast X})$ are the sections of a trivial bundle $S^{l,m}$ over $\sc{\bar{T}^\ast X}$, equipped with a $b$-connection. Namely, for every $V\in\VF_b(X)$ and every section $a$ of $S^{l,m}$, it holds true $Va\in S^{l,m}$.
\end{lemma}

With this notation set, we can finally express the principal symbol sequence.

\begin{prop}
	\label{prop:scattering principal symbol}
	The maps $\ft{N}_{sc}$ and $\sigma_{sc}$ extend from $\Diff_{sc}(X)$ to $\Psi_{sc}(X)$ to give the \textit{scattering joint symbol} $j_{sc,l,m}\colon\Psi_{sc}^{l,m}(X)\rightarrow\Cinf(B_{sc}X;S^{l,m})$, and we have the exact sequence
	\begin{equation}
		\label{eq:scattering principal symbol sequence}
		0\rightarrow\Psi_{sc}^{l-1,m-1}(X)\rightarrow\Psi_{sc}^{l,m}(X)\xrightarrow{j_{sc,l,m}}\Cinf(B_{sc}X;S^{l,m})\rightarrow 0.
	\end{equation}
	Furthermore, the joint symbol is multiplicative. Namely, for any $A\in\Psi_{sc}^{l_1,m_1}(X),B\in\Psi_{sc}^{l_2,m_2}(X)$, it holds true
	\begin{equation}
		\label{eq:joint symbol multiplicative}
		j_{sc,l_1+l_2,m_1+m_2}(AB)=j_{sc,l_1,m_1}(A)j_{sc,l_2,m_2}(B),
	\end{equation}
	with the product given component-wise.
\end{prop}

Before moving towards the discussion of the symplectic structure on $B_{sc}X$, we notice some special properties of the model case $X=\Sfp$. The first is that conjugation with the Fourier transform gives an automorphism of pseudo-differential operators, a peculiar feature of this setting. We report a proof of this fact, since it is instructive about the nice properties of $SG$ and $sc$-calculi.

\begin{prop}
	\label{prop:sc-operators conjugate with FT}
		Let $\FT\colon\Scal(\R^n)\rightarrow\Scal(\R^n)$ be the Fourier transformation and consider the map $\bar{\FT}\colon\Cinfdot(\Sfp)\rightarrow\Cinfdot(\Sfp)$ given by $\bar\FT\equiv (R^\ast)^{-1}\circ\FT\circ R^\ast$. Then
	\begin{equation}
		\label{eq:sc-operators conjugate with FT}
		\bar\FT\circ\Psi^{l,m}_{sc}(\Sfp)\circ\bar{\FT}^{-1}=\Psi^{m,l}_{sc}(\Sfp).
	\end{equation}
\end{prop}
\begin{proof}
	The action of $P\in\Psi_{sc}^{l,m}$ is expressed ``locally'' as $\mathfrak{P}(v)=R^\ast(Pu)$ for each $u\in\Cinfdot(\Sfp)$ and $v=R^\ast u\in\Scal(\R^n)$ (recall that $R^\ast\Cinfdot(\Sfp)=\Scal(\R^n)$), in particular
	\begin{equation}
		\label{eq:local form scattering psdos}
		\begin{aligned}
		    \mathfrak{P}v&=\int e^{\im x\xi}p_L(x,\xi)\ft{v}(\xi)\dbar\xi,\\
		    \mathfrak{P}\ft{v}(\xi)&=\int e^{\im (\xi-\eta)z}p_L(\xi,z)\ft{v}(\eta)\dbar z\de\eta.
		\end{aligned}
	\end{equation}
	Consider $\widetilde{\mathfrak{P}}=\FT^{-1}\circ \mathfrak{P}\circ\FT$. This is the ``local representation'' of $\bar\FT\circ P\circ\bar{\FT}^{-1}$, since clearly
	\begin{equation}
		\label{eq:local form conjugated psdo}
		\begin{aligned}
			\widetilde{\mathfrak{P}}v&=(\FT^{-1}\circ \mathfrak{P}\circ\FT)( R^\ast u)=\FT^{-1}\circ \mathfrak{P}\circ R^\ast(\bar\FT u)\\
			&=R^\ast(\bar{\FT}^{-1}\circ (R^\ast)^{-1}\circ \mathfrak{P}\circ R^\ast\circ\bar\FT u)\\
			&=R^\ast(\bar\FT^{-1}\circ P\circ\bar\FT u).
		\end{aligned}
	\end{equation}
	Computing the action of $\widetilde{\mathfrak{P}}$, we observe
	\begin{equation}
		\label{eq:action of conjugated psdo}
		\begin{aligned}
			\widetilde{\mathfrak{P}}v(x)&=\FT^{-1}(\mathfrak{P}\ft{v})(x)=\int e^{\im x\xi}e^{\im(\xi-\eta)z}p_L(\xi,z)\ft{v}(\eta)\dbar z\de\eta\dbar\xi\\
			&=\int \mathrm{e}^{\im(x+z)\xi}p_L(\xi,z)\left(\int e^{-\im\eta z}\ft{v}(\eta)\de\eta\right)\dbar z\dbar\xi\\
			&=(2\pi)^n\int e^{\im(x+z)\xi}p_L(\xi,z)v(-z)\dbar z\dbar\xi\\
			&=\int e^{\im(x-y)\xi}p_L(\xi,-y)v(y)\de y\dbar\xi.
		\end{aligned}
	\end{equation}
	Hence, the claim is in fact just the equivalence of the classes of left- and right-quantised $SG$-operators, namely Lemma \ref{lemma:operators with amplitudes or symbols}.
\end{proof}

The second aspect relates to the equivalence of the classical $sc$- and $SG$-calculi, which is especially manifest in the model case, as the next theorem shows (cf. \cite{egorov1997pseudodifferentialsingularities}, Section 8.2.2, for a proof).
\begin{theo}
	\label{theo:scattering equals classical SG}
	The following properties hold true.
	\begin{enumerate}
	\item 
		For any $a\in\sg{m},m=(m_e,m_\psi)\in\R^2$, the function $\braket{x}^{-m_e}\braket{\xi}^{-m_\psi}a(x,\xi)$ extends smoothly to $\Sfp\times\Sfp$.
	\item 
		For any $m\in\R^2$ there exist an isomorphism 
		\begin{equation}
			\label{eq:SG isomorphism scattering compactification}
			\begin{aligned}
				&\jmath^{m}\colon\sg{m}\xrightarrow{\sim} \rho_\del^{-m_e}\rho_\sigma^{-m_\psi}\Cinf(\Sfp\times\Sfp),\\
				&\jmath^{m}(a)(\rho_\del,z,\rho_\sigma,\zeta)=\rho_\del^{m_e}\rho_\sigma^{m_\psi}a(R^{-1}(\rho_\del,z), R^{-1}(\rho_\sigma,\zeta)),
			\end{aligned}
		\end{equation}
		where $(\rho_\del,z,\rho_\sigma,\zeta)$ are coordinates near the corner of $\Sfp\times\Sfp$, $\rho_\del$ and $\rho_\sigma$ being boundary-defining functions with $R^\ast \rho_\del=\braket{x}^{-1}$, $R^\ast \rho_\sigma=\braket{\xi}^{-1}$.
	\item 
		Under the isomorphism of the previous point 2., the principal symbol can be computed by restriction to the boundary hyper-surfaces as
		\begin{equation}
			\label{eq:SG and sc principal symbols}
			\begin{aligned}
			\sigma_{e}^{m_e}(a)(x,\xi)&=\abs{x}^{m_e}\jmath^m(a)\left(0,\frac{x}{\abs{x}},R(\xi)\right),\\
			\sigma_{\psi}^{m_\psi}(a)(x,\xi)&=\abs{\xi}^{m_\psi}\jmath^m(a)\left(R(x),0,\frac{\xi}{\abs{\xi}}\right),\\
			\sigma_{\psi e}^m(a)(x,\xi)&=\abs{x}^{m_e}\abs{\xi}^{m_\psi}\jmath^m(a)\left(0,\frac{x}{\abs{x}},0,\frac{\xi}{\abs\xi}\right).
			\end{aligned}
		\end{equation}
	\end{enumerate}
\end{theo}
\begin{rem}
	\label{rem:tensor product of smooth function on the sphere}
	In view of Remark \ref{rem:global classical symbols}, 1. and 2. in Theorem \ref{theo:scattering equals classical SG} are consequences of the fact that $\Cinf(\Sfp\times\Cinf(\Sfp))\cong\Cinf(\Sfp)\hat\tens_\pi\Cinf(\Sfp)$.
\end{rem}

We recall that the compatibility conditions for principal $SG$-symbols are actually sufficient for the existence of a global symbol. In the scattering picture, this corresponds to principal symbols being smooth functions on $B_{sc}X$ according to Remark \ref{rem:smooth functions across the corner}. By definition of $\Cinf(\sc{\bar{T}^\ast X})$ and $\Cinf(B_{sc} X)$, we have therefore the extension result of Proposition \ref{prop:compatible principal symbol implies global} below. First however, let us recall and introduce some notation we will use throughout the rest of our treatment (confer Example \ref{example:compactification of Rn 2})
\begin{defin}
	\label{def:boundary hypersurfaces}
	The space $B_{sc}X$ is given by the disjoint union $B_{sc}X=\widetilde{\Wcal}_\psi\cup\widetilde{\Wcal}_e\cup\widetilde{\Wcal}_{\psi e}$ where the manifolds $\widetilde{\Wcal_\bullet}$ are (the second equality is what happens in the model case $X=\Sfp$)
	\[\begin{aligned}
		\widetilde{\Wcal}_\psi&=\sc{\Sf{\ast}}\mathring{X}\quad (=\mathring{\Sfp}\times\Sf{n-1}),\\
		\widetilde{\Wcal}_e&=\sc{T^\ast_{\del X}X}\quad(=\Sf{n-1}\times\mathring{\Sfp}),\\
		\widetilde{\Wcal}_{\psi e}&=\sc{\Sf{\ast}_{\del X}X}\quad(=\Sf{n-1}\times\Sf{n-1}).
	\end{aligned}\]
	We set $\bar{\Wcal}_{\psi}=\widetilde{\Wcal}_\psi \cup\widetilde{\Wcal}_{\psi e}$, respectively $\bar{\Wcal}_e=\widetilde{\Wcal}_e \cup\widetilde{\Wcal}_{\psi e}$. In particular, $\bar{\Wcal}_e$ and $\bar{\Wcal}_\psi$ are the boundary hyper-surfaces of $\sc{\bar{T}^\ast X}$, both manifolds with boundary, $\widetilde{\Wcal}_e$ and $\widetilde{\Wcal}_\psi$ are the respective interiors, and $\widetilde{\Wcal}_{\psi e}$ is the corner. 
	
	Finally, we will also find use for the following spaces:
	\[\begin{aligned}
		\Wcal_\psi&=\widetilde{\Wcal}_{\psi}\times\R^+\cong T^\ast X\setminus\{0\}\quad (=\mathring\Sfp\times\R^n_0),\\
		\Wcal_e&=\R^+\times\widetilde{\Wcal}_{e}\cong T^\ast (\R^+\times\del X)\quad (=\R^n_0\times\mathring{\Sfp}),\\
		\Wcal_{\psi e}&=\R^+\times\widetilde{\Wcal}_{\psi e}\times\R^+\cong T^\ast (\R^+\times\del X)\setminus\{0\}\quad (=\R^n_0\times\R^n_0).
	\end{aligned}\]
	In the above formul\ae, $\{0\}$ denotes the zero section of the involved cotangent bundles. Also notice that $\R^+\times\del X$ is, topologically, the (interior of) a collared neighbourhood of $\del X$, seen however with the metric structure of a cone over $\del X$.
\end{defin} 
\begin{prop}
	\label{prop:compatible principal symbol implies global}
	Let $a_\bullet\in\Cinf(\bar{\Wcal}_\bullet),\bullet\in\{e,\psi\}$ and $a_{\psi e}\in\Cinf(\widetilde{\Wcal}_{\psi e})$ be smooth functions satisfying $a_e|_{\widetilde{\Wcal}_\psi e}=a_\psi|_{\widetilde{\Wcal}_{\psi e}}=a_{\psi e}$. There exists then a function $a\in\Cinf(\B^n\times\B^n)$ such that $a|_{\bar{\Wcal}_\bullet}=a_\bullet$. 
\end{prop}

Under pull-back with the radial compactification map, the associated symbol $\check{p}$ of \eqref{eq:associated sg symbol} is of course nothing else than a particular choice of such an extension to the interior. There is a similar extension result for maps between the boundary faces, namely Theorem \ref{theo:scattering maps extend to the interior} below (as formulated in \cite{coriasco2019lagrangiandistributions}, Proposition 1.30). To state it we need to discuss scattering maps between scattering manifolds.

\begin{defin}
	\label{def:scattering maps}
	Given two scattering manifolds $X,Y$ and a smooth map $C\colon X\rightarrow Y$, we say that $C$ is a \textit{scattering map} if, for any given $m\in\R$ and any $a\in\rho_Y^{-m}\Cinf(Y)$ it holds true that 
	\begin{enumerate}
		\item $C^\ast a\in\rho_X^{-m}\Cinf(X)$;
		\item if $q=C(p)$ for $p\in X$ and $\rho_Y^{-m}a(q)>0$, then $\rho_X^{-m}C^\ast a(p)>0$.
	\end{enumerate}
	Locally, this corresponds to the fact that scattering maps ($sc$-maps for short) are exactly those maps which pull back $\rho_Y$ to $\rho_X h$ for some positive function $h\in\Cinf(X)$. We call maps satisfying this condition \textit{local sc-maps} and extend their definition to manifolds with corners as follows. Given \textit{complete sets} of boundary-defining functions $(\rho_i)$ for $X$ and $(r_i)$ for $Y$ (that is, such that the whole boundary of the manifold can be identified with $\prod\rho_i=0$ or $\prod r_i=0$ respectively), $C$ is a local $sc$-map if there exist positive functions $h_i\in\Cinf(X)$ for which it holds true $C^\ast r_i=h_i\rho_i$. 
\end{defin}

In \cite{coriasco2019lagrangiandistributions} it is proven that $sc$-maps are really morphisms in the category of scattering manifolds. We will not need many facts from the theory explored there, so we'll only state what we'll need in the following. In particular, the next theorem is the aforementioned extension result near the corner of a product $X\times Y$ of manifolds with boundary.

\begin{theo}
	\label{theo:scattering maps extend to the interior}
	Consider manifolds with boundary $X_i,Y_i,i=1,2$, with boundary defining functions $\rho_{X_i},\rho_{Y_i}$, and the products $B_i=X_i\times Y_i$. Consider, for $\bullet\in\{e,\psi\}$, local $sc$-maps $C_\bullet\colon\bar\Wcal^1_\bullet\rightarrow\bar\Wcal^2_\bullet$ defined near a point $p\in\del X_1\times \del Y_1$, such that $C_e|_{\del X_1\times \del Y_1}=C_\psi|_{\del X_1\times \del Y_1}$. There exists then a local $sc$-map $C$ on a neighbourhood $p\neigh U\subset B_1$ such that $C_e=C|_{\del X_1\times Y_1}$, $C_\psi=C|_{X_1\times\del Y_1}$ and
	\begin{equation}
		\label{eq:scattering extension derivative property}
		\partder{C^\ast \rho_{Y_2}}{\rho_{X_1}}=\partder{C_\ast \rho_{X_2}}{\rho_{Y_1}}=0.
	\end{equation}
	Moreover, provided that both $C_\bullet$ are local diffeomorphisms near $p$, then $C$ is also a local diffeomorphism near $p$. If both $C_\bullet$ are diffeomorphisms defined in a neighbourhood of the whole corner in the respective boundary hyper-surface, then we can pick $C$ to be a diffeomorphism of a neighbourhood of the corner in $B_1$ onto a neighbourhood of the corner in $B_2$.
\end{theo}

\subsection{Symplectic and contact properties of the scattering bundle}

Recall that on every cotangent bundle $T^\ast X$ a canonical 1-form, known as the Liouville form $\lambda$, is defined. If $x$ are local coordinates on $X$ and $(x,\xi)$ are the induced canonical coordinates on the cotangent bundle, $\lambda$ takes the form $\xi\de x$. The differential $\omega=\de \lambda=\de\xi_i\wedge\de x^i$ is a symplectic form on $T^\ast X$, which is therefore an \textit{exact symplectic manifold}. In the classical theory of FIOs, the 1-form $\lambda$ plays a crucial role, in that it determines the conic Lagrangian submanifolds and, therefore, the microlocal form of the operators. In the global scattering calculus, it turns out that $\lambda$ does not suffice to describe the peculiar features that appear at ``spatial infinity'' $\del X$. In the existing literature, the problem has been obviated to by introducing a similar structure on a collared neighbourhood of $\del X$. Since this is paramount for our future discussion, we recall hereafter, following \cite{coriasco2017lagrangiansubmanifolds} and \cite{melrose1996scatteringflow}, the important concepts, all the while introducing the notation we shall refer to.

First, we note that the Poisson brackets $\{\cdot,\cdot\}$ associated with the symplectic form $\omega$ extends to $B_{sc}X$. If we keep in mind Theorem \ref{theo:scattering equals classical SG}, we understand that, in the model case $X=\Sfp$, we can reformulate Proposition \ref{prop:symplectic properties of sg} in the language of scattering geometry (cf. \cite{melrose1994spectralscatteringtheory}, Proposition 4).

\begin{prop}
	\label{prop:scattering poisson structure}
	The Poisson structure on $T^\ast \mathring{X}$, induced by the canonical symplectic form $\de \lambda$, extends to $B_{sc}X=\del(\sc{\bar{T}^\ast X})$ as a filtered bracket between (weighted) smooth functions on $B_{sc}$. More precisely, the Poisson bracket extends to a map 
	\[
	\{\cdot,\cdot\}\colon\Cinf(B_{sc}X;S^{l_1,m_1})\times\Cinf(B_{sc}X;S^{l_2,m_2})\rightarrow \Cinf(B_{sc}X;S^{l_1+l_2-1,m_1+m_2-1}),
	\]
	where we understand that we obtain such a map on each boundary hyper-surface and that $\{\cdot,\cdot\}$ preserves the compatibility condition in the corner. Namely, if we have tuples $(a_\bullet),(b_\bullet)$, extended to functions $a,b\in\Cinf(\sc{\bar{T}^\ast X})$ according to Proposition \ref{prop:compatible principal symbol implies global}, then the Poisson brackets of the tuples $(\{a_\bullet,b_\bullet\})$ satisfy the same conditions and admit therefore an extension to a smooth function $c\in\Cinf(\sc{\bar{T}^\ast X})$. In particular, the extensions $a,b$ and $c$ can be chosen so that $c=\{a,b\}$.
\end{prop}

Second, denote the Liouville 1-form by $\lambda_\psi$ and recall that $\lambda_\psi$, being homogeneous of degree $1$ in the fibres of $T^\ast X$, induces a contact structure on $\Sf{\ast}X$. In particular, $\lambda_\psi$ restricts to a contact form there, so it can be rescaled by a conformal factor without changing the actual structure. Of course, since $\widetilde{\Wcal}_\psi$ is diffeomorphic to $\Sf{\ast}\mathring{X}$, it also is a contact manifold and we give it a specific contact structure as follows. For $\rho_\sigma=R_\ast\braket{\xi}^{-1}$, our standard boundary-defining function, consider the inward-pointing radial vector field $\rho_\sigma\del_{\rho_\sigma}$. This is just the Euler vector field $\xi_j\del_{\xi_j}$ expressed in adapted coordinates at the boundary. Then, since $\lambda_\psi=\rho_\sigma\del_{\rho_\sigma}\lrcorner\omega$, and $\rho_\sigma$ is positive in the interior, we can consider the \textit{rescaled} radial vector $\rho_\sigma^2\del_{\rho_\sigma}$, which, inserted into $\omega$, gives the 1-form $\alpha_\psi\equiv\rho_\sigma^2\del_{\rho_\sigma}\lrcorner\omega$. This form is conformally equivalent to $\lambda_\psi$ \textit{in the interior}, and we consider it as the ``standard'' contact form on $\widetilde{\Wcal}_\psi$, the co-sphere bundle at infinity.

A similar process produces our ``standard'' contact form on $\widetilde{\Wcal}_e$. Choosing a collared neighbourhood and canonical coordinates $(\rho_N,z,\tau,\mi)$ near $\del \sc{T^\ast X}$, with $\rho_N$ a boundary-defining function, $\omega$ is the differential of a 1-form $\lambda_e$ given by 
\begin{equation}
	\label{eq:e-contact form}
	\lambda_e=\frac{1}{\rho_N}(\de \tau+\mi_k\de z^k).
\end{equation}
Notice that $\lambda_e=\rho_N\del_{\rho_N}\intprod\omega$ is \textit{not} a smooth 1-form on $T^\ast X$, since it blows up at $\del X$. It is however smooth as a scattering 1-form on $\sc{T^\ast X}$ and can be rescaled to give a smooth 1-form $\alpha_e=\rho_N^2\del_{\rho_N}\intprod\omega$ near the interior of the boundary hyper-surface. Since the collared neighbourhood gives a conic structure near $\widetilde{\Wcal}_e$, we obtain a contact structure on the boundary hyper-surface as above. Moreover, the contact distribution over $\widetilde{\Wcal}_e$ is unambiguously determined by the restriction of $\alpha_e$ to $\widetilde{\Wcal}_e$, and $\alpha_e$ is a contact form there. Notice again that this is conformally equivalent to $\lambda_e$ \textit{in the interior} but not at ``spatial infinity'' $\del X$.

We have then a pair of 1-forms $\alpha_\psi,\alpha_e$ on $T^\ast X$ which determine the symplectic structure at either spatial infinity or fibre infinity. In particular, since they induce contact distributions of $\widetilde{\Wcal}_\psi$ and $\widetilde{\Wcal}_e$, we can speak of Legendrian submanifolds at infinity. Recall that, classically, FIOs are operators whose singularities are contained in conic Lagrangian submanifolds of $T^\ast X\setminus\{0\}$, which in turn can be identified with Legendrian submanifolds of the co-sphere bundle by restriction. Parallel to this situation, Melrose and Zworski \cite{melrose1996scatteringflow} introduced a class of Legendrian distributions in $\widetilde{\Wcal}_e$, which has been subsequently generalised by Coriasco, Doll and Schulz in \cite{coriasco2019lagrangiandistributions} to include singularities in the whole $B_{sc}X$. As a last step, before introducing our notion of singular symplectomorphism, we recall the notion of the so-called $SG$/$sc$-Lagrangians/Legendrians, that have been studied by Coriasco and collaborators, and refer to the cited literature for more details.

\begin{defin}[$SG$-Legendrian]
	\label{def:SG Legendrian}
	Let $\Lambda=\bar{\Lambda}_e\cup\bar{\Lambda}_\psi\subset B_{sc}X$ be a closed submanifold, where $\Lambda_e\subset\bar{\Wcal}_e$ and $\Lambda_\psi\subset\bar{\Wcal}_\psi$ and $\bar{\Lambda}_\bullet$ denotes the closure of $\Lambda_\bullet$. $\Lambda$ is called an $SG$-Legendrian submanifold if it satisfies the following conditions:
	\begin{enumerate}
		\item
			$\Lambda_\psi$ is Legendrian in $\widetilde{\Wcal}_\psi$;
		\item 
			$\Lambda_e$ is Legendrian in $\widetilde{\Wcal}_e$;
		\item
			$\bar{\Lambda}_e$ has boundary if and only if $\bar{\Lambda}_\psi$ has boundary, in which case
			\[
			\Lambda_{\psi e}=\del\bar{\Lambda}_e=\del\bar{\Lambda}_\psi=\bar{\Lambda}_e\cap\bar{\Lambda}_\psi
			\]
			with clean intersection.
	\end{enumerate}
\end{defin}

Having finished our recap, we build upon this structure to introduce our notion of symplectomorphism. Here and later, $X,Y$ are scattering manifolds and we use the notation of Definition \ref{def:boundary hypersurfaces} to refer to subsets of $B_{sc}X$ and $B_{sc}Y$ indifferently (no confusion should arise, in particular, in view of the fact that all statements to come can be checked in local coordinates).

\begin{defin}
	\label{def:scattering-symplectomorphism}
	Let $U$, respectively $V$, be open in $B_{sc}X$, respectively $B_{sc}Y$. Assume $\chi\colon U\rightarrow V$ is a diffeomorphism (in particular, a $sc$-map), given as a pair of maps $\chi=(\chi_e,\chi_\psi)$ defined on $U\cap\bar\Wcal_e$ and $U\cap\bar\Wcal_\psi$, respectively (if $U\cap\bar{\Wcal_\bullet}=\void,\bullet\in\{e,\psi\},$ we understand that $\chi_\bullet$ is not present). We define the notion of \textit{scattering symplectomorphism} or \textit{scattering-canonical transformation} (SCT) depending on whether $U$ and (consequently) $V$ intersect the corner: If $U\cap\widetilde{\Wcal}_{\psi e}=\void$, we say that $\chi$ is an SCT if $\chi_\bullet$ is a contact diffeomorphism with respect to $\alpha_\bullet$; else, $\chi$ is an SCT if both $\chi_e$ and $\chi_\psi$ are contact diffeomorphism in the interior of the respective boundary face and $\chi$ preserves the Poisson bracket on $\Cinf(U;S^{l,m})$ across the corner.
\end{defin}

We have the following easy-to-prove properties of a scattering-symplectomorphism, which reflect classical behaviour of regular symplectomorphisms.

\begin{lemma}
	\label{lemma:properties of scattering-symplectomorphism}
	Let $\chi\colon U\rightarrow V$ be an SCT. Then $\chi$ maps $SG$-Legendrian submanifolds in $U$ to $SG$-Legendrian submanifolds in $V$. Moreover, if $U\cap\widetilde{\Wcal}_{\psi e}=\void$, then $\chi_e$, respectively $\chi_\psi$, extends to a homogeneous symplectomorphism on a collared neighbourhood of $\del T^\ast X$, respectively on $T^\ast_0 X$, which admits a local parametrisation via a $e$-homogeneous, respectively $\psi$-homogeneous, phase function. Finally $\chi$ induces Poisson maps on the boundaries and corners (i.e. it preserves the Poisson structure in Proposition \ref{prop:scattering poisson structure}). 
\end{lemma}
\begin{proof}
	First, recall that a diffeomorphism between contact manifolds $X$ and $Y$ is contact if and only if it preserves each Legendrian submanifold (cf. \cite{sasaki1964characterisation}). Then, since $\chi$ is a diffeomorphism by definition, we obtain at once that $SG$-Legendrians are preserved if $U\cap\widetilde{\Wcal}_{\psi e}=\void$. On the other hand $\chi$ preserves clean intersection so the preservation of $SG$-Legendrians for $U\cap \widetilde{\Wcal}_{\psi e}\neq \void$ also follows. 
	
	The second statement is obtained by applying the classical procedure of symplectisation to the contact manifolds $\widetilde{\Wcal}_\psi$ and $\widetilde{\Wcal}_e$ separately. Indeed notice that, if $U$ does not intersect the corner, then $U$ can be ``conified'' to a subset of $T^\ast X$. Near the boundary we just have to pick a collared neighbourhood $\del X\times [0,1)$ and pull-back $(U\cap\widetilde{\Wcal}_e)\times[0,1)$ with $R\times \id$, while on the co-sphere at $\infty$ we just identify $U\cap\Wcal_\psi$ with a subset of $\Sf{\ast}X$ and consider, as in the classical theory, the associated conic neighbourhood. Let us spend a few extra words to describe, for the $e$-component, how one obtains a symplectic form and can extend $\chi_e$ to a homogeneous symplectomorphism. More details can be found in \cite{mcduff2017introductionsymplectictopology}, Section <<Symplectization of contact manifolds>>, or in \cite{arnold1995mathematicalmethods}, Appendix 4. 
	We are given the contact form $\alpha_e=\de \tau+\mi_k\de z^k$ on $\widetilde{\Wcal}_e$. In the conified neighbourhood $\R^+\times U_e$, we are introducing the new coordinate $\rho_N$, effectively identifying $U_e$ with $\{1\}\times U_e$, and can consider 
	\[
	\omega_e=-\de\,(\rho_N\alpha_e)=-\rho_N\left(\frac{\de \rho_N}{\rho_N}\wedge\alpha_e+\de\alpha_e\right).
	\]
	It is readily checked that this is now a symplectic form on the conified neighbourhood. By definition of contact transformation, $\chi^\ast\alpha_e=g\alpha_e$ for a positive smooth function $g$. Then, the extension $C_e(\rho_N,z,\tau,\mi)\equiv(\rho_N/g(z,\tau,\mi),\chi(z,\tau,\mi))$ is homogeneous symplectic. Indeed, 
	\[
	C_e^\ast\omega_e=-\de C_e^\ast(\rho_N\alpha_e)=-\de\left(\frac{\rho_N}{g} g\alpha_e\right)=\omega_e,
	\]
	proving that $C_e$ is symplectic. The homogeneity is, on the other hand, manifest.
	
	We have now the homogeneous symplectic extensions $C_e$ and $C_\psi$. Now, the local parametrisation of $C_\psi$ is the classical result of Hörmander, Proposition 25.3.3 in \cite{hormander2009analysisFIO}. On the other hand, for $C_e$ it suffices to exchange the rôles of variables and covariables (also cf. \cite{melrose1996scatteringflow}, Section 6).
	
	Concerning the last statement, observe that, at the corner, the Poisson structure is preserved by definition. On the other hand, the homogeneous symplectic extensions just constructed guarantee that $\{,\}$ is preserved away from the corner.
\end{proof}

In the next Theorem \ref{theo:scattering canonical transformations} we present a more thorough analysis of the structure of a $sc$-symplectomorphism defined near the corner. To avoid overburdening the notation, let us first clarify that the local expressions given below hold true in coordinates $(x,\xi)$, obtained as the pull-back of standard systems of coordinates near the boundary faces (or possibly the corner). In particular, the $\alpha$'s are angular coordinates on $\Sf{n-1}$ and the boundary-defining function is the inverse of the radial coordinate in polar coordinates. That is, $x^i=\abs{x}X^i(\alpha)$ for smooth functions $X^i$ such that $(X^1)^2+\dots+(X^n)^2=1$ and $\rho_\del=1/\abs{x}$. We employ the same convention for $\xi$'s and $\beta$'s. Also, in this notation we will consider homogeneous extensions of functions $\Sf{n-1}\times\R^n$ to $\R^n_0\times \R^n$. To be precise, we will look for $\R^+$-equivariant maps $C_\bullet\colon\Wcal_\bullet\rightarrow\Wcal_\bullet$ agreeing with $\chi_\bullet$ on $\widetilde\Wcal_\bullet$. Any such map is of the form (for example $\bullet=e$, w.l.o.g.)
\begin{equation}
	\label{eq:equivariant extension 1}
	C_e(r,\alpha,\xi)=(f_e(\alpha,\xi)r,\chi_e(\phi,\xi))
\end{equation}
for some smooth $f_e\in\Cinf(\widetilde\Wcal_{e})$, where $(r=\abs{x},\alpha)$ are the above polar coordinates on $\R^n_0$ and $\xi$ are coordinates on $\R^n$. The inverse of such map, again taking polar coordinates on the first factor and global coordinates on the second, is given by
\begin{equation}
	\label{eq:equivariant extension inverse}
	C_e^{-1}(s,\alpha,\eta)=\left(\frac{s}{f_e(\chi_e^{-1}(\alpha,\eta))},\chi_e^{-1}(\alpha,\eta)\right).
\end{equation} 
Recalling again Section <<Symplectization of contact manifolds>> in \cite{mcduff2017introductionsymplectictopology}, it must be possible to choose $f_e$ appropriately to ensure that $C_e$ so extended is symplectic. We will find the explicit form of the section $f_e$ (and $f_\psi$ too, of course) in the next chapter, in the course of the proof of Lemma \ref{lemma:canonicaltransformation}. For the moment, we content ourselves with saying that such a choice is possible.
\begin{theo}
	\label{theo:scattering canonical transformations}
	Let $\chi$ be a $sc$-canonical transformation, between open sets $U, V$ as above, with $U\cap\widetilde\Wcal_{\psi e}\neq\void$. Then $\chi$ is given as the datum of a triple of diffeomorphisms $(\chi_e,\chi_\psi,\chi_{\psi e})$, for $\chi_\bullet\colon\widetilde\Wcal_\bullet\rightarrow\widetilde\Wcal_\bullet$, such that
	\begin{enumerate}
		\item 
			If $\chi_e(\alpha,\xi)=(T(\alpha,\xi),H(\alpha,\xi))$ for $T\colon\Sf{n-1}\times\R^n\rightarrow\Sf{n-1},H\colon\Sf{n-1}\times\R^n\rightarrow\R^n$, then the components of $H$ are elements of $\Cinf(\Sf{n-1};\sym{1}(\R^n))$;
		\item 
			If $\chi_\psi(x,\beta)=(Y(x,\beta),G(x,\beta))$ for $G\colon\R^n\times\Sf{n-1}\rightarrow\Sf{n-1},Y\colon\R^n\times\Sf{n-1}\rightarrow\R^n$ then the components of $Y$ are elements of $\Cinf(\Sf{n-1};\sym{1}(\R^n))$;
		\item 
			If $\chi_{\psi e}(\alpha,\beta)=(A(\alpha,\beta),B(\alpha,\beta))$ and we write $\chi_e,\chi_\psi$ as above, the principal symbol of $Y$, respectively $H$, restricted to  $\Sf{n-1}\times\Sf{n-1}$ coincides with $T$, respectively $F$. More generally, it holds true that
			\begin{align}
				A(\alpha,\beta)&=\lim_{\lambda\rightarrow+\infty}T(\alpha,\lambda\xi)=\lim_{\lambda\rightarrow+\infty}\frac{1}{\lambda}Y(\lambda x,\beta),\\
				B(\alpha,\beta)&=\lim_{\lambda\rightarrow+\infty}\frac{1}{\lambda}H(\alpha,\lambda\xi)=\lim_{\lambda\rightarrow+\infty}G(\lambda x,\beta);
			\end{align}
		\item 
			We can pick homogeneous extensions $C_e$ of $\chi_e$ in $\alpha$, $C_\psi$ of $\chi_\psi$ in $\beta$ and $C_{\psi e}$ of $\chi_{\psi e}$ in $\alpha$ and $\beta$ separately, so that $C_\bullet$ is a symplectomorphism, homogeneous in the respective variables;
		\item 
			Writing these extensions as $C_\bullet(x,\xi)=(Y_\bullet(x,\xi), H^\bullet(x,\xi))$ for $Y_\bullet=(Y_\bullet^1,\dots,Y_\bullet^n)$, $H^\bullet=(H^\bullet_1,\dots,H^\bullet_n)$, we have that each triple $(Y_\bullet^j)$, $(H^\bullet_k)$ can be continued to a classical $SG$-symbol near ``infinity''. In particular there is a diffeomorphism $C(x,\xi)=(Y^j(x,\xi),H_k(x,\xi))$ near ``infinity'' having $(C_e,C_\psi,C_{\psi e})$ as ``principal symbol''.
	\end{enumerate} 
\end{theo}

\begin{proof}
	
	For 1., notice that $\chi_e$ is given as a diffeomorphism of $\Sf{n-1}\times\B^n$. We pull it back to a diffeomorphism of $\Sf{n-1}\times\R^n$ using $\id\times R$. But then the $\R^n$-components of $\chi_e$ must be classical symbols of order 1 in $\xi$, depending in a smooth way on $\alpha\in\Sf{n-1}$. This is exactly the claim.
	
	For 2., we argue exactly as in 1., exchanging the rôles of the variables.
	
	To prove 3.\ notice that the expressions involving $H$ and $Y$ are just the standard formul\ae\, to compute the principal symbol for the classes $S^1(\R^n)$, depending on a parameter on $\Sf{n-1}$. Recalling Theorem \ref{theo:scattering equals classical SG} we see immediately that we can compute it also by restriction to $\Sf{n-1}\times\Sf{n-1}$. Now, $\chi_\bullet$ is obtained as a diffeomorphism of $\del (\Sf{n}_+\times\Sf{n}_+)$, so in the corner $\chi_e=\chi_\psi$. Pulling this back with $\id\times R$ and $R\times\id$ and comparing the respective components gives the claimed formulas.
	
	4.\ is clear if one exploits the close relation between canonical transformations and contact diffeomorphisms. However, a more explicit construction will be given in the proof of Lemma \ref{lemma:canonicaltransformation}, where we will see that the choice of order reductions uniquely determines the homogeneous extensions to be symplectic.
	
	5.\ is now a consequence of the above facts. Indeed, the components of the homogeneous extensions $C_\bullet$ satisfy symbol estimates in the non-homogeneous variables. In particular, each pair of components $(Y_e^j,Y_\psi^j)$, respectively $(H^e_k,H^\psi_k)$, can be continued to a symbol $Y^j\in\sg{\indi_e}(\R^n\times\R^n)$, respectively $H_k\in\sg{\indi_\psi}(\R^n\times\R^n)$. We can choose them so that the resulting map $C(x,\xi)=(Y(x,\xi),H(x,\xi))$ is a diffeomorphism, in view of Theorem \ref{theo:scattering maps extend to the interior}. Indeed, our maps are all globally defined on the boundary hyper-surfaces and the corner, so they can be patched together correctly. 
\end{proof}

\begin{rem}
	One would certainly hope that the extension in 5. of the above theorem could be achieved symplectic. However, the best of our efforts could not deduce this desirable fact from the properties of $C_e,C_\psi$ and $C_{\psi e}$. 
\end{rem}
\begin{rem}
	\label{rem:symplectic rotation as a scattering map}
	Notice that, for a scattering map on a manifold with corners, one preassigns an ordering on the set of the boundary-defining functions, so that at a corner we are specifying which boundary hyper-surface is mapped to which. For example on $\Sf{n}_{+}\times\Sf{n}_{+}$, seen as $\sc{\bar{T^\ast \Sf{n}_{+}}}$, we use the ordered set of boundary-defining functions $(\rho_N,\rho_\sigma)$. Then, it is easily seen that the symplectic rotation $F\colon(x,\xi)\rightarrow(\xi,-x)$, extended to the compactification as in Theorem \ref{theo:scattering equals classical SG}, is \textit{not} a scattering map in this sense, since $F^\ast \rho_N=\rho_\sigma$ and vice-versa. This reflects the fact that, on a manifold with boundary $X$, the two components of the joint scattering symbol live as smooth functions on two in principle different compact manifolds, namely, the scattering co-sphere bundle and the boundary of $X$ (pulled back to the compactified scattering cotangent bundle). Of course, nothing in principle prevents us from considering $F$ as some sort of ``generalised scattering map'' on the model case $\Sfp$. However we notice that pull-back along $F$ \textit{does not} preserve $SG$-classes, since it exchanges the two filtrations as in Proposition \ref{prop:sc-operators conjugate with FT}. We will therefore assume that SCTs cannot exhibit this kind of behaviour, although we will comment again on this point at the very end of Chapter \ref{chap:OPI}.
\end{rem}

We now come to the core of this section: the relation between scattering-symplectic maps and the classical $SG$-phase functions. 

\begin{theo}[Parametrising $sc$-symplectomorphisms]
	\label{theo:parametrisation sc-symplectomorphism}
	Let $\chi\colon \del(\Sfp\times\Sfp)\rightarrow\del(\Sfp\times\Sfp)$ be a (possibly only locally defined) scattering canonical transformation. Then, at each point $(p,q)$ on the graph of $\chi$, we can find a neighbourhood $\tilde U$ of $p$, a neighbourhood $\tilde{V}$ of $q$ and an $SG$-phase function $\phi(x,y,\xi)\in\sg{\indi}_{(x,y),\xi}$, parametrising a neighbourhood of $(p,q)$. More explicitly, if $p$ does not lie on the corner, then we can parametrise the homogeneous symplectic extension $C$ of $\chi$ near $p$ via a homogeneous phase function in the classical sense. On the other hand, if $p$ is in the corner then there is a conic neighbourhood $U_e$, respectively $U_\psi$, associated with a neighbourhood $\tilde{U}_e$, respectively $\tilde{U}_\psi$, of $p$ in $\Sf{n-1}\times\Sfp$, respectively $\Sfp\times\Sf{n-1}$, and we find a phase $\phi$ as above such that $\phi_e=\sigma_e(\phi)$, respectively $\phi_\psi=\sigma_\psi(\phi)$, parametrises the graph of $C_e$, respectively $C_\psi$, in the usual sense for conic Lagrangians, and $\phi_{\psi e}$ parametrises the bi-homogeneous extension $C_{\psi e}$ of $\chi_{\psi e}$.
\end{theo}
\begin{proof}
	The case $(p,q)\in \widetilde\Wcal_\psi\times\widetilde\Wcal_\psi$ is just an instance of the classical parametrization result for homogeneous symplectomorphisms of Hörmander, namely Proposition 25.3.6 in \cite{hormander2009analysisFIO}. To see this, consider neighbourhoods $U$ of $p$ and $V$ of $q$ which are away from the corner $\tilde\Wcal_{\psi e}$. Then, we can exploit the triviality of the bundle $\pi\colon U\times\R^n_0\rightarrow U\times\Sf{n}$ to define a homogeneous extension $C$ of $\chi$. By picking the correct section of $\pi$, we can ensure that $C$ is actually a homogeneous canonical transformation (namely it preserves the canonical 1-form $\lambda_\psi$) and apply Hörmander's result. Similarly, if $(p,q)\in\widetilde\Wcal_e\times\widetilde\Wcal_e$, we can consider the trivial bundle $\R^n_0\times\R^n\rightarrow\Sf{n-1}\times\R^n$. Here we can pick homogeneous extensions in $x$ and reproduce the proof of Hörmander (notice also \cite{melrose1996scatteringflow}, Section 6) exchanging the rôles of $x$ and $\xi$ (in particular, using the fact that $C^\ast\lambda_e=\lambda_e$ for $\lambda_e=x^j\de\xi_j$). The case $(p,q)\in\widetilde\Wcal_{\psi e}\times\widetilde\Wcal_{\psi e}$ is a bit more involved and we adapt the careful analysis of \cite{coriasco2017lagrangiansubmanifolds}.
	
	Mimicking the construction there, we work in a chart $\widetilde{U}\subset\widetilde\Wcal$ around $p$ where on $\widetilde{U}\cap\widetilde\Wcal_{\psi e}$ we have coordinates in the form
	\begin{equation}
		\label{eq:coordinates near the corner}
		(\alpha^1,\dots, \alpha^{n-1},\sqrt{1-(\alpha^1)^2-\dots-(\alpha^{n-1})^2},\sqrt{1-(\beta_2)^2-\dots-(\beta_{n})^2},\beta_2,\dots,\beta_n).
	\end{equation}
	On $\widetilde\Wcal_e$ and $\widetilde\Wcal_\psi$ we use adapted coordinates 
	\begin{equation}
		\label{eq:coordinates on the boundary faces}
		\begin{aligned}
			(\alpha^1,\dots,\alpha^{n-1},\sqrt{1-(\alpha^1)^2-\dots-(\alpha^{n-1})^2},\rho_2,\beta_2,\dots,\beta_n)&\in\widetilde\Wcal_\psi,\\
			(\alpha^1,\dots,\alpha^{n-1},\rho_1,\sqrt{1-(\beta_2)^2-\dots-(\beta_{n})^2},\beta_2,\dots,\beta_n)&\in\widetilde\Wcal_e,			
		\end{aligned}
	\end{equation}
	where $\rho_1=\sqrt{1-(\alpha^1)^2-\dots-(\alpha^{n-1})^2}$ and $\rho_2=\sqrt{1-(\beta_2)^2-\dots-(\beta_{n})^2}$ are defining equations for the common boundary $\widetilde{U}\cap\widetilde\Wcal_{\psi e}$. We can similarly choose coordinates $(\theta,r_1,r_2,\gamma)$ satisfying the same relations in a chart $\tilde{V}$ around $q$. 
	In these coordinates the map $\tilde C=(\tilde{C}_e,\tilde{C}_\psi)$ can be expressed as
	\begin{equation}
		\label{eq:corner map in coordinates}
		\tilde{C}_\bullet(\alpha,\rho_1,\rho_2,\beta)=(T_\bullet(\alpha,\rho_1,\rho_2,\beta),r_1^\bullet(\alpha,\rho_1,\beta),r_2^\bullet(\alpha,\rho_2,\beta),G^\bullet(\alpha,\rho_1,\rho_2,\beta)),
	\end{equation}
	namely $\theta=T_\bullet$ and $\gamma=G^\bullet$ are equations defining the graph of $\tilde{C}_\bullet$ in $\tilde{U}\times\tilde{V}$.
	Let $\widetilde{U_\bullet}=\widetilde{U}\cap\widetilde\Wcal_\bullet$ and $U_\bullet$ be the conic set associated with $\widetilde{U_\bullet}$ under inverse radial compactification. Then, on $U_\bullet$ we can introduce ``polar coordinates'' and extend $C_\bullet$ homogeneously. For example, on $U_e$ we choose a section $f_e(\alpha,\rho_2,\beta)\colon\Sf{n-1}\times\R^n\rightarrow\R^n_0\times\R^n$, pull back the covariables using $R$, and set, for $\mi>0$ and $\rho_2>0$,
	\begin{equation}
		\label{eq:polar coordinates on U_e}
		\begin{aligned}
			&(x,\xi)\equiv(\mi\alpha^1,\dots,\mi\alpha^{n-1}, \mi\sqrt{1-\abs{\alpha}^2},R^{-1}(\rho_2\beta_1,\dots,\rho_2\beta_n)),\\
			&C_e(x,\xi)=\left(T_e\left(\frac{x}{\mi},\mi,R(\xi)\right),r_1^e\left(\frac{x}{\mi},R(\xi)\right),r_2^e\left(\frac{x}{\mi},R(\xi)\right),G^e\left(\frac{x}{\mi},\mi,R(\xi)\right)\right).
		\end{aligned}
	\end{equation}
	Again as in \cite{mcduff2017introductionsymplectictopology}, the section $f_e$ can be appropriately chosen to ensure that $C_e$ is symplectic and homogeneous in the $x$ variables. Therefore, it preserves the 1-form $\lambda_e$. Similarly, we have an extension $C_\psi$ which preserves the Liouville 1-form $\lambda_\psi$ (using a section $f_\psi\colon \Sf{n-1}\times U_\psi\rightarrow \R^n_0\times U_\psi$), and we can also define a map $C_{\psi e}$ by extending $\chi_{\psi e}$ using both sections $f_e,f_\psi$. Then, in Cartesian coordinates on $\R^n\times\R^n$, taking into account Theorem \ref{theo:scattering equals classical SG}, we have then 3 symplectomorphisms $C_e,C_\psi,C_{\psi e}$ defined for $x\neq 0,\xi\neq 0$ and $x,\xi\neq 0$, respectively. Their components are parts of principal $SG$-symbols:
	\begin{equation}
		\label{eq:SG-symplectomorphisms}
		\begin{aligned}
			C_e(x,\xi)&=(Y^1_e(x,\xi),\dots,Y^{n}_e(x,\xi),H^e_1(x,\xi),\dots,H^e_{n}(x,\xi)),\\
			C_\psi(x,\xi)&=(Y^1_\psi(x,\xi),\dots,Y^{n}_\psi(x,\xi),H^\psi_1(x,\xi),\dots,H^\psi_{n}(x,\xi)),\\
			C_{\psi e}(x,\xi)&=(Y^1_{\psi e}(x,\xi),\dots,Y^{n}_{\psi e}(x,\xi),H^{\psi e}_1(x,\xi),\dots,H^{\psi e}_{n}(x,\xi)),\\
			Y_e^j&\in\sg{(1),0},\quad Y_\psi^j\in\sg{1,(0)},\sigma_e(Y^j_\psi)=\sigma_\psi(Y^j_e)=Y^{j}_{\psi e}\\
			H^e_k&\in\sg{(0),1},\quad H^\psi_k\in\sg{0,(1)},\sigma_e(H_k^\psi)=\sigma_\psi(H_k^e)=H_k^{\psi e}.
		\end{aligned}
	\end{equation}
	The $SG$ estimates for these functions follow directly from the previous considerations, Chapter 6 in \cite{andrews2009SGFIOcomposition}, and our particular choice of coordinates. Now, the twisted graphs $\Lambda_\bullet=\gr'(C_\bullet)$ of these maps are conic Lagrangians (either in $x, \xi$ or both). As in the classical theory of canonical graphs, we can find (possibly after rearranging) a partition $I=(1,\dots, d),J=(d+1,\dots, n)$ so that $(x^J,\xi_I, \eta)$ can be taken as coordinates on $\Lambda_\bullet$.
	
	Notice that, in principle, in what follows we should choose different sets of coordinates for each map. However, the boundary hyper-surfaces intersect cleanly and all the changes of coordinates just defined are either diffeomorphism or homogeneous extensions outside a compact neighbourhood of $0$, so they preserve this structure. Hence, near the corner we can always take the same partitions $I,J$.
	
	From here to the end of the proof, we employ a modified Einstein convention, namely: the indices $i$ belong to $I$, $j$ to $J$, $k$ to $\{1,\dots,n\}$, and repeated $i$ or $j$ means summing only over $I$ and $J$. The other coordinates are defined implicitly on $\Lambda_\bullet$ as 
	\begin{equation}
		\label{eq:implicit coordinates}
		x^i=X^i_\bullet(x^J,\xi_I,\eta),\quad \xi_j=\Xi^\bullet_j(x^J,\xi_I,\eta),\quad \eta_k=H_k^\bullet(x^J,\xi_I,\eta).
	\end{equation}
	In view of Chapter 6 of \cite{andrews2009SGFIOcomposition}, we see that these function $X,\Xi,H$ must satisfy $SG$-estimates. In particular, they belong to the following classes (the classicality is implied by the fact that we are pulling-back smooth functions on the compactified space along $R$):
	\begin{equation}
		\label{eq:SG classes of implicit functions}
		\begin{aligned}
			X^i_e&\in\sg{(1),0}(\R^{n+d}\times\R^{n-d}),\quad X^i_\psi\in\sg{1,(0)}(\R^{n+d}\times\R^{n-d}),\\
			\Xi_j^e&\in\sg{(0),1}(\R^{n+d}\times\R^{n-d}),\quad\Xi_j^\psi\in\sg{0,(1)}(\R^{n+d}\times\R^{n-d}),\\
			H_k^e&\in\sg{(0),1}(\R^{n+d}\times\R^{n-d}),\quad H_k^\psi\in\sg{0,(1)}(\R^{n+d},\R^{n-d}).
		\end{aligned}
	\end{equation}
	Still following the ideas of \cite{coriasco2017lagrangiansubmanifolds}, we now define directly homogeneous phase functions that parametrise locally the graphs of these diffeomorphisms and show they can be patched together to a single $SG$-phase function. To begin with, we look at the condition that $C_e$ be an $e-$homogeneous canonical transformation. This amounts to $C_e$ preserving the 1-form $\alpha_e$, namely $C_e^\ast\alpha_e=\alpha_e$. In the above coordinate patches, this is expressed as
	\begin{equation}
		\label{eq:e-canonical equations near the corner}
		\begin{aligned}
			0&=(x\de\xi-y\de\eta)|_{\Lambda_e}\\
			&=X^i_e\de\xi_i+x^j\left(\partder{\Xi^e_j}{x^{j_1}}\de x^{j_1}+\partder{\Xi^e_j}{\xi_i}\de\xi_i+\partder{\Xi^e_j}{y^k}\de y^k\right)-Y_e^k\de\eta_k\\
			&=\left(X_e^i-x^j\partder{\Xi^e_j}{\xi_i}\right)\de\xi_i
			+x^{j_1}\partder{\Xi^e_{j_1}}{x^j}\de x^j +\left(x^j\partder{\Xi^e_j}{\eta_k}-Y_e^k\right)\de \eta_k.
		\end{aligned}
	\end{equation}
	Therefore, all expressions in parenthesis must vanish on the graph. Very similar relations hold true for $C_\psi$, which we give explicitly hereafter:
	\begin{equation}
		\label{eq:psi-canonical equations near the corner}
		\begin{aligned}
			\Xi^\psi_j+\xi_i\partder{X^i_\psi}{x^j}-\eta_k\partder{Y^k_\psi}{x^j}&=0,\\
			\xi_i\partder{X^i_\psi}{\xi_I}-\eta_k\partder{Y_\psi^k}{\xi_I}&=0,\\
			\xi_i\partder{X^i_\psi}{\eta_l}-\eta_k\partder{Y^k_\psi}{\eta_l}&=0.
		\end{aligned}
	\end{equation}
	Let us consider first the functions $S_\bullet$ defined by
	\begin{equation}
		\label{eq:generating functions}
		\begin{aligned}
			S_e(x^J,\xi_I,\eta)&=x^j\Xi_j^e(x^J,\xi_I,\eta),\\
			S_\psi(x^J,\xi_I,\eta)&=-X_\psi^i(x^J,\xi_I,\eta)\xi_i+Y^k_\psi\eta_k.
		\end{aligned}
	\end{equation}
	We prove that they are generating functions for the canonical transformations $C_\bullet$. First looking at $S_e$ we have, using \eqref{eq:e-canonical equations near the corner}, that
	\begin{equation}
		\label{eq:derivatives of e-generating function}
		\begin{aligned}
			\partder{S_e}{x^j}&=\Xi^e_j+x^{j_1}\partder{\Xi^e_{j_1}}{x^j}=\Xi^e_j,\\
			\partder{S_e}{\xi_i}&=x^j\partder{\Xi^e_j}{\xi_i}=X^i_e,\\
			\partder{S_e}{\eta^k}&=x^j\partder{\Xi^e_j}{\eta_k}=Y_e^k;
		\end{aligned}
	\end{equation}
	hence $S_e$ generates $\Lambda_e$. The computation for $S_\psi$ using \eqref{eq:psi-canonical equations near the corner} is very similar and gives
	\begin{equation}
		\label{eq:derivatives of psi-generating function}
		\partder{S_\psi}{x^j}=\Xi_j^\psi,\quad \partder{S_\psi}{\xi_i}=-X^i_\psi,\quad\partder{S_\psi}{\eta_k}=Y^k_\psi.
	\end{equation}
	We have then established that $S_\bullet$ is a generating function for $\Lambda_\bullet$, so we now consider the phase functions
	\begin{equation}
		\label{eq:SG-phase functions near the corner}
		\begin{aligned}
		\phi_e(x^I,x^J,y,\xi_I,\eta)&\equiv x^i\xi_i+y^k\eta_k+x^j\Xi^e_j\\
		&=x^i\xi_i+y^k\eta_k+S_e(x^J,\xi_I,\eta)\in\sg{(1),1}(\R^{2n}\times\R^{n+d}),\\
		\phi_\psi(x^I,x^J,y,\xi_I,\eta)&\equiv x^i\xi_i-y^k\eta_k-\xi_i X^i_\psi+\eta_k Y^k_\psi\\
		&=x^i\xi_i-y^k\eta_k+S_\psi(x^J,\xi_I,\eta)\in\sg{1,(1)}(\R^{2n}\times\R^{n+d}).
		\end{aligned}
	\end{equation}
	Then $\de_{\,(\xi_I,\eta)}\phi_\bullet=0$ if and only if \eqref{eq:e-canonical equations near the corner} and \eqref{eq:psi-canonical equations near the corner} hold true. Then, computing the other derivatives gives the desired parametrisation. 
	
	It remains to show that the functions $\phi_\bullet$ can be realised as the principal symbol of an $SG$-function. To this end, the methods of \cite{coriasco2017lagrangiansubmanifolds} and \cite{coriasco2019lagrangiandistributions} still prove viable: using \eqref{eq:principal symbol as limit} and keeping in mind that the tuples $(X^i_e,X^i_\psi,X^i_{\psi e}),(\Xi^e_j,\Xi_j^\psi,\Xi_j^{\psi e})$ and $(H_k^e,H_k^\psi,H_k^{\psi e})$ are principal symbols, we compute $\sigma_e(\phi_\psi)-\sigma_\psi(\phi_e)$ restricted to the graph of $C_{\psi e}$:
	\begin{equation}
		\label{eq:principal symbols of phase functions}
		\begin{aligned}
			\sigma_e(\phi_\psi)&=\lim_{\lambda\rightarrow\infty}\frac{1}{\lambda}\phi_\psi(\lambda x,\xi_I,\lambda y)\\
			&=\lim_{\lambda\rightarrow\infty}\frac{1}{\lambda}(\lambda x^i\xi_i-\lambda y^k\eta_k-X^i_\psi(\lambda x^J,\xi_I,\eta)\xi_i+Y^k_\psi(\lambda x^J,\xi_I,\eta))\\
			&=x^i\xi_i-y^k\eta_k-X^i_{\psi e}\xi_i+Y^k_{\psi e}\eta_k,\\
			\sigma_\psi(\phi_e)&=\lim_{\lambda\rightarrow\infty}\frac{1}{\lambda}\phi_e(x,\lambda\xi_I,y)\\
			&=\lim_{\lambda\rightarrow \infty}\frac{1}{\lambda}(\lambda x^i\xi_i+\lambda y^k\eta_k+x^j\Xi^e_j(x^J,\lambda\xi_I,y))\\
			&=x^i\xi_i+y^k\eta_k+x^j\Xi^{\psi e}_j\\
			\hspace{-1cm}\implies\quad(\sigma_\psi(\phi_e)-\sigma_e(\phi_\psi))|_{\Lambda_{\psi e}}&=(2y^k\eta_k+x^j\Xi^{\psi e}_j+X^i_{\psi e}\xi_i-Y^k_{\psi e}\eta_k)|_{\Lambda_{\psi e}}\\
			&=X^i_{\psi e}\xi_i+x^j\Xi_j^{\psi e}+Y^k_{\psi e}\eta_k=\inner{(x,y)}{(\xi,\eta)}|_{\Lambda_{\psi e}}
		\end{aligned}
	\end{equation}
	However, recall that $\Lambda_{\psi e}$ is bi-conic, so the phase $\phi_{\psi e}$ parametrising it is bi-homogeneous of degree $\indi$ in $(x,y)$ and $(\xi_I,\eta)$. 
	Since on the graph we have $(\xi,\eta)=\de_{(x,y)}\phi_{\psi e}(x,y,\xi_I,\eta)$, we can apply Euler's equation for homogeneous functions twice to obtain
	\begin{equation}
		\label{eq:pairing vanishes on bi-conic}
		\begin{aligned}
			\inner{(x,y)}{(\xi,\eta)}|_{\Lambda_{\psi e}}&=\inner{(x,y)}{\de_{\,(x,y)}\phi_{\psi e}(x,y,\xi_I,\eta)}\\
			&=\phi_{\psi e}(x,y,\xi_I,\eta)|_{\Lambda_{\psi e}}=\inner{(\xi_I,\eta)}{\de_{\,(\xi_I,\eta)}\phi_{\psi e}(x,y,\xi_I,\eta)}\\
			&=0,
		\end{aligned}
	\end{equation}
	where for the last equality we noticed that $\de_{\,(\xi_I,\eta)}\phi_{\psi e}=0$ is exactly the relation defining the set in $\R^{3n+d}$ which parametrises the graph of $C_{\psi e}$. Therefore, on the graph of $C_{\psi e}$ we have that $\sigma_e(\phi_\psi)=\sigma_\psi(\phi_e)$, which is the compatibility condition for $SG$-principal symbols. This proves that $(\phi_e,\phi_\psi)$ can be realised as the principal symbol of a function $\phi\in\sg{\indi}(\R^{2n}\times\R^{n+d})$. This concludes the proof.

\end{proof}

\begin{rem}
	\label{rem:Q-phase function}
	Looking at the phase functions \eqref{eq:SG-phase functions near the corner}, it is clear that, in bi-conic neighbourhoods of infinity, they actually belong to the class $\Qcal$. Namely, they are given as a sum of two terms $f(x,\theta)+g(y,\theta)$ for $\theta=(\xi_I,\eta)$ satisfying appropriate $SG$ estimates. 
\end{rem}
\begin{rem}
	\label{rem:local model}
	We remark that, while our discussion above was limited to the model case $\B^n\times\B^n$, our definition of SCTs and the related results all have a local character (in the sense that they can be checked in coordinates near the corners). Therefore, they apply \textit{mutatis mutandis} to general scattering manifolds, their scattering cotangent bundles, and the fibre-wise compactifications thereof. However, the theory of FIOs on scattering manifolds has not yet reached a completely satisfactory status. In particular, the concept of elliptic FIO in this setting has not yet been defined and analysed to the extent that we need. So, for our purposes in the coming Chapter \ref{chap:OPI} we will stick to the model case.
\end{rem}

	\section{Order-preserving isomorphisms}
\label{chap:OPI}
\setcounter{equation}{0}

\subsection{Preliminary definitions and auxiliary results}
In this chapter, we work exclusively in the model case of $\R^n$ and its compactification $\Sf{n}_{+}$. Although we believe that most of what follows should hold true in general for operators defined on an asymptotically Euclidian manifold $X$ (or even scattering), the theory of FIOs in this setting has not been studied in the required depth to allow us to formulate certain results below. On the other hand the nature of the argument is such that, given the existence of a sufficiently precise calculus structure, the computations need only to be performed locally, thus reducing them to the model case. Therefore our choice of fixing $X=\Sf{n}_{+}$ and working with $SG$-classes below.

Recall that, until now, we have specialised to the sub-classes of classical symbols, in order to study the analytical and geometrical properties of the principal symbol. Here, we specialise further by assuming that the order of the involved operators is $(m_e,m_\psi)\in\Z^2$. There will be only a single exception to this rule, which will be mentioned explicitly. Again, we always omit to write $cl,cl(e),cl(\psi)$ in the corresponding notations in Definition \ref{def:classicalsgsymbols}.

We want to address the question of the order-preserving isomorphisms in the $SG$-setting. Our main object of investigation is the following.
\begin{defin}
	\label{def:sgorderpreserving}
	Consider the algebra $\lg{}$ and an algebra isomorphism (not necessarily topological nor a *-isomorphism) $\imath\colon\lg{}\rightarrow\lg{}$. We say that $\imath$ is an \textit{SG-order preserving isomorphism} (SGOPI) if for any $m\in\Z^2$ it holds true
	\begin{equation}
		\label{eq:sgorderpreserving}
		\imath(\lg{m})\subset\lg{m},
	\end{equation}
	that is, $\imath$ preserves the double filtration on $\lg{}$.
\end{defin}

The approach for this result is very much in line with the original paper \cite{duistermaat1976orderpreservingisomorphisms}, but a number of differences arise, due to the introduction of the second filtration. In particular, we have to work with products of manifolds with boundary, and many of the ideas in \cite{coriasco2019lagrangiandistributions} and \cite{melrose1996scatteringflow}, as we developed further in the previous chapters, are useful.

For later reference, we list various, easy properties of SGOPIs in the following lemma. These are direct algebraic consequences of Definition \ref{def:sgorderpreserving}

\begin{lemma}
	\label{lemma:properties of SGOPI}
	Let $\imath$ be a SGOPI. Then:
	\begin{enumerate}
		\item 
			$\imath$ maps ideals to ideals and in particular maximal ideals to maximal ideals;
		\item 
			$\imath(\RG)=\RG$.
	\end{enumerate}
\end{lemma}

We will need to employ, in the course of the proof of Theorem \ref{theo:SGOPI}, the principle which has come to be known as \textit{Milnor's exercise} (namely Problem 1-C in Section 1 of \cite{milnor1974characteristicclasses}). It is generally presented in the following form:

\begin{quote}
	For a compact smooth manifold $X$, the maximal ideals in $\Cinf(X)$ are given by functions vanishing at a point. Namely, $I\triangleleft \Cinf(X)$ is maximal if and only if $I=I_p\equiv\{f\in\Cinf(X)\sthat f(p)=0\}$ for some $p\in X$.
\end{quote}

A direct corollary is that any algebra isomorphism $F\colon\Cinf(X)\rightarrow\Cinf(Y)$ is induced by a diffeomorphism $C\colon X\rightarrow Y$ via pull-back and, hence, automatically continuous. For our purposes, we need to consider $X,Y$ manifolds with corners and an algebra isomorphism $F\colon\Cinf(\del X)\rightarrow\Cinf(\del Y)$, and ask ourselves the question whether $F$ is also induced by a diffeomorphism $C\colon \del X\rightarrow\del Y$. We state here a slight generalisation of this principle, applicable to certain sub-algebras of continuous functions on a compact topological space $X$. This version has arisen in a discussion with Philipp Schmitt, concerning the minimal conditions which such a subalgebra has to satisfy in order to be able to characterise the maximal ideals. 

\begin{prop}[Milnor's exercise]
	\label{prop:Milnor exercise}
	Let $X$ be a compact topological space and $\Acal\subset\Ccal(X)$ a sub-algebra having the same unit as $\Ccal(X)$. Assume the following:
	\begin{enumerate}
		\item 
			$\Acal$ is spectrally invariant in $\Ccal(X)$, namely, $\Ccal(X)^{-1}\cap \Acal=\Acal^{-1}$, where the superscript $-1$ denotes the group of invertibles;
		\item 
			$\Acal$ is closed under complex conjugation (or simply $\Acal$ consists of real-valued functions).
	\end{enumerate}
	Then, every maximal ideal in $\Acal$ is of the form $I_p$ for some $p\in X$. In particular it has codimension 1.
\end{prop}
\begin{proof}
	Let $I\triangleleft \Acal$. We claim that there exists $p\in X$ such that $f(p)=0$ for every $f\in I$. Arguing by contradiction, assume that for each point $x\in X$ we can find a function $f_x\in I$ with $f_x(x)\neq0$. In particular, by continuity of the elements in $\Acal$ there exists an open cover $\{U_x\}_{x\in X}$ of $X$ where $f_x(y)\neq0$ for all $y\in U_x$. By compactness, we can look at a finite sub-cover $\{U_0,\dots,U_n\}$ associated to the points $x_0,\dots,x_n$ and the functions $f_0,\dots, f_n$. Then for all $i$, $\abs{f_i}^2=\bar{f_i}f_i$ are non-negative elements of $I$ which only vanish (if anywhere) outside $U_i$. The pointwise sum $f=\sum_{i=0}^n \abs{f_i}^2$ is therefore everywhere positive and belongs to $I$. By spectral invariance, $f$ is invertible in $\Acal$ so $I=\Acal$, contradicting our assumption of maximality. The proof is complete.
\end{proof}

Applying this to $\Cinf(B_{sc}X)$ and $\Cinf(B_{sc}Y)$ (algebras which clearly satisfy the conditions above), for two scattering manifolds, $X,Y$ gives that any algebraic isomorphism is induced by a diffeomorphism $B_{sc}X\rightarrow B_{sc}Y$. Notice that there is a little extra structure hidden here: smooth functions of $B_{sc}X$ are actually pairs of smooth functions on manifolds with boundary together with an identification of the boundaries, so this really means that we obtain a triple of compatible diffeomorphisms. In the model case $\B^n\times\B^n$, this excludes directly the possibility of the symplectic rotation of Remark \ref{rem:symplectic rotation as a scattering map}. For the sake of clarity and to push the analogy between scattering and $SG$ as far as possible, we give in Lemma \ref{lemma:canonicaltransformation} an argument adapted to this situation. 

Next, we give a complete proof of (an adaptation of) the spectral argument used in \cite{mathai2017geometrypseudodifferentialalgebra} to exclude the possibility of a skew-symplectic diffeomorphism.

\begin{lemma}
	\label{lemma:spectral argument}
	Let $\imath$ be an SGOPI. Then the following holds true:
	\begin{enumerate}
		\item 
			If $A\in\lg{m}$ is elliptic, $\imath(A)$ is elliptic as well;
		\item 
			Let $A_\bullet$ be a $\bullet$-order reduction, that is $A_\bullet\in\lg{\indi_\bullet}$ is elliptic and self-adjoint, and let $B_\bullet=\imath(A_\bullet)$. If $a_\bullet>0$ and $\Im b_\bullet=0$, we also have $b_\bullet>0$.
	\end{enumerate}
\end{lemma}
\begin{proof}
	\begin{enumerate}
		\item
			Consider a parametrix $R$ of $A$. Then there exist operators $K_1,K_2\in\RG$ such that $AR-\delta=K_1,RA-\delta=K_2$. Applying $\imath$ to these relations gives immediately that $\imath(R)$ is a parametrix of $\imath(A)$, therefore $\imath(A)$ is elliptic.
		\item
			First notice that the assumption of self-adjointness is not really restrictive, since any elliptic operator with positive symbol is equal to a self-adjoint one modulo lower order operators. Therefore, assume $A_\bullet\in\lg{\indi_\bullet}$ has the required properties and let $B_\bullet=\imath(A_\bullet)$. By the previous point, $B_\bullet$ is elliptic as well, so its symbol can be either positive or negative everywhere in view of the assumption $b_\bullet\in\R$. Assume, arguing by contradiction, that $b_\bullet<0$. We can then find an operator $N_\bullet\in\lg{0}$ so that $B_\bullet=B_\bullet^W+N_\bullet$ where $B_\bullet^W$ is the Weyl operator associated with $b_\bullet$. Notice, in particular, that $B_\bullet$ is a bounded perturbation of its Weyl counterpart. By assumption, $B_\bullet^W$ is unbounded self-adjoint and has real spectrum bounded from above. We can also estimate 
			\begin{equation}
				\label{eq:perturbed spectrum}
				\mathrm{spec}(B_\bullet)\subset\{\lambda\in\C\sthat\mathrm{dist}(\lambda,\mathrm{spec}(B_\bullet^W+N_\bullet))\leq\norm{}{N_\bullet}\}.
			\end{equation}
			We conclude that there exists a constant $K_\bullet\in\R$ such that $B_\bullet^W+N_\bullet-t$ is invertible for all $t\notin (-\infty,K_\bullet]\times[-\norm{}{N_\bullet},\norm{}{N_\bullet}]$, with inverse being an operator lying in $\lg{-\indi_\bullet}$. Let now $M_\bullet=\imath^{-1}(N_\bullet)$ and consider $A_\bullet+M_\bullet-t$, which has to be invertible with inverse in $\lg{-\indi_\bullet}$ for the same $t$'s. The spectrum of $A_\bullet$ is real and bounded from below since $A_\bullet$ is positive, moreover $M_\bullet$ is bounded, so that, just like in \eqref{eq:perturbed spectrum}, the spectrum of $A_\bullet+M_\bullet$ is unbounded but contained in a tube $[D_\bullet,+\infty)\times[-\norm{}{M_\bullet},\norm{}{M_\bullet}]$ for some $D_\bullet\in\R$. Then, $A_\bullet+M_\bullet-t$ is invertible for all $t$ outside this set, but, at the same time, there exists at least one $\tilde{t}_\bullet\in\spec(B_\bullet^W+N_\bullet)$ for which $A_\bullet+M_\bullet-\tilde{t}_\bullet$ cannot be invertible, since the spectra are unbounded. This is a contradiction. We conclude that $b_\bullet$ has to be positive as well, completing the proof.	
	\end{enumerate}
\end{proof}

\subsection{The case of the formal symbol algebra}

We begin our investigation with the formal symbol algebra $\BG{}=\setquotient{\lg{}}{\RG}\cong\setquotient{\sg{}}{\sg{-\infty\indi}}$. At this level, we can exploit the explicit relation between (asymptotic expansions of) symbols and operators.

\begin{lemma}
	\label{lemma:canonicaltransformation}
	Given an SGOPI $\imath$, there exists a scattering canonical transformation $C\colon\del(\B^n\times\B^n)\rightarrow\del(\B^n\times\B^n)$ such that for all $(a_\bullet)\in\SymG{m}$ it holds true $\imath(a_\bullet)=a_\bullet\circ C_\bullet^{-1}$. 
\end{lemma}

\begin{proof}
	The algebraic properties of $\imath$ guarantee that
	\[
	\imath\left(\lg{m_1,m_2}\diagup\RG \right)=\lg{m_1,m_2}\diagup\RG.
	\]
	$\imath$ also acts on the space of principal symbols. Indeed, it preserves $\sg{-\indi_e}$, $\sg{-\indi_\psi}$, $\sg{-\indi_e}\oplus\sg{-\indi_\psi}$ as ideals in $\sg{0}$, so it descends to a map $\imath_{pr}$ on $\SymG{0}\cong\Cinf(\del(\Sf{n}_+\times\Sf{n}_+))$, whose elements can be identified with pairs of functions $(a_e,a_\psi)$ on the respective (open) boundary face $\widetilde{\Wcal}_e=\Sf{n-1}\times \R^n, \widetilde{\Wcal}_\psi=\R^n\times\Sf{n-1}$, having the same ``limit'' in the corner $\widetilde{\Wcal}_{\psi e}=\Sf{n-1}\times\Sf{n-1}$. That is, they extends smoothly to the whole $\del(\Sfp\times\Sfp)$. We have then the commutative diagram
	\begin{equation}
		\label{eq:OPI commutes with principal symbol}
		\begin{CD}
			\lg{0}				@>{\imath}>>		\lg{0}\\
			@V{\sigma_{pr}}VV					@V{\sigma_{pr}}VV\\
			\SymG{0}			@>{\imath_{pr}}>>	\SymG{0},
		\end{CD}
	\end{equation}
	meaning that we have maps $\imath_\bullet$ satisfying
	\begin{equation}
		\label{eq:SGOPI commutes with principal symbol}
		\sigma_{pr}(\imath(A))=(\imath_e a_e,\imath_\psi a_\psi,\imath_{\psi e}a_{\psi e}).
	\end{equation}
	In view of the multiplicative properties of $\sigma_{pr}$ in Proposition \ref{prop:principalsgsymbols} we see that the maps $\imath_\bullet$ are multiplicative on the respective spaces. We can then apply a Milnor-type argument to obtain bijections of $\widetilde{\Wcal}_\bullet$. Let $I^\bullet$ be a maximal ideal in $\Cinf(\widetilde{\Wcal}_\bullet)$. This is given by the set $I^\bullet_{p^\bullet}$ of those functions on $\widetilde{\Wcal}_\bullet$ which vanish at $p^\bullet\in\widetilde{\Wcal}_\bullet$. Then, $\imath_\bullet$ gives a correspondence $\chi_\bullet\colon\widetilde{\Wcal}_\bullet\rightarrow\widetilde{\Wcal}_\bullet$ defined by $\imath_\bullet(I^\bullet_{p^\bullet})=I^\bullet_{\chi_\bullet(p^\bullet)}$. 
	
	We may repeat the same argument with $\imath^{-1}$ to obtain another triple of bijections $\zeta_\bullet\colon\widetilde{\Wcal}_\bullet\rightarrow\widetilde{\Wcal}_\bullet$. By writing $I^\bullet_p=\imath_\bullet^{-1}\imath_\bullet(I^\bullet_p)$ we find then that $\zeta_\bullet=\chi_\bullet^{-1}$. Furthermore, we see that it holds true 
	\begin{equation}
		\label{eq:symbol image compose with diffeo}
		\imath_\bullet a_\bullet=a_\bullet\circ\chi_\bullet^{-1};
	\end{equation}
	indeed for $a_\bullet\in\Cinf(\widetilde{\Wcal}_\bullet)$ we have $a_\bullet-a_\bullet(p^\bullet)1\in I^\bullet_{p_\bullet}$, hence $\imath_\bullet(a_\bullet)(\chi_\bullet(p^\bullet))-a_\bullet(p^\bullet)1=0$, as claimed. We remark here in addition that the identification of principal symbols of order $0$ with smooth (in the sense of Remark \ref{rem:smooth functions across the corner}) functions on $\del(\B^n\times\B^n)$ is \textit{canonical}, since it does not depend on the choice of a boundary defining function.
	
	These maps must be smooth. To see this, for example, for $\bullet =e$, it suffices to choose local coordinates $(\theta^i,\xi_j)$ on $\widetilde{\Wcal}_e$ near a point $p_e$. By definition, these coordinates are smooth function on $\Sf{n-1}\times\R^n$, so they can be identified canonically with homogeneous symbols $\tilde\theta^i,\tilde\xi_j$ of order $0$, namely, with elements of $\sg{(0),1}$. We can then apply $\imath$ to obtain
	\begin{equation}
		\label{eq:smooth bijections}
		\begin{aligned}
			\theta^i\circ\chi_e^{-1}&=\imath(\tilde{\theta}_i)\in\sg{(0),1},\\
			\xi_j\circ\chi_e^{-1}&=\imath(\tilde{\xi}_j)\in\sg{(0),1}.
		\end{aligned}
	\end{equation}
	It follows that the components of $\chi_e^{-1}$ are smooth functions by composition, so $\chi_e^{-1}$ is smooth itself. The same argument, using $\imath_e^{-1}$, gives that $\chi_e$ is actually a diffeomorphism.
	
	Having determined the action of $\imath$ on principal symbols of order $(0,0)$, we use order reductions to extend it to $\SymG{m}$ for any $m\in\Z^2$. Namely, recalling Lemma \ref{lemma:classical order reductions} we write $A\in\lg{m}$ as
	\begin{equation}
		\label{eq:reduced order operator}
		A=P^{m_e}Q^{m_\psi}B
	\end{equation}
	for $B\in\lg{0}$ and $P$, resp. $Q$, an $e$-order reduction, resp. a $\psi$-order reduction. Thus the image of $A$ via $\imath$ can be computed as
	\begin{equation}
		\label{eq:image reduced order operator}
		\imath(A)=\tilde{P}^{m_e}\tilde{Q}^{m_\psi}\imath(B),
	\end{equation}
	where $\tilde P=\imath(P)\in\lg{\indi_e}, \tilde{Q}=\imath(Q)\in\lg{\indi_\psi}$ are elliptic and $\imath(B)$ is an operator of order $(0,0)$. To determine the action of $\imath$ on principal symbols of any order it suffices then to describe it on the order reductions. Looking at the pairs picture of principal symbols (cf. Proposition \ref{prop:principalsgsymbols}), we see that $\sigma_{pr}(\imath(B))=(b_e\circ \chi_e^{-1},b_\psi\circ \chi_{\psi}^{-1})$ and
	\begin{equation}
		\label{eq:principal symbol image reduced operator}
		\sigma_{pr}(\imath(A))=(\tilde{p}_e^{m_e}\tilde{q}_e^{m_\psi}b_e\circ \chi_e^{-1},\tilde{p}_\psi^{m_e}\tilde{q}_\psi^{m_\psi}b_\psi\circ \chi_\psi^{-1})
	\end{equation}
	for $(\tilde{p}_e,\tilde{p}_\psi)$ and $(\tilde{q}_e,\tilde{q}_\psi)$ the principal symbols of $\tilde{P}$ and $\tilde{Q}$, respectively. In particular we analyse closely the case $m=\indi$, for which we know that $\SymG{\indi}$ is a Lie algebra by Lemma \ref{prop:symplectic properties of sg}.
	
	For any $a,\alpha\in\sg{\indi}$ consider the relation $\imath\{\sigma_{pr}(\alpha), \sigma_{pr}(a)\}=\sigma_{pr}(\{\imath(\alpha),\imath(a)\})$ and write $a$ and $\alpha$ as in \eqref{eq:reduced order operator}, namely
	\begin{equation}
		\label{eq:decomposition of symbols}
		\begin{aligned}
			a&=pqb,\\
			\alpha&=pq\beta.
		\end{aligned}
	\end{equation}
	Denoting $r=pq, \tilde{r}=\imath(r)$, we pass to principal symbols and look at the single components. Let us work out in detail what happens for $\bullet=e$, since for the case $\bullet=\psi$ the same proof suffices up to an exchange in the homogeneities, and the exit behaviour is the main novelty here. The above relation for the Poisson brackets reads
	\begin{equation}
		\label{eq:poisson bracket reduced symbols}
		\begin{aligned}
		\{\tilde{r}_e b_e\circ{\chi_e^{-1}},\tilde{r}_e\beta_e\circ\chi_e^{-1}\}&=\imath_e(\{r_e b_e,r_e\beta_e\})\\
		&=\imath_e(r_e)\imath_e(r_e)^{-1}\imath(\{r_e b_e,r_e\beta_e\})\\
		&=\tilde{r}_e\imath(r_e^{-1}\{r_e b_e,r_e\beta_e\}).
		\end{aligned}
	\end{equation}
	In particular, if we choose $\beta_e=1$, we obtain that, for any $b_e$,
	\begin{equation}
		\label{eq:poisson bracket reduced symbols 2}
		\{\tilde{r}_e,b_e\circ \chi^{-1}_e\}=\{r_e,b_e\}\circ \chi^{-1}_e.
	\end{equation}
	We claim now that we can extend the map $\chi_e$ homogeneously to a map $C_e$ so that $\tilde{r}_e=r_e \circ C_e^{-1}.$ This choice will give that $\{r_e\circ C_e^{-1},b_e\circ C_e^{-1}\}=\{r_e,b_e\}\circ C_e^{-1}$, so that going back to \eqref{eq:poisson bracket reduced symbols} we find that for any two symbols $a,\alpha$ of order $\indi$ it holds true 
	\begin{equation}
	    \label{eq:poisson bracket is symplectic}
	    \{\alpha_e\circ C_e^{-1},a_e\circ C_e^{-1}\}=\{\alpha_e,a_e\}\circ C_e^{-1}.
	\end{equation}
	This is equivalent to $C_e$ being a canonical transformation.
	
	To see that we can indeed make such a choice, recall that homogeneous maps in $\R^n_0\times\R^n$ are written as in \eqref{eq:equivariant extension 1} for some $f_e$ real-valued and smooth on $\Sf{n-1}\times\R^n$. We now define $C_e$ to be the 1-homogeneous extension of $\chi_e$ given by
	\begin{equation}
		 \label{eq:e-homogeneous extension}
		 C_e(\rho_1,\theta,\xi)=\left(\frac{{p}_e(1,\theta,\xi)}{\tilde{p}_e(1,\chi_e(\theta,\xi))}\rho_1,\chi_e(\theta,\xi)\right).
	\end{equation}
	Recall that here $\theta$ are coordinates on $\Sf{n-1}$, $\rho_1\in\R^+$ and $\xi\in\R^n$. This choice satisfies $\tilde{r}_e=r_e\circ C_e^{-1}$. Indeed, notice that both $q_e$ and $\tilde{q}_e$ do not play a rôle here since they are $e$-homogeneous of degree $0$. More precisely, $q_e(1,\theta,\xi)=\tilde{q}_e(1,\chi_e(\theta,\eta))$. Keeping this in mind, and denoting $\chi_e(\theta,\xi)=(\phi,\eta)$ with $\sigma_1\in\R^+$ the newly introduced coordinate in the target space, the check is immediate, using the $e$-homogeneity of $p_e$ and $\tilde{p}_e$:
	\begin{equation}
	    \label{eq:extension preserves order reduction}
	    \begin{aligned}
	    \tilde{r}_e(\sigma_1,\phi,\eta)&=\tilde{p}_e(\sigma_1,\phi,\eta)\tilde{q_e}(1,\phi,\eta)\\
	    r_e\circ C_e^{-1}(\sigma_1,\phi,\eta)&=r_e\left(\frac{{p}_e(1,\chi_e^{-1}(\phi,\eta)}{\tilde{p}_e(1,\phi,\eta)}\sigma_1,\phi,\eta\right)\\
	    &=\frac{{p}_e(1,\chi_e^{-1}(\phi,\eta)}{\tilde{p}_e(1,\phi,\eta)}\sigma_1\cdot p_e(1,\chi_e^{-1}(\phi,\eta))q_e(1,\chi_e^{-1}(\phi,\eta))\\
	    &=\tilde{p}(\sigma_1,\phi,\eta)q_e(1,\chi_e^{-1}(\phi,\eta)).
	    \end{aligned}
	\end{equation}
	With this choice, $C_e$ is then a homogeneous diffeomorphism $\Wcal_e\rightarrow\Wcal_e$, preserving the Poisson bracket for $e$-principal symbols. 
	
	Having constructed the required extensions, it follows that $\chi$ is a scattering canonical transformation. The proof is complete.
\end{proof}

By Theorem \ref{theo:parametrisation sc-symplectomorphism}, we can locally parametrise the graph of $\chi$ via $SG$-phase functions of order $\indi$ and type $\Qcal_{gen}$. Covering $\graph\chi$ with these coordinate patches and picking local amplitudes, we can construct an elliptic FIO $F$ of type $\Qcal_{gen}$ associated with $\chi$, in the sense that its principal symbol can be identified with a function on the graph. If we denote by $F^\sharp$ the parametrix of $F$, which is again elliptic of type $\Qcal_{gen}$ and is associated with the inverse map $\chi^{-1}$, we have then that $\jmath(P)\equiv FPF^\sharp$ is an automorphism of $\setquotient{\lg{}}{\RG}$ preserving principal symbols, in view of Theorem \ref{theo:SG Egorov}. Our next goal is to analyse $\jmath$ more closely by mirroring the argument of \cite{duistermaat1976orderpreservingisomorphisms} and refining it to $SG$-classes. Before we start with that task, we need to prove an auxiliary result, adapted from Theorem 2.2.10 in \cite{abraham2008foundations}.
\begin{lemma}
	\label{lemma:derivations are vector fields}
	Consider $X=\R^n_0\times\R^k$ and the space $\Hcal^{m,l}(X)$ of functions of $(x,y)\in X$ being positively homogeneous of degree $m$ in $x$ and $y-$classical of degree $l$. $\Hcal=\bigcup\Hcal^{m,l}$ is a filtered algebra. Consider a derivation $\theta\colon\Hcal\rightarrow\Hcal$. Then there exists $V\in\VF(X)$ such that $\theta=\mathcal L_V$ as a derivation. 
\end{lemma}
\begin{proof}
	Let us make a few remarks to begin with. First, given a vector field $V\in\VF(X)$, the Lie derivative $\Lcal_V$ is well defined for functions in $\Hcal\subset\Cinf(X)$. In particular, it is a derivation. However, for a general $V$, there is no guarantee that $V(\Hcal)\subset\Hcal$ since the local expressions of the coefficients of $V$ need not be $y-$classical. 
	
	Second, $\theta$ is a local operator on $\Hcal$, i.e.\ if $a\in\Hcal, (x_0,y_0)\intern U$ and $a|_U=0$, then $\theta(a)(x_0,y_0))=0$. To see this, pick a smooth cut-off function $g$ such that $g=1$ on $U\neigh(x_0,y_0), U\subset V$ and $g=0$ outside $V$. It follows that $a=(1-g)a$ everywhere, so by the derivation property and the assumptions we have
	\begin{equation}
		\label{eq:derivations are local}
		\begin{aligned}
			\theta(a)(x_0,y_0)&=\theta(a)(x_0,y_0)(1-g(x_0,y_0))-\theta(g)(x_0,y_0)a(x_0,y_0)\\
			&=0,
		\end{aligned}
	\end{equation}
	as required.
	
	We can now proceed to the main part of the proof. The locality property implies that we can define restrictions of $\theta$ to open subsets $(x,y)\intern V\subset X$ by
	\begin{equation}
		\label{eq:restrictions of derivations}
		\theta|_V(a)(x,y)\equiv\theta(ga)(x,y),
	\end{equation}
	where $g$ is a smooth cut-off such that $g=0$ outside $V$ and $g=1$ on some $(x,y)\intern U\subset V$. Again by locality, it follows that $\theta|_V$ doesn't actually depend on the choice of such a cut-off. We keep denoting by $\theta$ the restrictions to open subsets.
	
	Pick a chart $(U,\rho)$ on $X$ with coordinates $(x^i,y^\alpha)$, let $p\in U$ and $a\in\Hcal$. Assume $\rho(p)=q=(x_0^i,y_0^\alpha)$. Then, we can write, in a sufficiently small $W\neigh q$ and denoting $c_q(t)\equiv q+t(x^i-x^i_0,y^\alpha-y^\alpha_0)$
	\begin{equation}
		\label{eq:local first order expansion}
		\begin{aligned}
			(\push{\rho}a)(x^i,y^\alpha)&=(\push{\rho}a)(q)+\int_0^1\partder{}{t}[(\push\rho a)(c_q(t))]\de t\\
			&=(\push\rho a)(q)+(x^i-x^i_0)\int_0^1\partder{\push\rho a}{x^i}(c_q(t))\de t+(y^\alpha-y^\alpha_0)\int_0^1\partder{\push\rho a}{y^\alpha}(c_q(t))\de t.
		\end{aligned}
	\end{equation}
	Let $U'=\rho^{-1}(W)\neigh p$ and $u\in U'$. There exist then functions $g_i,g_\alpha\in\Cinf(U)$ such that 
	\begin{equation}
		\label{eq:preimage of partial derivatives}
		g_i(p)=\partder{\push\rho a}{x^i}(q),\quad g_\alpha(p)=\partder{\push\rho a}{y^\alpha}(q)
	\end{equation}
	and 
	\begin{equation}
		\label{eq:preimage first order expansion}
		a(u)=a(p)+(x^i-x^i_0)g_i(u)+(y^\alpha-y^\alpha_0)g_\alpha (u).
	\end{equation}
	We apply $\theta$ to \eqref{eq:preimage first order expansion} to obtain
	\[
	\theta(a)(u)=\theta(x^i)(u)g_i(u)+(x^i-x^i_0)\theta(g_i)(u)+\theta(y^\alpha)(u)g_\alpha(u)+(y^\alpha-y^\alpha_0)\theta(g_\alpha)(u),
	\]
	so that, evaluating this expression at $p$ and using \eqref{eq:preimage of partial derivatives}, we have
	\begin{equation}
		\label{eq:derivation evaluated at p}
		\theta(a)(p)=\theta(x^i)(p)\partder{\push\rho a}{x^i}(p)+\theta(y^\alpha)(p)\partder{\push\rho a}{y^\alpha}(p).
	\end{equation}
	It is readily seen that a change of coordinates does not affect this expression. We define then a vector field $V_\rho$ on $U$ by setting
	\begin{equation}
		\label{eq:vector field associated to a derivation}
		V_\rho(x^i,y^\alpha)\equiv\left((x^i,y^\alpha),(\theta(x^i)(u),\theta(y^\alpha)(u))\right)
	\end{equation}
	for $\rho(u)=(x^i,y^\alpha)$. It follows that $V_\rho|_U$ is independent of the chart, so that the collection of these objects defines a vector field $V\in\VF(X)$. Now, by definition, the Lie derivative with respect to $V$ in a local chart $(U,\rho)$ of $a\in\Hcal$ is
	\begin{equation}
		\label{eq:lie derivative equals derivation}
		\begin{aligned}
			\Lcal_V a|_U&=D(a\circ\rho^{-1})(x^i,y^\alpha)V_\rho(x^i,y^\alpha)\\
			&=\partder{}{x^i}(a\circ\rho^{-1})(x^i,y^\alpha)\theta(x^i)(x^i,y^\alpha)+\partder{}{y^\alpha}(a\circ\phi^{-1})(x^i,y^\alpha)\theta(y^\alpha)(x^i,y^\alpha)\\
			&=\theta(a)(u),
		\end{aligned}
	\end{equation}
	and that the claim follows. The proof is complete.
\end{proof}
\begin{rem}
	This lemma shows that it suffices to have a derivation on (certain) subalgebras of $\Cinf(\R^n_0\times\R^n)$ to determine a vector field on $\R^n_0\times\R^n$. We will find use for this fact in the proof of Lemma \ref{lemma:automorphism at order l} below. In addition, we will see that the properties of the subalgebra (in this case, homogeneity in $x$ and classicality of order $0$ in $y$) are reflected in the properties of the coefficients of the obtained vector field. 
\end{rem}
\begin{rem}
	Notice that Duistermaat and Singer do not need this specialised result. Indeed in their setting they obtain derivations on the whole $\Cinf(\Sf{\ast}X)$, which are given by vector fields by the standard theory. 
\end{rem}

We can now begin our analysis of the map $\jmath$, which will lead to the following first main result.
\begin{theo}
	\label{theo:automorphism is given by elliptic conjugation}
	Assume given an automorphism $\jmath\colon\setquotient{\lg{}}{\RG}\rightarrow\setquotient{\lg{}}{\RG}$ preserving principal symbols, namely $\jmath(P)-P\in\setquotient{\lg{m-\indi}}{\RG}$ whenever $P\in\lg{m}$. Then, $\jmath$ is given by conjugation with some elliptic $B\in\lg{s}$ for some $s\in\C^2$.
\end{theo}

For the sake of clarity, we split the proof into a series of lemmata. 
\begin{lemma}
	\label{lemma:automorphism at order l}
	Assume that, for some $l\geq 1$, we have $\jmath(P)-P\in\lg{m-l\indi}$ for any $P\in\lg{m}$ with principal symbol $(p_e,p_\psi)$. Then $\sigma_{pr}(\jmath(P)-P)$ only depends on the principal symbol of $P$ and it is obtained as $(\beta_e p_e,\beta_\psi p_\psi)$ for two vector fields $\beta_e,\beta_\psi$. Moreover these vector fields are Hamiltonian, that is, there exist functions $f_e,f_\psi$ such that $\beta_\bullet p_\bullet=H_{f_\bullet}p_\bullet\equiv\{f_\bullet,p_\bullet\}$.
\end{lemma}
\begin{proof}
	For a fixed $l\geq 1$ consider, for all $m\in\Z$, the map $Z^m\colon\lg{m}\rightarrow \setquotient{\lg{m-l\indi}}{\lg{m-(l+1)\indi}}$, $Z^m(P)\equiv \jmath(P)-P\mod\lg{m-(l+1)\indi}$. $Z^m$ only depends on the principal symbol of $P$. Indeed, if $Q=P+W$ for some $W\in\lg{m-\indi}$, it follows that 
	\begin{equation}
		\label{eq:Z principal symbol}
		\begin{aligned}
			Z^m(Q)&=Z^m(P)+Z^m(W)\\
			&=\jmath(P)-P+\jmath(W)-W\mod \lg{m-(l+1)\indi}\\
			&=Z^m(P)
		\end{aligned}
	\end{equation}
	since, by assumption, $\jmath(W)-W\in\lg{m-(l+1)\indi}$. Hence, by composition with the principal symbol map, $Z^m$ descends to a map $\beta^m\colon\SymG{m}\rightarrow\SymG{m-le}$. We show that $\beta^m$ (or rather the direct sum $\beta=\bigoplus_{m\in\Z^2}\beta^m$) is a bi-derivation of the bi-algebra $\SymG{}$. To this end, consider $Z(PQ)$ for some operators $P\in\lg{m},Q\in\lg{k}$. Recalling the algebraic properties of $\jmath$, we have by definition,
	\[\begin{aligned}
		Z^{m+k}(PQ)&=\jmath(PQ)-PQ \mod\lg{m+k-l\indi}\\
		&=\jmath(P)\jmath(Q)-\jmath(P)Q+\jmath(P)Q-PQ\mod\lg{m+k-l\indi}\\
		&=\jmath(P)Z^k(Q)+Z^m(P)Q\mod\lg{m+k-l\indi}\\
		&=Z^m(P)Z^k(Q)+PZ^k(Q)+Z^m(P)Q\mod\lg{m+k-l\indi}\\
		&=PZ^k(Q)+Z^m(P)Q\mod \lg{m+k-l\indi},
	\end{aligned}\]
	where we noticed that $Z^m(P)Z^k(Q)\in\lg{m-l\indi}\cdot\lg{k-l\indi}\subset\lg{m+k-2l\indi}\subset\lg{m+k-l\indi}$. Taking principal symbols gives the Leibniz rule. Similarly, for $Z[P,Q]$ we obtain
	\[
		Z^{m+k-\indi}[P,Q]=[Z^m(P),Q]+[P,Z^k(Q)]\mod\lg{m+k-(l+1)\indi},
	\]
	so that $\beta$ is a derivation with respect to the Poisson bracket. Therefore, $\beta$ acts as a bi-derivation on the space of principal symbols $\SymG{}$. Keeping in mind the pairs picture of Proposition \ref{prop:principalsgsymbols}, denote the action of $\beta$ as
	\begin{equation}
		\label{eq:components of beta}
		\beta(p_\psi,p_e)=(\beta_\psi p_\psi,\beta_ep_e).
	\end{equation}
	Similarly to \eqref{eq:Z principal symbol}, one sees that $\beta_\psi p_\psi$, respectively $\beta_e p_e$, only depends on the component $p_\psi$, respectively $p_e$, of the principal symbol, and it holds true that $\sigma_{\psi}^{m_\psi-l}(\beta_ep_e)=\sigma_{e}^{m_e-l}(\beta_\psi p_\psi)$. We can write, more explicitly,
	\begin{equation}
		\label{eq:beta components associated symbol}
		\begin{aligned}
			\beta_\psi p_\psi&=\sigma_\psi^{m_\psi-l}(\beta \check {p})\\
			\beta_e p_e&=\sigma_e^{m_e-l}(\beta \check{p}).
		\end{aligned}
	\end{equation}
	Applying Lemma \ref{lemma:derivations are vector fields}, we obtain that both $\beta_\psi$ and $\beta_e$ are given by vector fields on $\R^n\times\R^n_0$ and $\R^n_0\times \R^n$, respectively. We have then
	\begin{equation}
		\label{eq:beta vector field}
		\begin{aligned}
			\beta_\psi&=\gamma^i_\psi\partder{}{x^i}+\rho^\psi_k\partder{}{\xi_k},\\
			\beta_e&=\gamma^i_e\partder{}{x^i}+\rho^e_k\partder{}{\xi_k},
		\end{aligned}
	\end{equation}
	where, by definition, $\gamma_\bullet^i=\beta_\bullet x^i,\rho^\bullet_k=\beta_\bullet \xi_k$. In particular, by definition of $\beta$, it holds true that
	\begin{equation}
		\label{eq:vector fields have classical coefficients}
		\begin{aligned}
			\gamma^i_\psi(x,\xi)&=\beta_\psi x^i\in \sg{(-l),1-l}_{cl(x)},\\
			\rho_k^\psi(x,\xi)&=\beta_\psi\xi_k\in\sg{(1-l),-l}_{cl(x)},\\
			\gamma_e^i(x,\xi)&=\beta_e x^i\in\sg{-l,(1-l)}_{cl(\xi)},\\
			\rho^e_k(x,\xi)&=\beta_e \xi_k\in\sg{1-l,(-l)}_{cl(\xi)},
		\end{aligned}
	\end{equation}
	so that the components of the obtained vector fields mirror the extra properties of the algebra of functions from which they are derived.
	Recall also that $\{x^i,x^j\},\{x^i,\xi_j\}$ and $\{\xi_i,\xi_j\}$ are all constant, hence, if we apply $\beta_\bullet$, to them we obtain 0. On the other hand, by using the derivation property, we see that it must hold true that
	\begin{equation}
		\label{eq:properties of coefficients of beta}
		\partder{\gamma^i_\bullet}{\xi_j}=\partder{\gamma^j_\bullet}{\xi_i},\quad
		\partder{\rho^\bullet_i}{x^j}=\partder{\rho^\bullet_j}{x^i},\quad
		\partder{\gamma_\bullet^i}{x^j}=-\partder{\rho^\bullet_j}{\xi^i}.
	\end{equation}
	But this is the same as saying that the (symplectic) dual 1-form to $\beta_\bullet$ is closed, hence locally equal to $\de f_\bullet$ for some smooth $f_\bullet$. Locally, we have then shown $\beta_\bullet=H_{f_\bullet}$, the Hamiltonian vector field defined by $f_\bullet$. The proof is complete.
\end{proof}

The next step consists in establishing under which conditions $\jmath$ is given by conjugation with an $SG\Psi$DO at the level of principal symbols.
\begin{lemma}
	\label{lemma:automorphism at principal symbol level}
	There exist $B\in\lg{s}$, $s\in\C^2$, such that $\sigma_{pr}(\jmath(P))=\sigma_{pr}(BPB^{-1})$ if and only if $\sigma_{pr}(B)=(e^{-\im f_e},e^{-\im f_\psi})$ for functions $f_\bullet$ such that $\beta_\bullet=H_{f_\bullet}$.
\end{lemma}
\begin{proof}
	Notice first that $BPB^{-1}-P=[B,P]B^{-1}$, so that taking principal symbols yields
	\begin{equation}
		\label{eq:principal symbol conjugation 2}
		\begin{aligned}
			\sigma_{pr}(BPB^{-1}-P)&=\sigma_{pr}([B,P])\sigma_{pr}(B^{-1})\\
			&=\left(\frac{1}{\im b_{e}}\{b_e,p_e\},\frac{1}{\im b_\psi}\{b_\psi,p_\psi\}\right)\\
			&=\left(H_{\im\log b_e}(p_e),H_{\im\log b_\psi}(p_\psi)\right).
		\end{aligned}
	\end{equation}
	Remark that, for an invertible $b$, both logarithms exist. Then applying Lemma \ref{lemma:automorphism at order l} with $l=1$ gives that $\sigma_{pr}(BPB^{-1}-P)=\sigma_{pr}(\jmath(P)-P)=(H_{f_e}(p_e),H_{f_\psi}(p_\psi))$ for some smooth $f_e,f_\psi$ if and only if
	\begin{equation}
		\label{eq:principal symbol of B is homogeneous}
		\left\{\begin{aligned}
			b_\psi&=e^{-\im f_\psi}\\
			b_e&=e^{-\im f_e}
		\end{aligned}\right.
	\end{equation}
	are homogeneous in $\xi$, respectively $x$, of degree $s_\psi$, respectively $s_e$. Using Euler's equation, we can rewrite this as
	\begin{equation}
		\label{eq:Euler equation principal symbol}
		\left\{\begin{aligned}
			\xi_k\partder{f_\psi}{\xi_k}&=\im s_\psi,\\
			x^j\partder{f_e}{x^j}&=\im s_e.
		\end{aligned}\right.
	\end{equation}
	Recalling that $\del_{\xi_k}{f_\bullet}$ and $\del_{x^j}{f_e}$ are, respectively, the components $\gamma_\psi^k$ and $\rho_j^e$ of the vector fields $\beta_\bullet$, we see that the claim is equivalent to $\xi_k\gamma^k_\psi$ and $x^j\rho_j^e$ being constant. We check this directly by computing the derivatives of these expressions w.r.t. $x^r$ and $\xi_r$. Considering, for instance, the vector field $\beta_\psi$, we have $\del_{x^j}{\gamma_\psi^i}=-\del_{\xi^i}{\rho^\psi_j}$ and $\del_{\xi_j}{\gamma^i_\psi}=\del_{\xi_i}{\gamma^j_\psi}$, in view of the symmetry relations of \eqref{eq:properties of coefficients of beta}. Then, recalling the homogeneities of \eqref{eq:vector fields have classical coefficients} and that $l=1$, we obtain
	\begin{equation}
		\label{eq:partial derivatives homogeneity relations}
		\begin{aligned}
			\partder{(\xi_k\gamma^k_\psi)}{x^r}&=\xi_k\partder{\gamma^k_\psi}{x^r}
			=-\xi_k\partder{\rho^\psi_r}{\xi_k}=0,\\
			\partder{(\xi_k\gamma^k_\psi)}{\xi_r}&=\gamma^r_\psi+\xi_k\partder{\gamma^k_\psi}{\xi_r}
			=\gamma^r_\psi+\xi_k\partder{\gamma^r_\psi}{\xi_k}=0.
		\end{aligned}
	\end{equation}
	In a completely analogous fashion one proves the corresponding results for $\beta_e$. This concludes the proof.
\end{proof}

\begin{lemma}
	\label{lemma:automorphism for l>1}
	Assume that for some $l>1$ we have $\jmath(P)-P\in\lg{m-l\indi}$ for any $P\in\lg{m}$. Then there exist $C\in\lg{(1-l)\indi}$ such that $(I-C)\circ\jmath(P)\circ(I-C)^{-1}-P\in\lg{m-(l+1)\indi}$.
\end{lemma}
\begin{proof}
	We apply Lemma \ref{lemma:automorphism at order l} again by defining directly the Hamiltonian functions of the vector fields. Recall that $\gamma^j_\bullet$ and $\rho_k^\bullet$ are the components of the vector field $\beta_\bullet$ determining the action of $\jmath$ on principal symbols. Using these functions, set
	\begin{equation}
		\label{eq:Hamiltonian for l>1}
		\left\{\begin{aligned}
			c_\psi&\equiv\frac{1}{1-l}\xi_j\gamma^j_\psi\in\sg{(1-l),1-l}_{cl(x)},\\
			c_e&\equiv\frac{1}{l-1}x^j\rho^e_j\in\sg{1-l,(1-l)}_{cl(\xi)}.
		\end{aligned}\right.
	\end{equation}
	With these definitions, we see that $H_{c_\bullet}=\beta_\bullet$. Indeed, for instance
	\begin{equation}
		\label{eq:hamiltonian vector field for l>1}
		\begin{aligned}
			\partder{c_\psi}{x^k}&=\frac{1}{1-l}\xi_j\partder{\gamma^j_\psi}{x^k}=\frac{1}{l-1}\xi_j\partder{\rho^\psi_k}{\xi_j}\\
			&=\frac{1}{l-1}(1-l)\rho^\psi_k=-\rho^\psi_k,\\
			\partder{c_\psi}{\xi_k}&=\frac{1}{1-l}\left[\gamma^k_\psi+\xi_j\partder{\gamma^j_\psi}{\xi_k}\right]\\			&=\frac{1}{1-l}\left[\gamma^k_\psi+\xi_j\partder{\gamma_\psi^k}{\xi_j}\right]=\frac{1}{1-l}\left[\gamma^k_\psi-l\gamma^k_\psi\right]\\
			&=\gamma^k_\psi,
		\end{aligned}
	\end{equation}
	where we have used again \eqref{eq:vector fields have classical coefficients} and \eqref{eq:properties of coefficients of beta}. Wit the aim of assiociating an operator with $c=(c_e,c_\psi)$, we verify that, indeed, $c\in\SymG{(1-l)\indi}$, namely, that $\sigma^{1-l}_e(c_\psi)=\sigma_\psi^{1-l}(c_e)$. Computing these symbols, we have to prove that
	\begin{equation}
		\label{eq:c is a principal symbol}
		\xi_j\gamma^j_{\psi,1-l}=-x^j\rho_j^{e,1-l},
	\end{equation}
	where $\gamma^j_{\psi,1-l}$, respectively $\rho_j^{e,1-l}$, is the $(1-l)$-homogeneous component in the asymptotic expansion of $\gamma^j_\psi$, respectively $\rho_j^e$. To this end, consider the third relation in \eqref{eq:properties of coefficients of beta}, multiply it by $\xi_k$ and take the trace to obtain
	\begin{equation}
		\label{eq:c is a principal symbol 2}
		\xi_k\partder{\gamma^k_\psi}{x^j}=-\xi_k\partder{\rho^\psi_k}{\xi_j}=(l-1)\rho^\psi_k.
	\end{equation}
	Here we took again advantage of the homogeneity properties in \eqref{eq:vector fields have classical coefficients}.	The first and last sides of \eqref{eq:c is a principal symbol 2} are $x-$classical symbols, so we can expand them in $x-$homogeneous functions. Since asymptotic expansions are uniquely determined, we find that the two must be equal term by term, so the top order relation reads
	\begin{equation}
		\label{eq:c is a principal symbol 3}
		\xi_k\partder{\gamma^k_{\psi,1-l}}{x^j}=(l-1)\rho^{\psi,-l}_j.
	\end{equation}
	Multiplying by $x^j$ (namely, taking the trace of the matrix $(x^r\xi_k\del_{x^j}\gamma^k_{\psi,1-l})$), we obtain by homogeneity
	\begin{equation}
		\label{eq:c is a principal symbol 4}
		(l-1)x^j\rho_j^{\psi,-l}=(1-l)\xi_j\gamma^j_{\psi,1-l}.
	\end{equation}
	On the other hand, considering $p=\xi_r$ as a symbol and applying $\beta$ gives
	\begin{equation}
		\label{eq:c is a principal symbol 5}
		\left.\begin{aligned}
			(\beta p)_\psi&=\rho^\psi_r\in\sg{(1-l),-l}_{cl(x)}\\
			(\beta p)_e&=\rho^e_r\in\sg{1-l,(-l)}_{cl(\xi)}
		\end{aligned}\right\}\implies \rho^{\psi,-l}_r=\rho^{e,1-l}_r,
	\end{equation}
	where we computed the principal symbol of $\beta p$. Since \eqref{eq:c is a principal symbol 5} holds true for any $r$, it follows that we can substitute it in \eqref{eq:c is a principal symbol 4} to obtain \eqref{eq:c is a principal symbol}, as required. 
	
	Let then $C\in\lg{-l\indi}$ be an operator whose principal symbol is $\im c$ and let $(I-C)^{\sharp}$ be a parametrix of $(I-C)$, that is
	\[
	(I-C)(I-C)^{\sharp}=(I-C)^{\sharp}(I-C)=I+R
	\]
	for some $R$ smoothing. Recall here that $l>1$ so $C$ has negative integer order and thus $I-C\in\lg{0}$ with $\sigma^0_{pr}(I-C)=1$. Therefore, the parametrix of $I-C$ exists and has a compact remainder since $R$ has kernel in the Schwartz class. Moreover, $(I-C)^\sharp=I-C'$ for some $C'\in\lg{-\indi}$. Let then $P\in\lg{m}$. By commuting $P$ and $I-C$ and noticing that $[I,P]=0$, we obtain 
	\begin{equation}
		\label{eq:conjugation with I-C}
		\begin{aligned}
			(I-C)^{\sharp} P (I-C)&=(I-C)^\sharp\left((I-C)P+[P,I-C]\right)\\
			&=(I-C)^\sharp (I-C)P+(I-C)^\sharp [C,P]\\
			&=P+[C,P]\mod\lg{m-(l+1)\indi}.
		\end{aligned}
	\end{equation}
	Thence, the principal symbol of $(I-C)^{\sharp}P(I-C)-P$ equals the principal symbol of the commutator, namely, the Poisson bracket of the principal symbols. Specifically, we have
	\begin{equation}
	    \begin{aligned}
            \sigma_{pr}\left((I-C)^{\sharp}P(I-C)-P\right)&=-\im H_{\im c}(p)=H_c(p)\\
            &=\{c,p\}=\beta(p)\\
			&=\sigma_{pr}(\jmath(P)-P).
        \end{aligned}
	\end{equation}
	It follows that $\jmath(P)-P=(I-C)^{\sharp}P(I-C)-P\mod\lg{m-(l+1)\indi}$, so in particular there exists $Q\in\lg{m-(l+1)\indi}$ such that $\jmath(P)=(I-C)^{\sharp}P(I-C)+Q$. Conjugating with the parametrix $(I-C)^{\sharp}$ gives now
	\begin{equation}
		\label{eq:conjugating with (I-C)^{-1}}
		\begin{aligned}
			(I-C)\jmath(P)(I-C)^{\sharp}&=(I-C)(I-C)^{\sharp}P(I-C)(I-C)^{\sharp}+(I-C)Q(I-C)^{\sharp}\\
			&=P+RP+PR+RPR+(I-C)Q(I-C)^{\sharp}.\\
			\end{aligned}
	\end{equation}
	Now, on the one hand $PR$ and $RP$ are smoothing, since $R\in\RG$, on the other hand we have, thanks to Theorem \ref{theo:SG Egorov},
	\begin{equation}
	    (I-C)Q(I-C)^\sharp\in\lg{m-(l+1)\indi}.
	\end{equation}
	Thence, it holds true that $(I-C)^{\sharp}\jmath(P)(I-C)-P\in\lg{m-(l+1)\indi}$. The proof is complete.
\end{proof}
\begin{proof}[Proof of Theorem \ref{theo:automorphism is given by elliptic conjugation}]
	Exploiting Lemmas \ref{lemma:automorphism at principal symbol level} and \ref{lemma:automorphism for l>1} we can set up an inductive procedure which constructs a sequence of pseudo-differential operators $B_0,C_1,C_2,\dots$, where:
	\begin{enumerate}
	    \item $B_0$ is elliptic of some order $s\in\C$ and $C_j\in\lg{-j\indi}$; 
	    \item conjugation with $B_l=(I-C_l)\dots(I-C_1)B_0$ gives an automorphism of $\setquotient{\lg{}}{\RG}$, approximating $\jmath$ up to order $m-(l+2)\indi$.
	\end{enumerate}
	A computation of the asymptotic expansion of the symbol of $B_l$ shows that $(I-C_{l+1})B_l$ only changes the symbol up to $s-(l+1)\indi$. There exists, then, an elliptic operator $B\in\lg{s}$, such that, for each $l$, we have $B-B_l\in\lg{s-(l+1)\indi}.$ Thus, the difference $B\jmath(P)B^{-1}-P$ is smoothing, in view of said asymptotic expansion. This proves the claim.
\end{proof}

The preceding results proven thus far enable us to prove the main result of this section, namely the characterisation of the SGOPS of the formal symbol algebra $\BG{}$.

\begin{theo}[OPIs of the formal $SG$-algebra]
	\label{theo:automorphisms of the formal algebra}	
	Let $\imath\colon\BG{}\rightarrow\BG{}$ be an $SG$OPI on the formal symbol algebra $\BG{}$. There exists then an elliptic SGFIO $A$, of type $\Qcal_{gen}$, such that $\imath(P+\RG)=A^{\sharp}PA+\RG$ for any $P\in\lg{m}$.
\end{theo}
\begin{proof}
	On the one hand, we know that there exists an elliptic SGFIO $F$ of type $\Qcal_{gen}$ such that $F\imath(P)F^\sharp-P$ is an automorphism of $\BG{}$ preserving principal symbols. On the other hand, Theorem \ref{theo:automorphism is given by elliptic conjugation} guarantees that every such automorphism is given by conjugation with an elliptic $B\in\lg{s}$ for some $s\in\C^2$. Therefore, we see that, mod $\RG$, $F\imath(P)F^\sharp=BPB^\sharp$, so that setting $A\equiv B^\sharp F$ gives that $\imath(P)=A^\sharp PA$. This concludes the proof.
\end{proof}

\subsection{Lifting the characterisation to $\lg{}$}
We now turn to the problem of lifting this characterisation to the whole algebra. We notice first that we are able to take advantage of the Eidelheit Lemma of \cite{duistermaat1976orderpreservingisomorphisms} without any further hassle.

\begin{lemma}[Eidelheit-type Lemma]
    \label{lemma:eidelheit}
	Given an algebra isomorphism $\phi\colon\RG \rightarrow\RG$ there exists a topological isomorphism $V\colon\Scal\rightarrow\Scal$ such that $\phi(P)=VPV^{-1}$ for any $P\in\RG$.
\end{lemma}
\begin{proof}
	We show that the assumptions of Lemma 3 in \cite{duistermaat1976orderpreservingisomorphisms} hold true for $\mathcal{E}=\widetilde{\mathcal{E}}=\Scal, \mathcal U=\tilde{\mathcal{U}}=\RG$. Indeed, $\Scal$ is an infinite dimensional Fréchet space and $\RG$ comprises linear bounded operators on $\Scal$. Moreover, for $u,v\in\Scal$, the rank 1 operator $u\tens v$ lies in $\RG$. Then, picking a sequence $v_j$ converging to $v$ in the weak topology, we see that $u\tens v_j$ converges to $u\tens v$ in the operator topology, so that $\Scal'=\mathcal F=\widetilde{\mathcal F}$ in the notation of \cite{duistermaat1976orderpreservingisomorphisms}. The claim follows then directly from the quoted result.
\end{proof}

On the other hand, Lemma 4 of \cite{duistermaat1976orderpreservingisomorphisms} is not as straightforward to generalise to $SG$-operators. Indeed, when looking at the proof there, one is confronted with the possible incapability of choosing a function $u\in\Scal(\R^n)$ with the property that $u(x)\neq u(y)$ for each $x,y\in\R^n$. Accordingly, it is not clear whether such a claim is at all true. We set out then to prove directly that the composition of the Eidelheit isomorphism with the FIO coming from the formal algebra is a multiple of the identity up to some operator with Schwartz kernel. Consider to this end the composition $E=VA$ of the Eidelheit isomorphism $V$ with the FIO $A$ coming from Theorem \ref{theo:automorphisms of the formal algebra}.

\begin{lemma}
	\label{lemma:sobolev space extension}
	$E\colon \Scal\rightarrow\Scal$ is bounded and extends to a bounded operator $E\colon \HG{k\indi}\rightarrow\leb{2}$ for some $k\in\N$.
\end{lemma}
\begin{proof}
	$E$ is clearly bounded as an operator $\Scal\rightarrow\leb{2}$. Then, there is a finite set of semi-norms $\{p_0,\dots,p_n\}$ on $\Scal$ which estimate $\lnorm{2}{Eu}$, namely, $\lnorm{2}{Eu}\leq\max_{i\in \{0,\dots,n\}}p_i(u)$. Thus, there is an integer $k$ so that $\lnorm{2}{E}\leq\hnorm{k}{u}$. This shows that $E$ extends as a bounded operator $\HG{k\indi}\rightarrow\leb{2}$. The proof is complete.
\end{proof}

We choose now order reductions $P,Q$, as in the proof of Lemma \ref{lemma:canonicaltransformation} and consider $K\equiv ER^{-k}$ as an operator $\leb{2}\rightarrow\leb{2}$, where we denote $R=PQ$. Our goal is to prove that $K$ is an $SG$-pseudo-differential operator of order $(0,0)$. For this, we look at commutators and use the characterisation of Schrohe \cite{schrohe1987weightedsobolevspacesmanifolds} of $SG$-pseudo-differential operators on the weighted Sobolev spaces $\HG{(l,k)}$. Here and later we write $\ad K$ for the operator on $\lg{}$ acting by commutation with $K$, namely $(\ad K)(P)=[K,P]$. We owe the idea of the following strategy to Ryszard Nest, whom we thank for the helpful suggestion. We start with the following easy lemma. 

\begin{lemma}
    \label{lemma:commutator on weighted sobolev}
	$\ad{K}$ preserves the double filtration and $K$ extends to a bounded operator $K\colon\HG{r\indi}\rightarrow\HG{r\indi}$ for every $r\in\R^2$.
\end{lemma}
\begin{proof}
	Fix $r\in\R^2$ and consider $v\in\HG{r\indi}$. Setting $v_0\equiv \Lambda^{r\indi}v$ for $\Lambda$ an elliptic $SG\Psi$DO of order $\indi$, we have that $v_0\in\leb{2}$ and 
	\[
	Kv=K\Lambda^{-r\indi}v_0=K\Lambda^{-r\indi}K^{-1} Kv_0.
	\]
	Remarking that $\ad{E}$ preserves the double filtration of $\lg{}$, we see that
	\[
	KPK^{-1}=\Lambda^{-l\indi}E\Lambda^{-k\indi}P\Lambda^{k\indi}E^{-1}\Lambda^{l\indi},
	\]
	so that also $\ad{K}$ preserves the double filtration. It follows directly that $(\ad {K})(\Lambda^{-r\indi})$ has order $-r\indi$, and since $Kv_0\in\leb{2}$ by assumption, we have $Kv\in\HG{r\indi}$.
\end{proof}

\begin{prop}
    \label{prop:K is a pseudodiff}
	$K$ is a (non necessarily classical) $SG$-pseudo-differential operator of order $(0,0)$.
\end{prop}
\begin{proof}
	We prove that for every $\alpha,\beta\in\N^n$ there exists an operator $R^\alpha_\beta\in\lg{-\abs\alpha,-\abs\beta}$ such that, continuously,
	\begin{equation}
		\label{eq:commutator characterization sg operators}
		(\ad M_x)^\alpha (\ad \del)^\beta K=R^\alpha_\beta K\colon\HG{r}\rightarrow\HG{r+(\abs\alpha,\abs\beta)},
	\end{equation}
	where $M_{x^j}$ is the multiplication operator by $x^j$. This is known by \cite{schrohe1988psistaralgebra} to be equivalent to $K\in\lg{0}$. We show first that, for every $\beta\in\N^n$, there exists $Q_\beta\in\lg{0,-\abs\beta}$ such that 
	\begin{equation}
		\label{eq:inductive step 1}
		(\ad \del)^\beta K=Q_\beta K.
	\end{equation} We argue by induction on $\abs\beta$. 
	
	For $\abs\beta=1$ we have 
	\[\begin{aligned}
		(\ad \del_{x^j}) K&=\left[\del_{x^j},K\right]=(\del_{x^j}-K\del_{x^j}K^{-1})K\\
		&=Q_j K.
	\end{aligned}\]
	Here, $Q_j\in\lg{0,-1}$, since $\ad K$ is an automorphism preserving the principal symbol. So the base step holds true.
	
	Assume the claim holds true for every $\abs\beta\leq r$ and consider then $\gamma\in\N^n$ with $\abs\gamma=r+1$. Then $\gamma=\beta+\indi_j$ for some $j\in\{1,\dots n\}$ and some $\beta$ with $\abs\beta=n$. We write
	\[\begin{aligned}
		(\ad \del)^\gamma K&=(\ad\del_{x^j})\left[(\ad\del)^\beta K\right]\\
		&=(\ad\del_{x^j})(Q_\beta K)=(\ad \del_{x^j})(Q_\beta)K+Q_\beta(\ad \del_{x^j})K\\
		&=\left[\del_{x^j},Q_\beta\right]K+Q_\beta Q_j K,
	\end{aligned}\]
	having used the inductive hypothesis twice and the properties of $\ad $. Now, $Q_\beta Q_j=\widetilde{Q}_{\beta j}\in\lg{0,-\abs\beta-1}$ by composition, on the other hand $\del_{x^j}Q_\beta,Q_\beta\del_{x^j}\in\lg{1,-\abs\beta}$. However, they have the same principal symbol, so that in view of 5. in Proposition \ref{prop:principalsgsymbols} we have $[\del_{x^j},Q_\beta]=\check{Q}_{\beta j}\in\lg{0,-\abs\beta-1}$. It follows now that 
	\begin{equation}
		\label{eq:inductive step 2}
		(\ad \del)^\gamma K=(\check{Q}_{\beta j}+\widetilde{Q}_{\beta j})K\equiv Q_\gamma K,
	\end{equation}
	with $Q_\gamma\in\lg{0,-\abs\beta-1}=\lg{0,-\abs\gamma}$. By induction, then, \eqref{eq:inductive step 1} holds true for any $\beta\in\N^n$, as claimed.
	
	We now prove \eqref{eq:commutator characterization sg operators} for any $\alpha,\beta\in\N^n$. For $\abs\alpha=0$ there is nothing to prove. For clarity's sake, we spell out the case $\alpha=\indi_j$. Then,
	\[
		\ad x^j (\ad \del)^\beta K=\ad x^j(Q_\beta K)=\left[x^j,Q_\beta\right]K+Q_\beta \left[x^j,K\right].
	\]
	Similarly as above $[x^j,K]=(x^j-Kx^j K^{-1})K=\widetilde{P}_jK$ with $\widetilde{P}_j\in\lg{-1,0},$ so that $Q_\beta \widetilde{P}_j=\widetilde{R}^j_\beta\in\lg{-1,-\abs\beta}$ by composition. On the other hand, $[x^j,Q_\beta]=\bar{R}^j_\beta\in\lg{-1,-\abs\beta}$ in view of the observation above about the order of commutators in the $SG$-calculus.
	
	Assume now that \eqref{eq:inductive step 2} holds true for any $\alpha\in\N^n$ such that $\abs\alpha\leq n$, and let $\gamma=\alpha+\indi_j$ for some $j$. Then, using the properties of $\ad$ as a derivation,
	\begin{equation}
		\label{eq:inductive step 3}
		\begin{aligned}
			(\ad x)^\gamma (\ad\del)^\beta (K)&=\ad x^j (\ad x)^\alpha (\ad \del)^\beta (K)\\
			&=\ad x^j (P^\alpha_\beta K)=\left[x^j,P^\alpha_\beta\right]K+P^\alpha_\beta\left[x^j,K\right]\\
			&=\left(\left[x^j,P^\alpha_\beta\right]+P^\alpha_\beta P^j\right)K,
		\end{aligned}		
	\end{equation}
	with $P^\alpha_\beta$ given by inductive hypothesis and $P^j$ given by the previous step with $\abs\beta=0$. Now, $\widetilde{P}^{\alpha j}_\beta\equiv\left[x^j,P^\alpha_\beta\right]\in\lg{-\abs\alpha-1,-\abs\beta}$ by assumption and the properties of $[\,,]$. On the other hand, by composition, $\bar{P}^{\alpha j}_\beta\equiv P^\alpha_\beta P^j\in\lg{-\abs\alpha-1,-\abs\beta}$ as well. Thus, setting
	\begin{equation}
		\label{eq:inductive step 4}
		P^{\alpha+\indi_j}_\beta\equiv \widetilde{P}^{\alpha j}_\beta+\bar{P}^{\alpha j}_\beta
	\end{equation}
	it follows $P^\gamma_\beta\in\lg{-\abs\gamma,-\abs\beta}$. The induction is complete.
	
	Armed with this relation, it is now easy to show that the required mapping properties hold true. Indeed, $K$ maps $\HG{r}\rightarrow\HG{r}$ continuously for each $r\in\R^2$ by Lemma \ref{lemma:commutator on weighted sobolev}. On the other hand, $P^\alpha_\beta\in\lg{-\abs\alpha,-\abs\beta}$ gives exactly that $P^\alpha_\beta\colon\HG{r}\rightarrow\HG{r+(\abs\alpha,\abs\beta)}$ continuously for each $r\in\R^2$. The composition $P^\alpha_\beta K$ satisfies then the same properties. Thus, we have proven the characterization \eqref{eq:commutator characterization sg operators} and $K$ is a pseudo-differential operator of $SG$-type of order $0,0$.	
\end{proof}

Notice, in addition, that, while Lemmas \ref{lemma:sobolev space extension} and \ref{lemma:commutator on weighted sobolev}, together with Proposition \ref{prop:K is a pseudodiff}, imply that $E$ is a pseudo-differential operator as well, its order is not necessarily integral. Therefore, $\ad E$ is not an inner automorphism, cf. also \cite{mathai2017geometrypseudodifferentialalgebra}. This is the reason why, in general, we cannot expect to obtain an FIO of integer order. 

If we start with the inverse of the Eidelheit isomorphism, $V^{-1}$, we obtain, by the same argument, another pseudo-differential operator, $\tilde{E}$. They satisfy $EP-PE\in\RG, \tilde{E}P-P\tilde{E}\in\RG$, for any $P\in\lg{m}$. We notice that this means, in particular, that $E$ almost commutes with Shubin operators since $\Gamma^m(\R^n)\subset\sg{m,m}$. Notice that here we are disregarding classicality, notion which has different meanings for $\Gamma$ and $\sg{}$. However, Lemma \ref{lemma:e=c+r} below, suggested in a private communication by Elmar Schrohe, is a statement about smooth, bounded functions on $\R^{2n}$ which does not need classicality in any sense. Therefore, it can be proven almost exactly as in the Master thesis of Robert Hesse \cite{hesse2021orderpreserving}. We reproduce here the proof since the aforementioned work may not be readily available.

\begin{lemma}
	\label{lemma:e=c+r}
	Let $E,\tilde{E}\colon\Scal\rightarrow\Scal$ be $\Psi$DOs of $SG$-type, parametrices of each other, such that $[E,P],[\tilde{E},P]\in\RG$ for each $P\in\lg{}$. Then, $E=cI+R$ for some $c\in\C,R\in\RG$.
\end{lemma}
\begin{proof}
	First, notice that the conditions on $E$ and $\tilde{E}$ imply that their symbols $e,\tilde{e}$ are of order at most 0 and their derivatives are rapidly decreasing. Indeed, $\{e,p\}\in\Scal\, \forall p\in\sg{}\iff \del_xe,\del_\xi e\in\Scal$, and it follows that $e$ is bounded. Moreover, $\nabla e$ is a conservative vector field with potential $e$. Namely, for each path $\gamma\colon[a,b]\rightarrow\R^{2n}$ we have
	\begin{equation}
		\label{eq:conservative vector field}
		e(\gamma(a))-e(\gamma(b))=\int_{a}^{b}\nabla(e)(\gamma(s))\cdot\dot\gamma(s)\de s.
	\end{equation}
	Fix now some point $z$ in $\Sf{2n-1}$ (an oriented direction in $\R^{2n}$) so that, for $1<t_1<t_2$, it holds true
	\begin{equation}
		\label{eq:integral between t1 and t2}
		e(t_2z)-e(t_1z)=\int_{t_1}^{t_2}\nabla(e)(sz)\cdot z\de s.
	\end{equation}
	By assumption, $\nabla e$ has rapidly decaying component. It follows that, for any $v\in\R^{2n}$, we have $\abs{\nabla e(v)}\lesssim_N\braket{v}^{-N}$ and, for each $M\geq 2$, we can estimate the integral as
	\begin{equation}
		\label{eq:estimate of the integral}
		\abs{\int_{t_1}^{t_2}\nabla e(sz)\cdot z\de s}\lesssim_M\braket{s}^{2-2M}|_{t_1}^{t_2}.
	\end{equation}
	So, the integral converges to 0 uniformly in $t_2$ as we take the limit $t_1\rightarrow \infty$. We can, on the other hand, also estimate $\abs{e(t_2z)-e(t_1z)}\lesssim_k\braket{t_1}^k$. In particular then, we can pass to the limit $t_2\rightarrow\infty$ in this expression to obtain, for the radial limit $l(z)\equiv\lim_{t\rightarrow\infty}e(tz)$, the bound
	\begin{equation}
		\label{eq:estimate of the radial limit}
		\abs{l(z)-e(tz)}\lesssim_k\braket{t}^{-k}.
	\end{equation}
	We claim now that $l(z)$ does not depend on $z$, tht is, the radial limit is constant on the sphere $\Sf{2n-1}$. To see this, choose another $w\neq z$ on $\Sf{2n-1}$ and assume, possibly after having applied an orthogonal transformation, that $z=(1,0,\dots,0)$ and $w=(\cos\alpha,\sin\alpha,0,\dots,0)$. We consider a family of paths $\gamma_t\colon[0,\alpha]\rightarrow\R^{2n}$ given by $\gamma_t(\theta)=(t\cos\theta,t\sin\theta,0,\dots,0)$. Pick an $\eps>0$. Using \eqref{eq:estimate of the radial limit} for both directions $z,w$ we know that there exists a $T>1$ such that for all $t\geq T$ we have $\abs{l(z)-e(tz)}<\eps$ and $\abs{l(w)-e(tw)}<\eps$. On the other hand, for each fixed $t\geq T$, we have
	\begin{equation}
		\label{eq:difference of directions}
		e(tz)-e(tw)=\int_{\gamma_t}\nabla e(z)\cdot\de z=\int_0^\alpha\nabla e(\gamma_t(\theta))\cdot\dot\gamma_t(\theta)\de\theta, 
	\end{equation}
	so that taking absolute values and noticing that $\abs{\gamma_t}=\abs{\dot\gamma_t}=t$, we estimate, for each $M\geq 2$,
	\begin{equation}
		\label{eq:estimate of difference of directions}
		\abs{e(tz)-e(tw)}\leq\int_0^\alpha\abs{\nabla e(\gamma_t(\theta))\cdot\dot\gamma_t(\theta)}\de\theta\lesssim_M\frac{\alpha t}{\braket{t}^{M}},
	\end{equation}
	where we again used the fact that $\nabla e$ has rapidly decaying components. Clearly, for each $M$, the right-hand side decays to zero. We collect then
	\begin{equation}
		\label{eq:independence of direction}
		\abs{l(w)-l(z)}\leq\abs{l(w)-e(tw)}+\abs{e(tz)-l(z)}+\abs{e(tw)-e(tz)}\leq 3\eps,
	\end{equation}
	which shows that $l(z)=l(w)$ as claimed. We let $c$ be the constant value of $l$ on $\Sf{2n-1}$ and look at the function $f=e-c$. Clearly, $f$ has rapidly decreasing derivatives. On the other hand, we can compute, for $z\in\Sf{2n-1}$, that
	\begin{equation}
		\label{eq:estimate of f}
		f(tz)=e(tz)-l(z)=-\int_t^{+\infty}\nabla e(sz)\cdot z\de s
	\end{equation}
	and conclude, as before, that $f(tz)$ is rapidly decreasing as a function of $t$. As above, the convergence to zero is also uniform in $\Sf{2n-1}$ and we discover that $f$ is rapidly decreasing itself. Summing up, we have proven that $f\in\Scal(\R^{2n})$.
	
	Now, set $R=\Op(f)\in\RG$. Then, $e=c+f$ implies $E=cI+R$, and, since $\tilde{E}$ is a parametrix for $E$, we have, for some $R'\in\RG$,
	\begin{equation}
		\label{eq:final step}
		(cI+R)\tilde{E}=I+R', \implies c\tilde{E}=I\mod\RG.
	\end{equation}
	Therefore, $c$ must be different from $0$, and we deduce that also $\tilde{E}=\frac{1}{c}I+S$ for some $S\in\RG$. This concludes the proof.
\end{proof}

With all the above pieces in place, we can now state and prove our third main result.

\begin{theo}[Characterisation of $SG$-order-preserving isomorphisms]
	\label{theo:SGOPI}
	Let $\imath\colon\lg{}\rightarrow\lg{}$ be an SGOPI. Then, there exists an invertible, classical SGFIO $A$ of type $\Qcal_{gen}$ such that, for all $P\in\lg{}$, we have
	\begin{equation}
		\label{eq:SGOPI}
		\imath(P)=A^{-1} PA.
	\end{equation}
\end{theo}
\begin{proof}
	Consider $E=FV^{-1}$ where $V$ is the Eidelheit isomorphism of Lemma \ref{lemma:eidelheit} and $F$ the SGFIO obtained from Theorem \ref{theo:automorphism is given by elliptic conjugation} by considering the induced isomorphism on the formal symbol algebra $\BG{}$. Then, by the above discussion, we have that $E\colon\Scal(\R^n)\rightarrow\Scal(\R^n)$ is continuous. Moreover, it is an elliptic $SG\Psi$DO of order $(0,0)$, with parametrix $\tilde{E}=VF^\sharp$. Furthermore, both $E$ and $\tilde{E}$ commute $\mod \RG$ with every $P\in\lg{}$. By Lemma \ref{lemma:e=c+r}, it follows that $E=cI+R$ with $R\in\RG$, so that $F=cV+RV$. But then $V=c^{-1}(F-RV)$ and $V^{-1}=c(F-RV)^\sharp+S$ with some $S\in\RG$, so that $V$ is an invertible SGFIO, as claimed.
\end{proof}

The following corollary is now completely straightforward.

\begin{cor}
	\label{theo:SGOPI2}
	Let $\imath\colon\lg{}\rightarrow \lg{}$ be an algebra isomorphism satisfying the condition
	\begin{equation}
		\label{eq:SGOPI2}
		\imath(\lg{m_1,m_2})\subset\lg{m_2,m_1}\quad \forall(m_1,m_2)\in \Z^2.
	\end{equation}
	Then, $\imath(P)=(\FT A)^{-1} P\FT A$, where $A$ is an invertible $\Qcal$-FIO and $\FT$ is the Fourier transform.
\end{cor}
\begin{proof}
	This is obtained by combining Theorem \ref{theo:SGOPI} with Proposition \ref{prop:sc-operators conjugate with FT}. Namely, consider the isomorphism $\jmath(P)=\FT\imath(P)\FT^{-1}$, which, by assumption, is now an SGOPI. It follows that $\jmath(P)=A^{-1}PA$ for some invertible $\Qcal$-FIO. Then $\imath(P)=\FT A^{-1}P A\FT$, as claimed.
\end{proof}

\bibliographystyle{alpha}
\bibliography{contini2023SGorderpreserving.bib}

\end{document}